\documentclass[authoryear,aos,preprint]{imsart}

\makeatletter
\renewcommand\tableofcontents{%
    \leftline{{\bf\contentsname}}
    \smallskip\noindent
    \@starttoc{toc}%
    }
\makeatother
\usepackage{amsmath,amssymb,graphicx,natbib,color,graphicx}
\bibpunct{(}{)}{;}{a}{,}{,}

\startlocaldefs
\let\hat\widehat
\let\tilde\widetilde
\newtheorem{theorem}{Theorem}
\newtheorem{lemma}[theorem]{Lemma}

\newtheorem{corollary}[theorem]{Corollary}

\newenvironment{proof}{{\bf Proof.}}{$\Box$}
\newenvironment{proofof}{\noindent{\bf Proof of}}{$\Box$}
\DeclareMathOperator*{\argmax}{argmax}

\newcommand\cA{{\cal A}}

\newcommand\cC{{\cal C}}

\newcommand\cE{{\cal E}}
\newcommand\cF{{\cal F}}

\newcommand\cP{{\cal P}}

\newcommand\cT{{\cal T}}
\newcommand\cU{{\cal U}}

\newcommand\mathand{\;\ {\rm and}\ \;}

\newcommand\union{\cup}
\newcommand\Union{\bigcup}
\newcommand\intersect{\cap}

\newcommand\norm[1]{\left\|#1\right\|}

\newcommand\R{\mathbb{R}}
\newcommand\Set[1]{\left\{#1\right\}}

\endlocaldefs
\newcommand{\HRule}{\rule{\linewidth}{0.5mm}}

\begin{document}

\begin{frontmatter}
\title{The Geometry of Nonparametric Filament Estimation}
\runtitle{Filaments}

% indicate corresponding author with \corref{}
% \author{\fnms{John} \snm{Smith}\corref{}\ead[label=e1]{smith@foo.com}\thanksref{t1}}
% \thankstext{t1}{Thanks to somebody} 
% \address{line 1\\ line 2\\ printead{e1}}
% \affiliation{Some University}

\author{\fnms{Christopher} \inits{R.} \snm{Genovese,}\ead[label=e1]{genovese@stat.cmu.edu}}
\thankstext{t1}{Thanks to Pierpaolo Brutti for pointing us to some useful references
and Tony Cai for helpful comments.
This research was partially supported by NSF grant DMS-08-060009.}
\address{Department of Statistics\\ Carnegie Mellon Univerity\\ \printead{e1}}
\author{\fnms{Marco} \snm{Perone-Pacifico,}\ead[label=e2]{marco.peronepacifico@uniroma1.it}}
\address{Department of Statistical Sciences\\ Sapienza Univerity of Rome\\ \printead{e2}}
\author{\fnms{Isabella} \snm{Verdinelli}\ead[label=e3]{isabella@stat.cmu.edu}}
\address{Department of Statistics\\ Carnegie Mellon Univerity\\
and Department of Statistical Sciences\\ Sapienza University of Rome\\ \printead{e3}}
\and
\author{\fnms{Larry} \snm{Wasserman}\corref{}\ead[label=e4]{larry@stat.cmu.edu}}
\address{Department of Statistics\\ Carnegie Mellon Univerity\\ \printead{e4}}
\affiliation{Department of Statistics\\ Carnegie Mellon University\\
 and \\
Department of Statistical Sciences\\Sapienza University of Rome\\ \today}
\runauthor{Genovese et al}

\begin{abstract}
We consider the problem of 
estimating filamentary structure from planar point process data.
We make some connections with computational geometry and
we develop nonparametric methods for estimating the filaments.
We show that, under weak conditions, the filaments
have a simple geometric representation
as the medial axis of the data distribution's support.
Our methods convert an estimator of the support's boundary
into an estimator of the filaments.
We also find the rates of convergence of our estimators.
\end{abstract}


\end{frontmatter}

\newpage

\tableofcontents

\newpage

\section{Introduction}

Filaments are one-dimensional curves embedded in 
$\mathbb{R}^d$ where $d>1$.
Filament estimation has important applications in many fields
including astronomy, geology, and medicine.
Our basic filament model is
\begin{equation}\label{eq::first}
Y_i = f(U_i) +  \epsilon_i
\end{equation}
where $f:[0,1] \to \mathbb{R}^d$.
The unobserved variables $U_1,\ldots, U_n$ are drawn from
a distribution $H$ on $[0,1]$ and
$\epsilon_1, \ldots, \epsilon_n$ are 
drawn from a mean zero noise distribution $F$.
The goal is to estimate
\begin{equation}
\Gamma\equiv \Gamma_f = \{ f(u): 0 \leq u \leq 1\}.
\end{equation}
Later, we extend the model to include background clutter,
other $Y_i$'s drawn uniformly from a compact set containing the filaments.
See Figure \ref{fig::model}.
Estimating $f$ is an example of one-dimensional manifold learning.
It may also be regarded as a type of principal curve
estimation.

There is a plethora of available statistical methods 
that can, in principle, be used for estimating filaments.
These include:
principal curves (\cite{hs::1989},
\cite{kegl::2000}, \cite{sandilya::2002}, and \cite{smola});
nonparametric, penalized, maximum likelihood
\citep{tibs::1992};
beamlets (\cite{Donoho01beamletsand}, and \cite{arias::2006});
parametric models (\cite{stoica-etal:2007});
manifold learning techniques
(\cite{isomap}, \cite{lle}, and \cite{Huo02locallinear});
gradient based methods
(\cite{novikov-etal:2006}, and
\cite{us::2009} and
methods from computational geometry 
(\cite{Dey}, \cite{Lee:1999}, and \cite{Cheng:2005}).

In this paper, 
we make some connections between the statistical problem
and some ideas from computational geometry.
We propose new, simple,
nonparametric estimators for $\Gamma_f$,
and we find their rates of convergence.
To the best of our knowledge, our methods are the first
that are computationally simple,
consistent, and have given rates of convergence
with the exception of
\cite{Cheng:2005}.
However, our methods are simpler than those in
\cite{Cheng:2005},
our assumptions are weaker,
our loss function is more stringent
and our estimators have faster rates of convergence.

The optimal rates of convergence
for this problem appear to be unknown.
In related work (\cite{us})
we derived the minimax rate under stringent conditions.
In ongoing work, we are finding the minimax rate
under more general conditions.
These rates depends critically on various features of the noise distribution $F$.
The methods in this paper are unlikely to be minimax optimal.
Nonethless, they achieve reasonable rates of convergence and are simple to compute.

Our basic strategy involves two steps:
\begin{enumerate}
\item Construct a set of fitted values that are close in Hausdorff distance to the filament.
\item Extract a curve from this set of fitted values.
\end{enumerate}

\begin{figure}
\vspace{-.5in}
\begin{center}
\includegraphics[width=4in,height=4in]{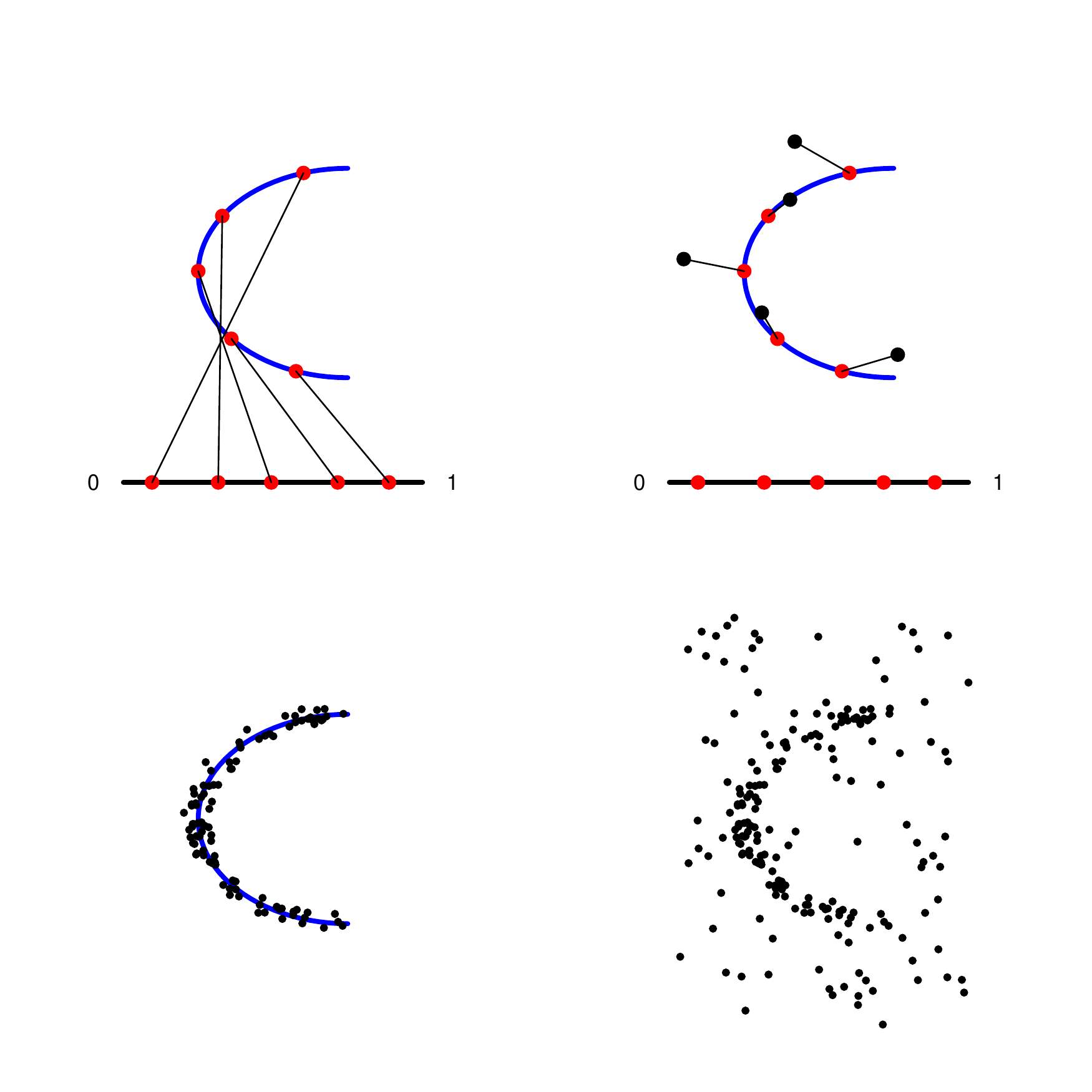}
\end{center}
\vspace{-.5in}
\caption{These plots illustrate the filament model.
Top left: some points $U_i$ on $[0,1]$ are mapped to $\Gamma_f$ by $f$.
Top right: noise is added to the points.
Bottom left: a larger sample.
Bottom right: background clutter has been added.}
\label{fig::model}
\end{figure}

{\em Motivation.}
The need to identify filamentary structures 
arises in a wide variety of applications.
In medical imaging, for instance,
filaments arise as networks of blood vessels in tissue
and need to be identified and mapped.
In remote sensing,
river systems and road networks are common filamentary structures
of critical importance (\cite{lacoste-etal:2005,stoica-etal:2004}).
In seismology, the concentration of earthquake epicenters
traces the filamentary network of fault lines.
Filaments are of particular interest in astronomy
because the distribution of galaxies in the universe is concentrated
on a network of filaments that is often called the ``cosmic web.''
Indeed, astronomers have substantial literature on the problem of
estimating filaments; see
\cite{luo-vishniac:1995},
\cite{weygaert},
\cite{martinez-saar:2002},
\cite{barrow::1985},
\cite{stoica-etal:2005},
\cite{eriksen-etal:2004},
\cite{novikov-etal:2006},
\cite{sousbie-etal:2006} and
\cite{stoica-etal:2007}.

\vspace{.5cm}

{\em Summary of Results.}
Two key geometric ideas underlie our results -- the medial axis of a set
and the thickness $\Delta(f)$ of a curve $f$ -- both of which are defined in Section 3.
The medial axis is like the median of a set.
The thickness of a curve measures both the curvature and 
how close the curve comes to being self-intersecting.

Our main results are the following:
\begin{enumerate}

\item If the noise level $\sigma$ of $F$ is less than the thickness
  $\Delta(f)$, the filament equals the medial axis of the support of
  $Y$'s distribution (Theorem \ref{thm::filamentismedial}).
\item Any  estimate of the boundary of the support of the distribution can be
  converted into an estimate of the filament that is close in
  Hausdorff distance to the true filament (Theorems
  \ref{thm::edtestimator} and \ref{thm::methodII}).  If the rate of
  convergence of the boundary estimator is $r_n$ then the rate of
  convergence of the filament estimator is also $r_n$.
\item Our estimators produce a set of fitted values that contain the
  filament and are close to it in Hausdorff distance.  In Section 5,
  we show how to extract curves from the set estimators that are Hausdorff close
  to the true filament.
\end{enumerate}
Proofs of all results are given in Section \ref{sec::proofs}.

\medskip
{\em Notation.}
The boundary of a set $S$ is denoted by $\partial S$.
The Hausdorff distance
between two sets $A$ and $B$ is
\begin{equation}
d_H(A_1,A_2) =\min \left\{\delta:\strut\  A_1 \subset A_2\oplus\delta \;\mathand  A_2 \subset A_1\oplus\delta\right\}
\end{equation}
where
\begin{equation}
A\oplus\delta = \Union_{x\in A} B(x,\delta)
\end{equation}
denotes the $\delta$-\emph{enlargement} of the set $A$,
and $B(x,\delta) = \{ y:\ ||y-x|| \le \delta\}$
denotes a closed ball centered at $x$ with radius $\delta$.
If $A$ is a set and $x$ is a point then we write
$d(x,A) = \inf_{y\in A}||x-y||$.
The closure of $A$ is denoted by $\overline{A}$ and the complement of $A$ by $A^c$.
A curve is a map $f:[0,1]\to \mathbb{R}^d$.
Throughout, we use symbols like $C, c_0, c_1\ldots $ to denote generic positive constants
whose value may be different in different expressions.

\section{The Model}
\label{sect::model}

We will focus on finding filaments in a two dimensional
point process although the ideas extend to higher dimensions.
We begin with a single filament.
Suppose we observe $Y_1,\ldots, Y_n$ where
\begin{equation}
Y_i = f(U_i) +  \epsilon_i,\ \ \ \ \ i=1,\ldots, n
\end{equation}
where
$f:[0,1]\to \mathbb{R}^2$,
$U_1,\ldots,U_n \sim H$
where $H$ is a distribution on $[0,1]$
and $\epsilon_1,\ldots,\epsilon_n$ are drawn from $F$.

Denote the graph of the filament $f$ by
\begin{equation}
\Gamma\equiv \Gamma_f = \left\{f(u):\strut\ u\in[0,1]\right\}.
\end{equation}
With some abuse of terminology, 
we refer to both $f$ and $\Gamma_f$ as the filament.
We assume that $\Gamma_f$ is contained in a compact set
which, without loss of generality, we take to be
$[-1,1]^2$.

The output of our algorithms will be a set
$\hat\Gamma$ which need not be a curve.
Our loss function is Hausdorff distance
\begin{equation}
d_H(\Gamma_f,\hat\Gamma) = \inf\Bigl\{ \delta:\ 
\hat\Gamma \subset \Gamma_f \oplus \delta\ {\rm and}\ \Gamma_f \subset \hat\Gamma \oplus \delta \Bigr\}.
\end{equation}
We will also show how to extract a curve from $\hat\Gamma$.

\smallskip
Next we define a smoothness condition for $f$.
For any three distinct points $x,y,z$ on $\Gamma_f$ let
$r(x,y,z)$ be the radius of the circle passing through
the three points.
Define the \emph{thickness} of the curve $\Gamma_f$,
\citep{gonzalez}
denoted $\Delta \equiv \Delta(f)$, by
\begin{equation}\label{eq::Delta}
\Delta \equiv \Delta(f) \equiv \Delta(\Gamma_f)= \min_{x,y,z} r(x,y,z)
\end{equation}
where the minimum is over all triples of distinct points on $\Gamma_f$.
$\Delta$ is also called the {\em minimum global radius of curvature}, 
and the {\em normal injectivity radius} of $f$
and the {\em condition number} (\cite{smale}).
The thickness $\Delta$ has the following interpretation:
it is the minimum radius of all circles
that are tangent to one point of $\Gamma_f$ while passing
through another point of $\Gamma_f$.
A ball of radius $r > \Delta$ tangent to a point $y$ on $\Gamma_f$ 
can contain points in $\Gamma_f$ other than $y$.
This can occur because the radius of curvature of $\Gamma_f$ is
smaller than $r$ or because
the curve comes within $r$ of self-intersecting.
See Figure \ref{fig::thickness}. Hence the thickness
combines information about curvature and
separation, capturing both local and global features of the curve.
A useful way to think of $\Delta$
is that it is the largest radius of a ball that can roll
freely around $\Gamma_f$.

\smallskip
If $f(0) \neq f(1)$ we say that $f$ is \emph{open}.
If $f(0) = f(1)$ we say that $f$ is \emph{closed}.
If, for $u,v\in(0,1)$, $u \neq v$ implies that $f(u)\neq f(v)$ then we say that
$f$ is \emph{simple}, or non-self-intersecting.
Otherwise, we say it is self-intersecting. Unless stated otherwise,
we assume that $f$ is smooth (non-zero, finite gradient at every point)
and simple.
We assume that the filament is parameterized
with respect to arclength, normalized to $[0,1]$.

\newpage
\begin{figure}[h]
\begin{center}
\hbox to\textwidth{\hss \includegraphics[width=6.5in]{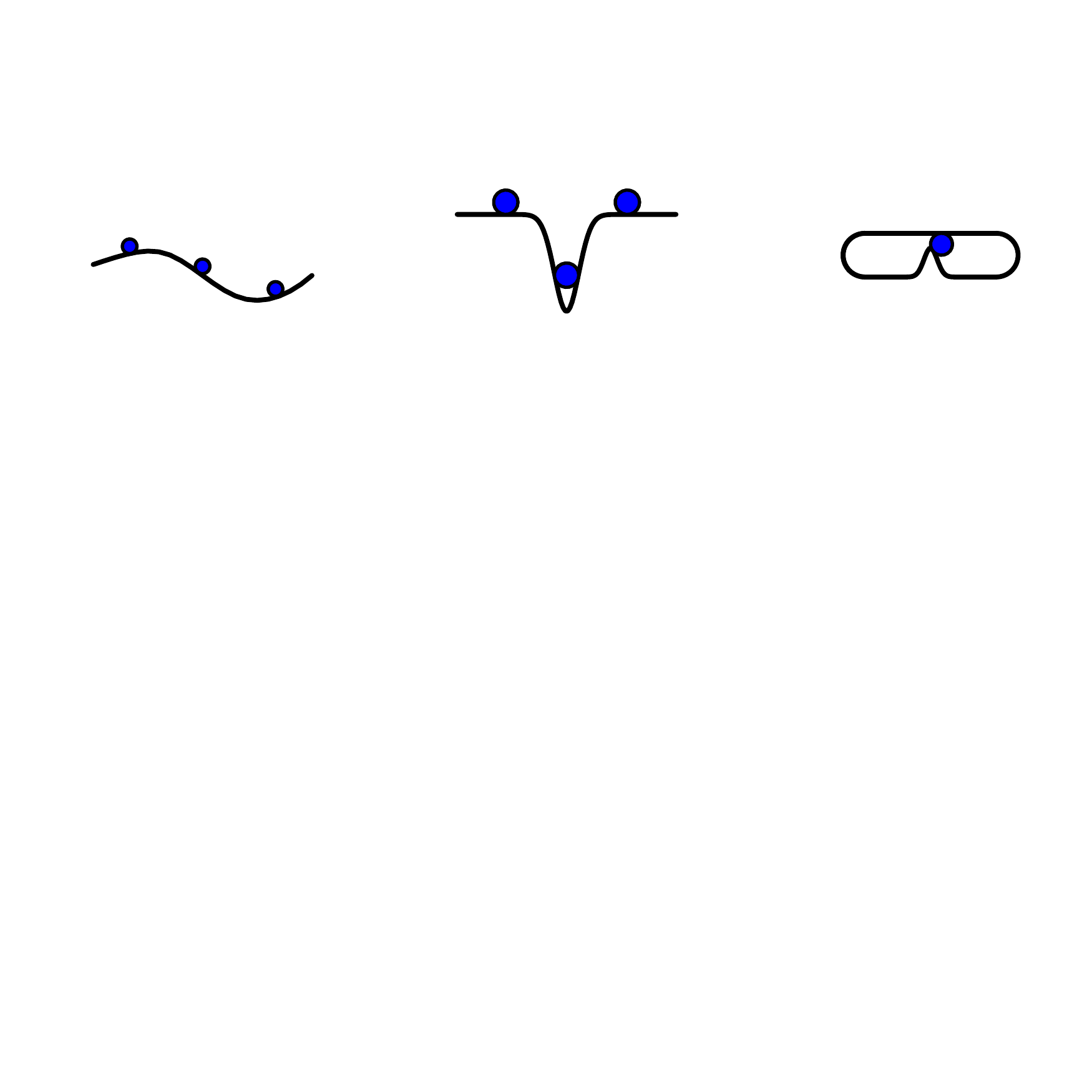}\hss}
\end{center}
\vspace{-4.53in}
\caption{A ball of radius $r \leq \Delta$ can roll freely (left).
A ball of radius $r > \Delta$ cannot roll freely
because either it hits a region of high curvature (center)
or it hits a region with a near self-intersection (right).}
\label{fig::thickness}
\end{figure}

We make the following assumptions:

\begin{description}
\item [(A1)]
$H$ has density $h$ with respect to Lebesgue measure on $[0,1]$
that is bounded and bounded away from zero:
\begin{equation}
0 < c_1 \leq \inf_{0\leq u \leq 1}h(u) \leq \sup_{0\leq u \leq 1} h(u)\leq c_2 < \infty
\end{equation}
for some $c_1,c_2$.

\smallskip
\item [(A2)]
The noise distribution $F$ satisfies these conditions:
\begin{enumerate}
\item $F$ has support $B(0,\sigma)$.
\item $F$ has bounded continuous density $\phi$ with respect to
Lebesgue measure on $\mathbb{R}^2$ and $\phi(y) >0$ for all $y$ in the interior of $B(0,\sigma)$.
\item $\phi$ is nonincreasing, that is,
$||u|| \leq ||v||$ implies that $\phi(u) \geq \phi(v)$.
\item $\phi$ is symmetric, i.e. $||x|| = ||y||$ implies that $\phi(x)=\phi(y)$.
\item There exists $0 \leq \beta \leq \infty$ and $C>0$ such that
$$
\phi(x)\sim C (\sigma-||x||)^\beta\ \ \ {\rm as}\ \ \  ||x||\to\sigma.
$$
\end{enumerate}

\item [(A3)] $f$ is sufficienty smooth, i.e., $\sigma < \Delta(f)$. If $f$ is
open, then also \\ $||f(1)-f(0)||/2 > \Delta(f)$.

\end{description}

The parameter $\beta$ controls the behavior of $\phi$ near the boundary of its support.
The marginal density of $Y_i$ is
$q(y) =  \int \phi(y-f(u)) dH(u)$.
Let
\begin{equation}
S = \left\{y:\strut\ q(y) > 0\right\}
\end{equation}
denote the support of $q$.
It follows from assumption (A2) that
\begin{equation}
S = \Union_{0\le u\le 1} B(f(u),\sigma).
\end{equation}
We will let $Q=Q_{f,h,\sigma}$ denote the distribution of the data
corresponding to density $q$.
The boundary behavior of $q$ is related to $\beta$.
Let 
\begin{equation}
\alpha = \beta + (1/2).
\end{equation}

\begin{lemma}
\label{lemma::new-boundary}
There exist constants $c_1,c_2>0$ 
such that the following is true.
Let $y = (y_1,y_2)$ be in the interior of $S$.
For small enough $d(y,\partial S)$ we have that
\begin{equation}
c_1 d(y,\partial S)^\alpha \leq q(y) \leq c_2 d(y,\partial S)^\alpha.
\end{equation}
\end{lemma}

\smallskip
We remark that if the noise density is uniform on
$B(0,\sigma)$, then
$\alpha = 1/2$ and so
$q$ is not uniform over its support.
In fact, $q(y)=0$ on $\partial S$.

\bigskip

Multiple filaments can be modeled by allowing $f$
to be piecewise continuous instead of continuous.
Multiple filaments can also be represented as follows.
Let $f_1,\ldots,f_k$ be a set of one dimensional curves in $\mathbb{R}^2$ where
$f_j:[0,1] \to \mathbb{R}^2$, $j=1,\ldots, k$.
Let $\Omega$ be a distribution on
$\{1,\ldots, k\}$ and let
$H_1,\ldots, H_k$ denote $k$ different distributions on $[0,1]$.
For $i=1,\ldots, n$ let
\begin{align*}
Z_i &\sim \Omega\\
U_i &\sim H_{Z_i}\\
Y_i &= f_{Z_i}(U_i) +  \epsilon_i.
\end{align*}

We can also extend the model to allow for clutter,
as in \cite{raftery::1998}.
Let $Q_0$ denote a uniform distribution
on a compact set $C\subset \mathbb{R}^2$ and define the mixture
$(1-\eta)Q_0 + \eta Q_{f,h,\sigma}$
where $0 \leq \eta \leq 1$.
We call points drawn from $Q_0$ {\em background clutter}.
Until Section \ref{sec::decluttering},
we will assume no clutter is present (i.e., $\eta = 1$).
Another generalization of the model
is to allow $f$ to be self-intersecting,
which we consider briefly later.

\section{Estimation}
\label{section::tractable}

It will be helpful
to first make some connections with some concepts
from computational geometry.

\subsection{Some Backgound on Geometry}

Let $S\subset\mathbb{R}^2$ be a compact set.
A ball $B\subset S$ is called \emph{medial} if
\begin{enumerate}
\item ${\rm interior}(B) \cap \partial S = \emptyset$ and
\item $B\cap \partial S$ contains at least 2 points.
\end{enumerate}
The {\em medial axis} $M\equiv M(S)$, shown in Figure \ref{fig::medial}, is
the closure of the set
\begin{equation}
\left\{x\in S\strut:\ B(x,r)\ \mbox{\rm is medial for some}\ r > 0\right\}.
\end{equation}
See \cite{Dey} and references therein for more information about
the properties of the medial axis.

\begin{figure}
\vspace{-.5in}
\begin{center}
\includegraphics[width=3.5in,height=3.5in]{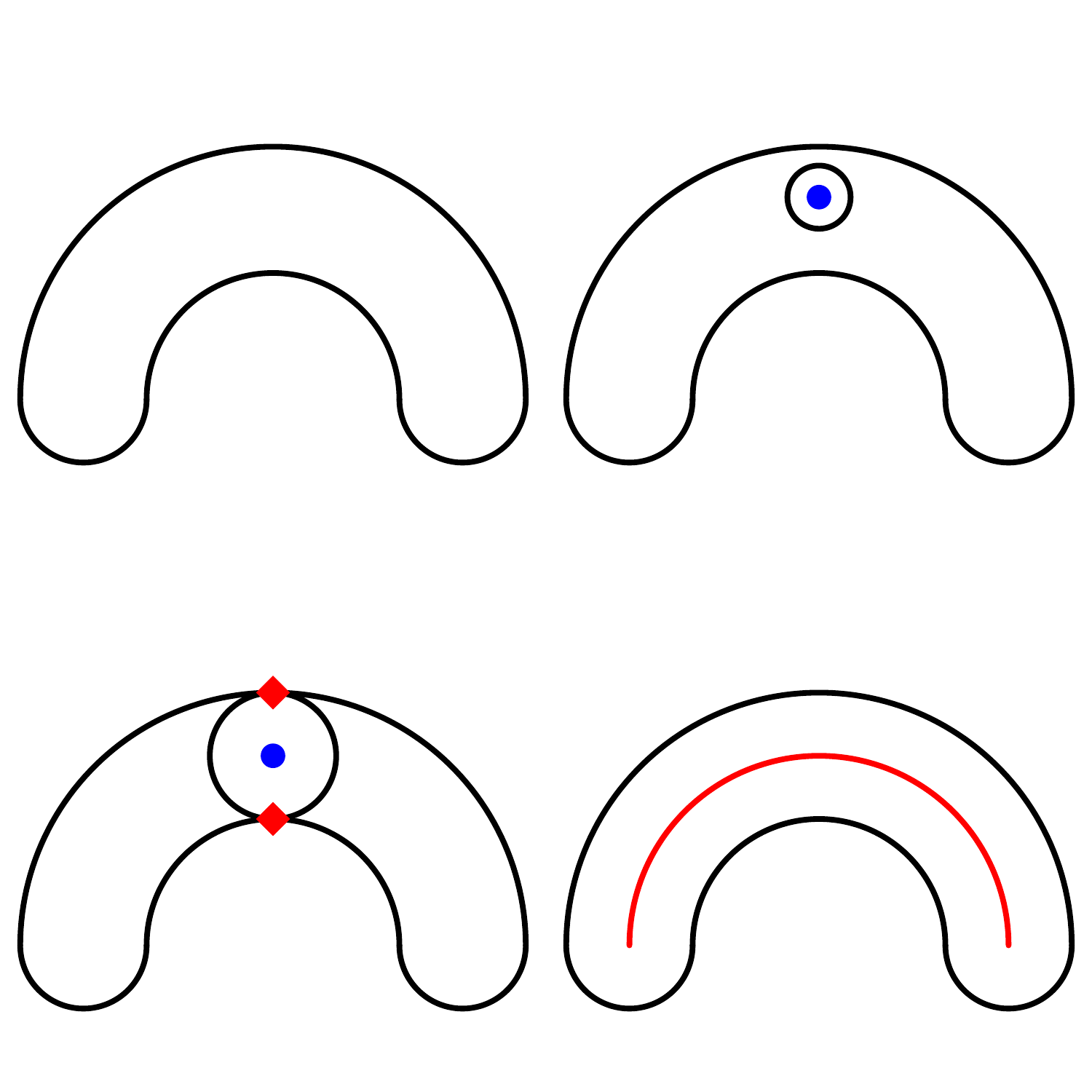}
\end{center}
\vspace{-.2in}
\caption{The Medial Axis.
Top left: a set $S$.
Top right: a non-medial ball contained in $S$;
Bottom left: a medial ball that touches the boundary of $S$ in 2 places.
Bottom right: the medial axis consists of the centers of the medial balls.}
\label{fig::medial}
\end{figure}

For each $u$ let $N(u)$ denote the normal vector at $f(u)$
and $T(u)$ the tangent vector at $f(u)$.
Define the {\em fiber},
\begin{equation}\label{eq::tube}
L(u) = \Bigl\{ f(u) + t N(u):\ -\sigma \leq t \leq \sigma \Bigr\}
\end{equation}
and the tube $ \cT = \Union_{0 \le u \le 1} \ L(u).$

For open curves define the initial and final \emph{end caps}, respectively,
by
\begin{equation}\label{eq::endcaps}
\cC_0 = B(f(0),\sigma) - \cT \qquad \mbox{and} \qquad \cC_1 = B(f(1),\sigma) - \cT.
\end{equation}
When $f$ is a closed curve, the end caps are empty, and when $f$ is open with
$||f(1)-f(0)|| > 2 \sigma$, $\cC_0 \cap \cC_1 =\emptyset.$

\smallskip
The next lemma gives a useful decomposition of the
support set $S$.

\medskip
\begin{lemma} \label{lemma::disjoint}

\begin{enumerate}
\item $S = \cT \union \cC_0 \union \cC_1$, and in particular, when $f$ is closed, $S = \cT$.
\item For every $u\ne v\in[0,1]$, $L(u)$ and $L(v)$ are disjoint.
\item For every $y\in \cT$, there exists a unique fiber containing $y$.
\item For every $y\in \cT$, the closest point on $\partial S$ to $y$ is either
$f(u) + \sigma N(u)$ or $f(u) - \sigma N(u)$.
\item When $f$ is closed $\partial S = \partial S_0 \cup \partial S_1$, when $f$ is open
$\partial \cT = \partial S_0 \cup \partial S_1$, where
$$
\partial S_0 = \{ f(u) + s(u)\sigma N(u):\ 0 < u < 1\} 
$$
and
$$
\partial S_1 = \{ f(u) + t(u) \sigma N(u):\ 0 < u < 1\}
$$
are two non intersecting connected curves where $s(u) \in \{-1,+1\}$ and $t(u) = -s(u)$.
\end{enumerate}
\end{lemma}

\smallskip
The following theorem relates the filament to its medial axis.

\begin{theorem} 
\label{thm::filamentismedial} 
\vspace{-.05in}
\begin{enumerate}
\item 
If $f$ is closed and $\sigma < \Delta(f)$ then
$\Gamma_f = M(S)$.
\item If $f$ is open and
$\sigma < \Delta(f)$ then
$\Gamma_f \subset M(S)$.
If, in addition,
$\sigma < ||f(1)-f(0)||/2$ then
$\Gamma_f = M(S)$.
\end{enumerate}
\end{theorem}

This result holds both good news and bad news.
The good news is that $\Gamma_f = M(S)$,
relating the filament to a well defined geometric quantity.
The bad news is that the medial axis is not continuous in Hausdorff distance.

\begin{figure}[h]
\vspace{-.8in}
\begin{center}
\includegraphics[width=3in]{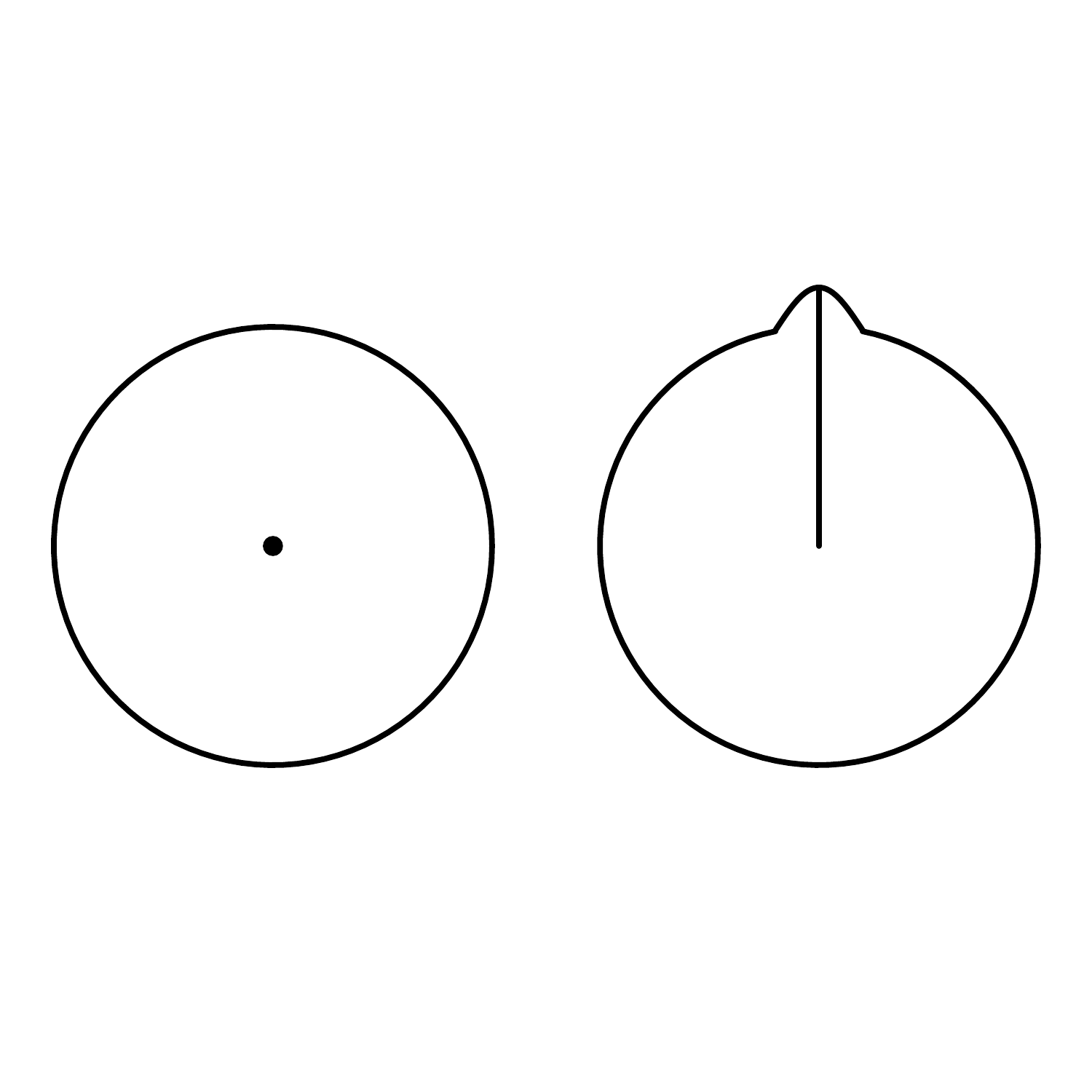}
\end{center}
\vspace{-1in}
\caption{A stylized example
showing that small perturbations in $S$ can lead to large changes in $M(S)$.
The medial axis of a circle (left) is the center.
If a small perturbation is added to the circle (right) then the medial axis changes completely.}
\label{fig::delicate}
\end{figure}

Small perturbations to $S$ give a completely different medial axis,
as illustrated in Figure \ref{fig::delicate}.
Thus, estimating the medial axis is non-trivial.
From now on, we assume that $\sigma < \Delta(f)$.

\bigskip
The {\em Euclidean distance 
transform} (EDT) (\cite{breu}) is a mapping from $\mathbb{R}^2\to[0,\infty)$ defined by
$\Lambda(y) = d(y,\partial S)$.
The next result gives another characterization of
the filament $\Gamma_f$: the filament maximizes $\Lambda(y)$.
In particular, $\Gamma_f = \{y\in S:\ \Lambda(y)=\sigma\}$.

\begin{lemma}
\label{lemma::edt}
\begin{enumerate}
\item $y\in M(S)$ if and only if
$\Lambda(y)=\sigma$.
\item For any $y\in S- M(S)$,
$\Lambda(y) < \sigma$.
\item For any $y\in S$,
$d(y,M(S))+ \Lambda(y) = \sigma$.
\end{enumerate}
\end{lemma}

Let $\hat S$ be an estimate of $S$
and $\hat{\partial S}$ be an estimate of $\partial S$. 
For $y\in\mathbb{R}^2$, define the empirical EDT by
$\hat\Lambda(y) = d(y,\hat{\partial S})$.
We estimate the noise level $\sigma$ by 
$\hat\sigma = \sup_{y\in\hat S}\hat\Lambda(y) \equiv \hat\Lambda(\hat y)$,
where
\begin{equation}\label{eq::yhat}
\hat{y} = {\rm argmax}_{y\in \hat S}\hat\Lambda(y).
\end{equation}

\begin{theorem}
\label{thm::edt}
Suppose that
$d_H(\partial S,\hat{\partial S}) \leq \epsilon$. Then:
\begin{enumerate}
\item $\sup_{y\in \mathbb{R}^2} |\hat\Lambda(y) - \Lambda(y)| \leq \epsilon$.
\item $|\hat\sigma- \sigma| \leq  \epsilon$.
\item $d(\hat y,M(S)) \leq 2\epsilon$.
\end{enumerate}
\end{theorem}

\smallskip
Following \cite{cuevas-boundary}, we say that a set $S$ is 
{\em $(\chi,\lambda)$-standard} if there exist positive numbers
$\chi$ and $\lambda$ such that
\begin{equation}
\nu(B(y,\epsilon)\cap S) \geq \chi \ \nu(B(y,\epsilon)) \ \ \ \ \ 
{\rm for \ all\ }y\in S,\  0< \epsilon \leq \lambda
\end{equation}
where $\nu$ is Lebesgue measure. We say that $S$ is {\em partly expandable} 
if there exist $r>0$ and $R\geq 1$ such that 
$d_H(\partial S, \partial (S\oplus\epsilon)) \leq R \epsilon$
for all $0\leq \epsilon < r$. (Recall that $S\oplus\epsilon$ is the enlargement 
of $S$). A standard set has no sharp peaks while a partly expandable 
set has not deep inlets.

\begin{lemma}
\label{lemma::standard}
$S$ is standard with $\chi = 1/4$ and $\lambda = \sigma$.
Also, $S$ is partly expandable with $R=1$ and $r=\Delta-\sigma$.
\end{lemma}

\subsection{Estimating Boundaries}

We estimate the support $S$ and its boundary $\partial S$.
The estimate of $\partial S$ will be
converted into an estimator of the filament.
The performance of these estimators, in Hausdorff-distance loss, translates directly
to the performance of the filament estimators.
We use $r_n$ to denote the rate of convergence
of the boundary estimator; that is, 
$d_H(\hat{\partial S},\partial S) =O_P(r_n)$.

\begin{figure}
\hbox to\textwidth{\hss \includegraphics[width=5in]{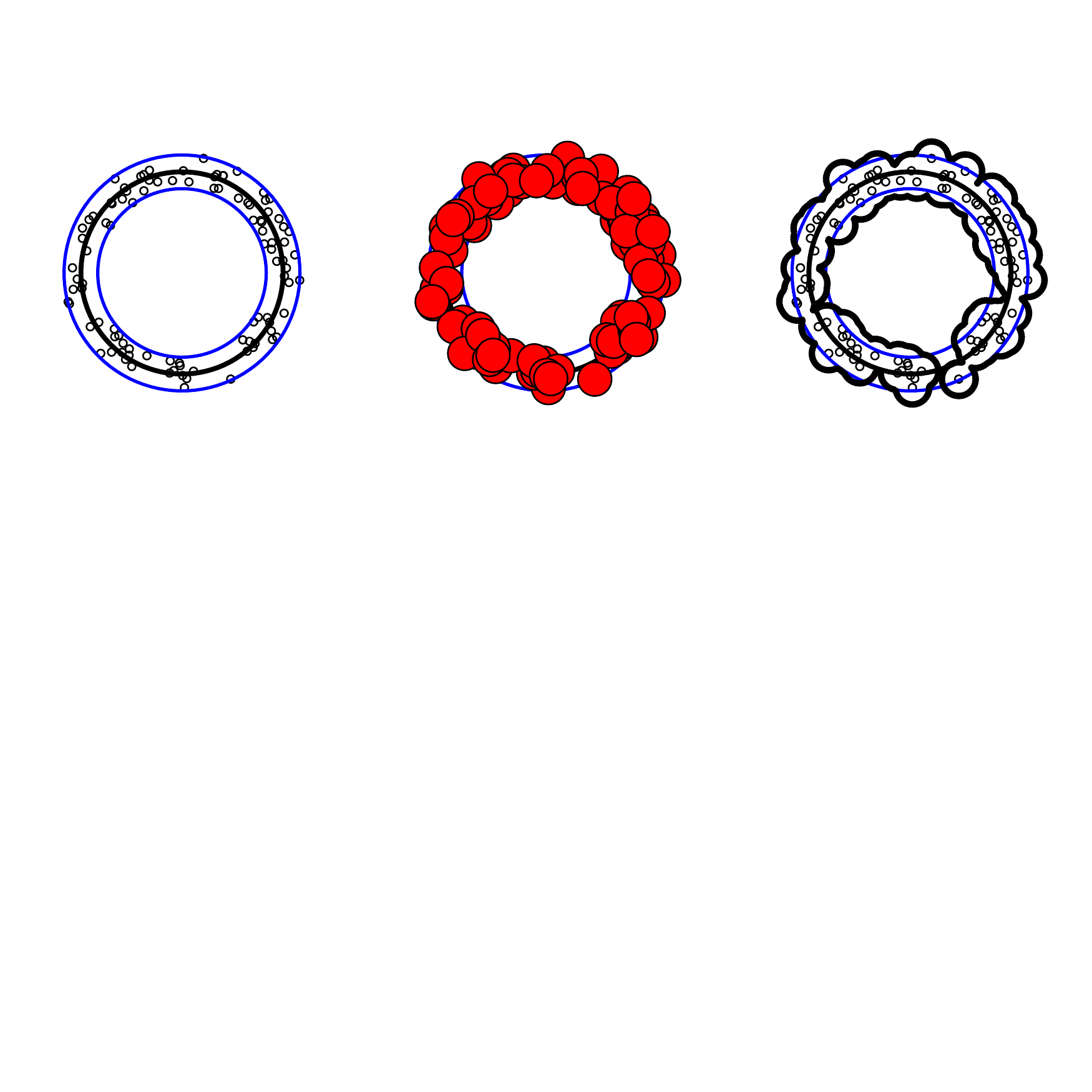}\hss}
\vspace{-3in}
\caption{These plots illustrate the estimators $\hat S$ and $\hat{\partial S}$.
Left: A closed filament, data and the true support. Center:The estimator of 
the support $\hat S$ is a union of balls. Right: The boundary estimator.}
\label{fig::unionofballs}
\end{figure}

In practice, we will use the estimator from \cite{cuevas-boundary}
and \cite{dw},
described 
in the following result. An example is shown in Figure \ref{fig::unionofballs}.
This estimator
is simple to use and fast to compute. 
Recall that $\alpha = \beta + (1/2)$
where $\beta$ is defined in condition (A2).

\begin{lemma} {\rm (\cite{cuevas-boundary}).}
\label{lemma::supp}
Let $Y_1,\ldots, Y_n$ be a random sample from a distribution with support $S$.
Let $S$ be compact, $(\lambda,\chi)$-standard and partly expandable.
Suppose the distribution $Q$ has positive density $q$
and that for all $y\in S$,
$q(y) \geq C d(y,\partial S)^\alpha$
for some $C>0$ and some $\alpha \geq 0$.
Let 
\begin{equation}
\hat{S} = \Union_{i=1}^n B(Y_i,\epsilon_n)
\end{equation}
and let $\hat{\partial S}$ be the boundary of $\hat S$.
If $C > \sqrt{2/(\chi \pi)}$ and $\epsilon_n = C (\log n /n)^{1/(2+\alpha)}$ then, 
with probability one
\begin{equation}
d_H(S,\hat{S}) \le r_n \qquad\mathand\qquad
d_H(\partial S,\hat{\partial S}) \le r_n
\end{equation}
for all large $n$,
where $r_n =C (\log n/n)^{1/(2+\alpha)}$.
Also, $S\subset \hat S$ almost surely for all large $n$.
\end{lemma}

{\bf Proof Outline.}
The proof is essentially the same as the proof in
\cite{cuevas-boundary}.
They implicitly assume that
$\inf_{y\in S} q(y) >0$.
In particular their proof (see page 348 of their paper)
argues that, for any $y\in S$,
$Q(B(y,\epsilon)) \geq c\epsilon^2$ for some $c>0$.
This is true under standardness and assuming that
$\inf_{y\in S} q(y) >0$.
However, we allow $q$ to be 0 at the boundary
and only require $q(y) \geq C d(y,\partial S)^\alpha$.
In this case,
by applying Lemma \ref{lemma::new-boundary},
we have that
$Q(B(y,\epsilon)) \geq c\epsilon^{2+\alpha}$.
The result then follows as in their proof
by replacing $\epsilon^2$ with
$\epsilon^{2+\alpha}$. $\Box$

\smallskip
\noindent We will also need the following property of the 
union-of-balls estimator $\hat{\partial S}$.

\begin{lemma} 
\label{lemma::nice-boundary}
Let $Y_1,\ldots, Y_n$ be a sample from $Q_{f,\sigma,h}$.
If $f$ is open and if $S\subset\hat{S}$ then
$\hat{\partial S}$ is a simple, closed curve.
If $f$ is closed and if $S\subset\hat{S}$ then
$\hat{\partial S}$ consists of two simple, closed curves
$\hat{\partial S}_0$ and $\hat{\partial S}_1$.
\end{lemma}

\subsection{From Boundaries to Filaments}
\label{sect::estimation}

We now give two estimators of $\Gamma_f$ which we call the {\em EDT estimator} 
and the {\em medial estimator}. By condition (A3), $\sigma < \Delta$ so that
$\Gamma_f = M(S)$.

The first estimator is inspired by the fact that the $\Gamma_f$ maximizes the EDT.
The second estimator is inspired by the 
following fact.
For a closed curve, $\partial S$ consists of two disjoint
pieces
$\partial S_0$ and $\partial S_1$ and
the medial axis is midway between 
$\partial S_0$ and $\partial S_1$.

\smallskip
The algorithm for the EDT estimator is as follows.
An example is shown in Figure \ref{fig::EDT}.

\smallskip

\HRule

\begin{center}
\underline{\sf The EDT Estimator}
\end{center}

{\sc Input}: support and boundary estimates $\hat S$ and $\hat{\partial S}$
and a radius $\epsilon > 0$.

\smallskip
{\sc Output}: a set of fitted values $\hat\Gamma$.

\smallskip
{\sc Algorithm}:
\begin{enumerate}
\item Compute $\hat\Lambda(y) = d(y,\hat{\partial S})$, for all $y \in \hat S.$  \vadjust{\vskip 2pt}
\item Set $\hat\sigma = \max_{y\in\hat S}\hat\Lambda(y)$.  \vadjust{\vskip 2pt}
\item Let $\delta = 2\epsilon$ and set
$\displaystyle \hat\Gamma = \{ y\in\hat S:\ d(y,\hat{\partial S}) \geq \hat\sigma -\delta\}$.
\end{enumerate}

\HRule

\smallskip
\noindent We remark that the choice $\delta = 2\epsilon$ in the EDT 
procedure is mainly for theoretical purposes. In practice, $\delta$ can 
be used as a tuning parameter.

\begin{figure}
\hbox to\textwidth{\hss \includegraphics[width=5in]{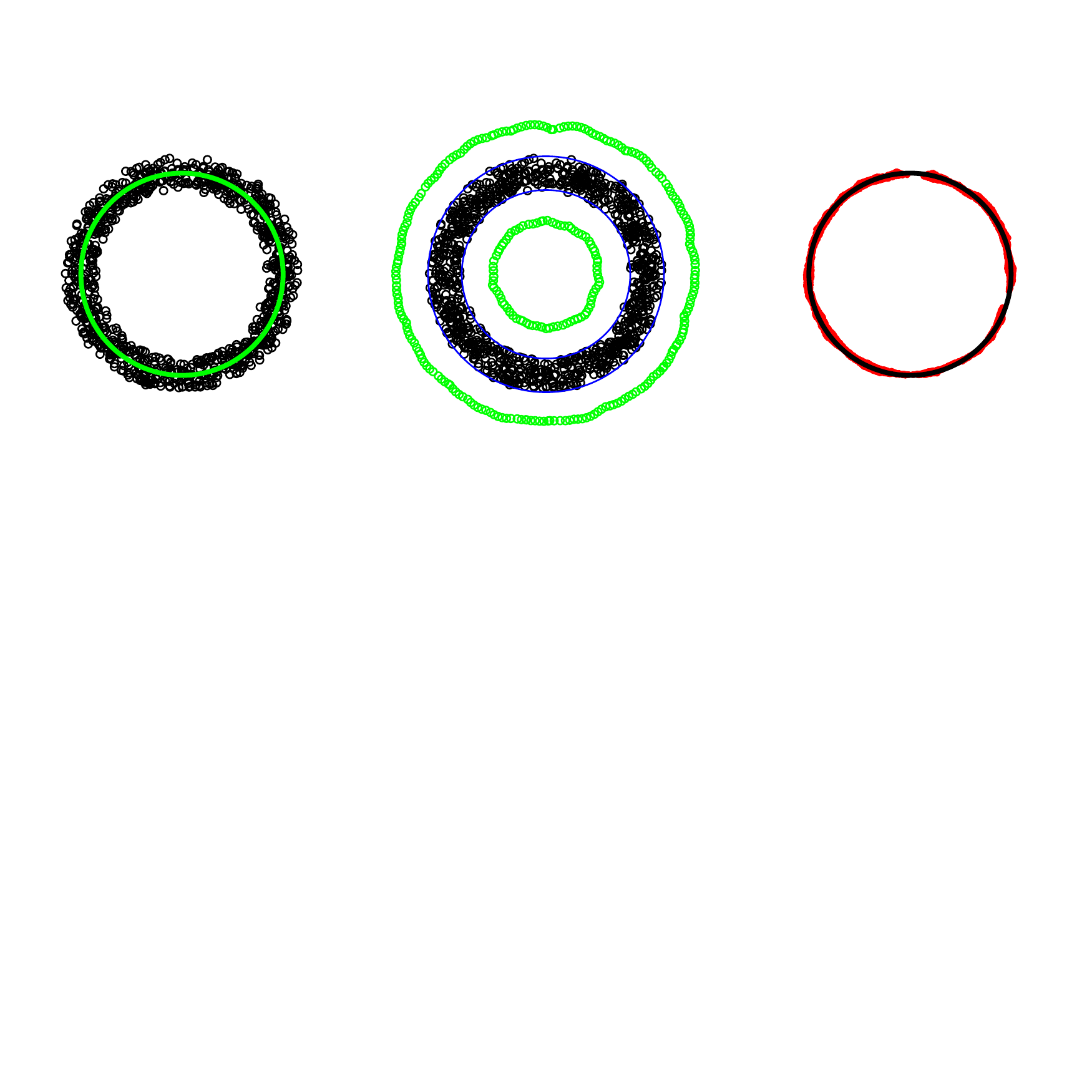}\hss}
\vspace{-3in}
\caption{These plots illustrate the EDT-based estimator.
Left: filament and data.
Center: Estimated boundary.
Right: EDT estimator $\hat\Gamma$.}
\label{fig::EDT}
\end{figure}

\smallskip
\begin{theorem}
\label{thm::edtestimator}
Let
$\hat\Gamma = \left\{y\in\hat S:\strut\ d(y,\hat{\partial S})\geq \hat\sigma-\delta \right\}$
be the EDT estimator,
where $\delta = 2 \epsilon$.
\begin{enumerate}
\item If $d_H(\partial S,\hat{\partial S}) \le \epsilon$, then 
$\Gamma \subset \hat\Gamma \subset \Gamma\oplus (4\epsilon)$,
and $d_H(\Gamma_f, \hat\Gamma) \le 4\epsilon$.
\item If $\hat{S} = \Union_{i=1}^n B(Y_i,\epsilon_n)$
where
$\epsilon_n = C (\log n/n)^{1/(2+\alpha)}$,
$C > \sqrt{2/(\chi \pi)}$
and $\chi=1/4$,
then, with probability one,
\begin{equation}
d_H(\Gamma_f,\hat\Gamma) = O(r_n)
\end{equation}
for all large $n$,
where
$r_n = \left(\frac{\log n}{n}\right)^{1/(2+\alpha)}$.
\end{enumerate}
\end{theorem}

\noindent Now we consider the medial estimator. In this case, we estimate the fibers
$L(u)$ by joining points on opposite sides of the estimated boundary. The 
algorithm for constructing the medial estimator follows:

\smallskip
\HRule
\begin{center}
\underline{\sf The Medial Estimator}
\end{center}

\setbox4=\hbox{{\sc Indent}:\hskip 0.75em}
\hangindent=\wd4
\hangafter=1
{\sc Input}: support and boundary estimates $\hat S$ and $\hat{\partial S}$, where
$\hat{\partial S}$ consists of two, disjoint
curves $\hat{\partial S}_0$ and $\hat{\partial S}_1$.

{\sc Output}: a set of fitted values $\hat\Gamma$.

\smallskip
{\sc Algorithm}:
\begin{enumerate}
\item For each $y\in \hat{\partial S}_0$, 
let $\hat y$ be the closest point on $\hat{\partial S}_1$ and \hfil\break
let $\hat\ell_y$ be the line segment connecting $y$ and $\hat y$.           \vadjust{\vskip 2pt}
\item Set $\hat\mu(y)$ to be the midpoint of $\hat\ell_y$.
\item Set $\displaystyle \hat\Gamma = \{ \hat\mu(y):\ y\in \hat{\partial S}_0\}$.
\end{enumerate}
\HRule

We will focus on analyzing this algorithm for closed curves.
The case of open curves is discussed in
Section \ref{sec::open}.

\begin{theorem}
\label{thm::methodII}
Let $\hat\Gamma$ be the medial estimator.
Then:
\begin{enumerate}
\item If $d_H(\partial S_0, \hat {\partial S_0}) \leq \epsilon$ and
$d_H(\partial S_1, \hat {\partial S_1}) \leq \epsilon$,
with $\epsilon < (\Delta - \sigma)/2$, then
\begin{itemize}
\item[(i)]
For every $\hat{\mu} \in \hat\Gamma$ there is a filament point $f(u)\in \Gamma_f$ such 
that $||\hat{\mu} - f(u)||\leq 2\epsilon$.
\item[(ii)]
There exists $C > 0$ such that, for each $f(u)\in \Gamma_f$ 
there is $\hat{\mu} \in \hat\Gamma$ such that $||\hat{\mu} - f(u)||\leq C \sqrt \epsilon$.
\item[(iii)]
$d_H(\hat\Gamma, \Gamma_f) = O(\sqrt \epsilon)$.
\end{itemize}
\item
If $\hat{S}=\Union_{i=1}^n B(Y_i,\epsilon_n)$
where
$\epsilon_n = C (\log n /n)^{1/(2+\alpha)}$,
$C > \sqrt{2/(\chi \pi)}$
and $\chi=1/4$,
then, with probability one, for all large $n$,
\begin{equation}
d_H(\Gamma_f,\hat\Gamma) = O(\sqrt{r_n}) 
\end{equation}
where $r_n = (\log n/n)^{1/(2+\alpha)}$.
\end{enumerate}
\end{theorem}

\begin{figure}
\hbox to\textwidth{\hss \includegraphics[width=5in]{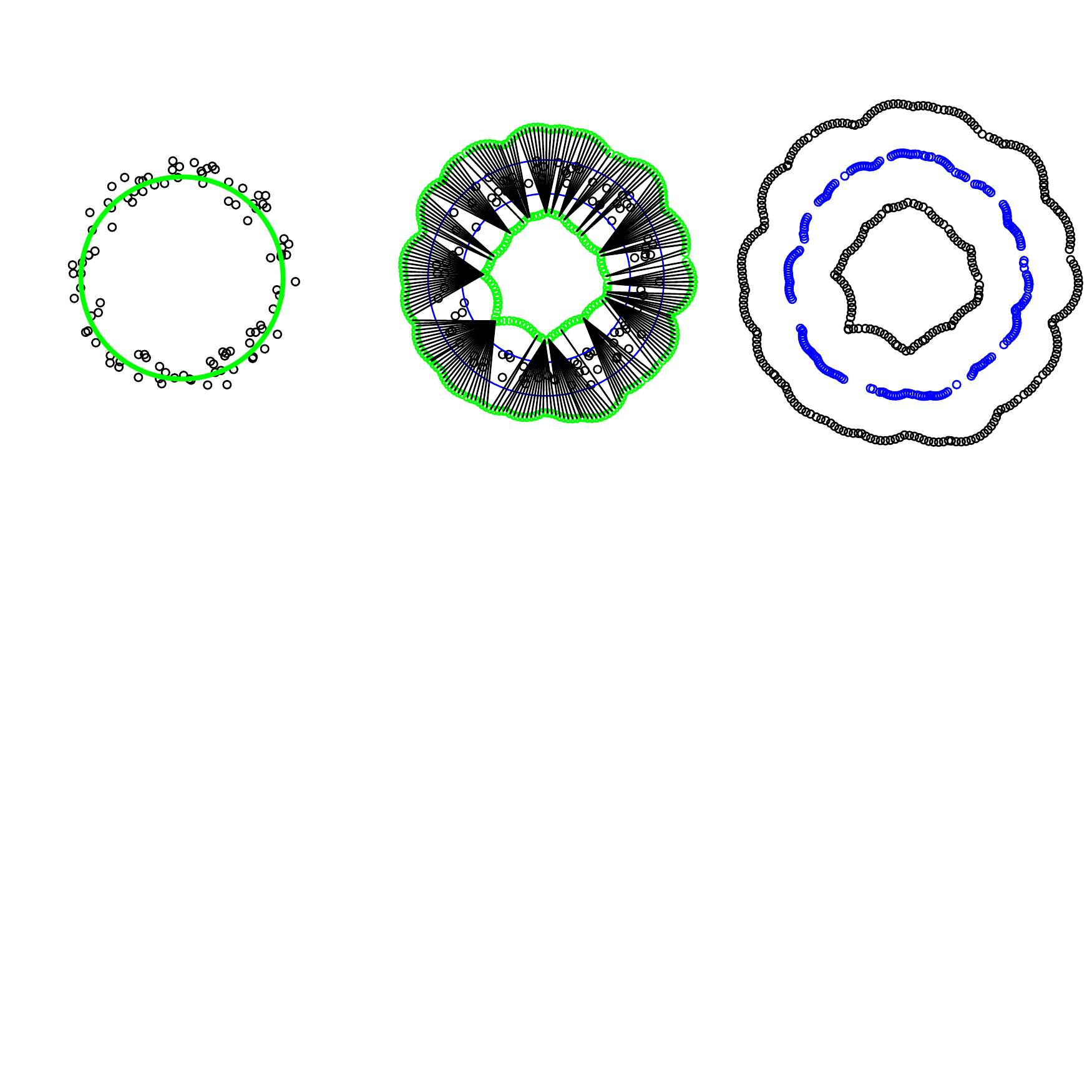}\hss}
\vspace{-3in}
\caption{These plots illustrate the medial estimator.
Left: filament and data.
Center: lines connecting the two boundary estimators.
Right: the medial estimator $\hat\Gamma$.}
\label{fig::MEDIAL}
\end{figure}

\medskip
An example is in Figure \ref{fig::MEDIAL}.
The medial estimator has a slower rate of convergence than the
EDT estimator.
However, Lemma \ref{lemma::union-of-open-curves} and Theorem \ref{thm::the-completion} below
show that it is easy to extract a curve from the fitted values. The extracted curve has 
the faster rate $r_n$ rather than~$\sqrt{r_n}$.

\smallskip
Let $\hat \Gamma$ be the medial estimator and assume that $f$ is closed. 
(The case 
where $f$ is open is considered in Subsection \ref{sec::open}.)
The fitted values $\hat\Gamma$ are derived from the 
estimated boundary $\hat{\partial S}$.
These fitted values have gaps.
All we have to do is connect the gaps with straight lines
to get a curve.
Surprisingly, this also improves the rate of convergence.
Here are the details.

Recall that, from Lemma \ref{lemma::nice-boundary},
$\hat{\partial S} = \hat{\partial S}_0 \union \hat{\partial S}_1$
and that $\hat{\partial S_0}$ is a closed simple curve. The medial estimator
takes each point $y\in\hat{\partial S}_0$ and outputs a fitted value $\hat\mu(y)$.
Let $g$ be a parameterization of $\hat{\partial S}_0$,
so
$\hat{\partial S}_0 = \{(g(u):\ 0 \leq u \leq 1\}$.
Define
$\hat{f}(u) = \hat\mu(g(u))$.

\begin{lemma}
\label{lemma::union-of-open-curves}
The function
$\hat f:[0,1]\to\mathbb{R}^2$
is a union of open curves.
In particular, there exist
$0 = a_0 < a_1 < \cdots < a_N =1$ such that
$\hat f$ is a continuous, open curve on each
$(a_j,a_{j+1})$ but
$\hat f$ is possibly discontinuous at each $a_j$.
\end{lemma}

Now we define $f^*$ as follows.
In general, $\hat f(a_j^-) \neq \hat f(a_j^+)$.
We define $f^*$ to be the curve obtained by joining
$\hat f(a_j^-)$ and $\hat f(a_j^+)$ by linear interpolation.
We call
$\Gamma^*= \{f^*(u): \ 0 \leq u \leq 1\}$
the {\em completed medial estimator}.

\begin{theorem}
\label{thm::the-completion}
$f^*$ is a simple, closed curve.
Furthermore,
$d_H(\Gamma_{f^*},\Gamma_f) = O_P(r_n)$.
\end{theorem}

\vspace{.5cm}

\emph{Multiple Filaments}.
Suppose now that there are finitely many
filaments $f_1,\ldots, f_k$.
First suppose that
$d_{\rm min}(\Gamma_{f_j},\Gamma_{f_k}) > 2\sigma$ for all $j\neq k$
where
$d_{\rm min}(A,B)=\min_{x\in A,y\in B}||x-y||$.
The properties of $\hat S$ guarantee that
for large enough $n$,
$\hat S$ will consist of disjoint, connected sets
$\hat{S}_1, \ldots, \hat{S}_k$.

\begin{corollary}
Suppose that $\sigma < \min_j \Delta(f_j)$, where $\Delta(f_j)$ denotes 
the thickness of the curve $f_j$, and that
$d_{\rm min}(\Gamma_{f_j},\Gamma_{f_k}) > 2\sigma$ for all $j\neq k$.
If the EDT or medial procedure is applied then
$$
\max_j d_H(\Gamma_{f_j},\hat\Gamma_j) = O_P(r_n)
$$
where $r_n$ is as before.
\end{corollary}

When the condition
$d_{\rm min}(\Gamma_{f_j},\Gamma_{f_k}) > 2\sigma$ fails,
then the curves can get close to each other or even could be
self-intersecting.
In that case, we cannot claim to estimate the entire curve well.
However, we can estimate the well-separated portions of the curves.
Let $\Gamma = \Union_{j=1}^k \Gamma_{f_j}$.
For each $y\in \Gamma$ let
$N(y) = \{j:\ B(y,2\sigma)\cap\Gamma_{f_j} \neq \emptyset\}$.
Let $\Gamma_0 = \{y\in\Gamma:\ |N(y)|=1\}$.

\begin{corollary}
Suppose that $\sigma < \min_j \Delta(f_j)$.
If either the EDT or medial procedures are applied then
$$
d_H(\Gamma_0,\hat\Gamma) = O_P(r_n)
$$
where $r_n = \sqrt{\log n/n}$ for the EDT estimator
and $r_n = (\log n/n)^{1/4}$ for the medial estimator.
\end{corollary}

\subsection{Extracting a curve from EDT estimator}
\label{sect::ALGextract}
Now we discuss how to extract a curve from the fitted values.
We assume that we have already computed 
the union of balls estimator $\hat S$ with an appropriate choice of $\epsilon_n$
and hence that $d_H(S,\hat S) \leq C r_n$ and
$d_H(\partial S,\hat{\partial S}) \leq C r_n$
for some $C>0$.

\smallskip 
Let $\hat\Gamma$ denote the fitted values from the EDT estimator.
Our goal is to use $\hat\Gamma$ to find a curve $\hat f$ such
$d_H(\Gamma_{\hat f},\Gamma_f) \le C d_H(\partial S,\hat{\partial S})$.
Such a curve $\hat f$ can be identified both for open and closed filaments.
The precise statement is given in Theorem \ref{thm::extract} below. 

More informally, recall first that from Theorem \ref{thm::edtestimator}, 
$\Gamma \subset \hat \Gamma \subset \Gamma \oplus (4\epsilon)$.
Now, when $f$ is an open filament, from Lemma \ref{lemma::disjoint},
$S = \cT \cup \cC_0 \cup \cC_1$. Thus, let
$y_0 \in \hat \Gamma \intersect\cC_0$ and $y_1 \in \hat \Gamma \intersect \cC_1$
be points in $\hat \Gamma$ and the two end-caps of $S$.
Any curve $\Gamma_{\hat f}$ between $y_0$ and $y_1$ that lies entirely in $\hat \Gamma$
must cut through every fiber in $\cT$ at a distance at most $4\epsilon$
and it is at most $4\epsilon$ from the end points $f(0)$ and $f(1)$.
Hence $d_H(\Gamma_{\hat f},\Gamma_f) \le 4\epsilon$.

When, instead, $f$ is a closed filament, let $y_0$ be a point in 
${\hat S}^c$ surrounded by $\hat{\partial S_0}$. 
Any closed curve $\Gamma_{\hat f}$ that lies entirely within $\hat \Gamma$
and has winding number $1$ with respect to $y_0$ cuts through every fiber in $\cT$
at a distance at most $4\epsilon$ from $\Gamma_f$. In this case too
$d_H(\Gamma_{\hat f},\Gamma_f) \le 4\epsilon$.

\smallskip
The extraction algorithm is based on the remarks above. 
In the open filament case, because $\cC_0$ and $\cC_1$ are unknown, 
we replace $y_0$ and $y_1$ by estimated end-points $\hat x_0$ and $\hat x_1$ that maximize
the minimum path length between two points in $\hat \Gamma$, as illustrated later in 
Subsection \ref{sec::Extrac}. In the closed 
filament case we use a slightly different implementation, that generalizes 
more readily to the case where it is not known if the filament is open or 
closed.

\smallskip
\HRule
\begin{center}
\underline{\sf EDT Curve Extraction Algorithm}
\end{center}

\setbox4=\hbox{{\sc Indent}:\hskip 0.75em}
\hangindent=\wd4
\hangafter=1
{\sc Input}: EDT Estimate $\hat\Gamma$ and corresponding $\epsilon > 0$,
and constraint sets $\cE_0$ and $\cE_1$. ($\cE_0 = \cE_1 = \mathbb{R}^2$
by default).

\smallskip
{\sc Output}: the graph of a curve $\hat{\Gamma}$.

\smallskip
{\sc Algorithm (Open-Curve Case)}:
\begin{enumerate}
\item Find end points $\hat x_0$ and $\hat x_1$ satisfying
\begin{equation}
\hat x_0, \hat x_1 = \argmax_{u\in\hat\Gamma\intersect\cE_0, 
v\in\hat\Gamma\intersect\cE_1} \min_{\pi\in\cP_{u,v}} {\rm length}(\pi),
\end{equation}
where $\cP_{u,v}$ is the set of paths in $\hat\Gamma$ from $u$ to $v$.
In practice, this is accomplished by constructing a $\xi$-net of points in $\hat\Gamma$
with $0 < \xi < \epsilon/4$; forming the minimum spanning tree 
of this net;
and finding the points that maximize the minimum path length in the tree.
\item Join the end points by a curve in $\hat\Gamma$.
In practice, this is obtained from the minimum spanning via
Dijkstra's algorithm (\cite{Dijkstra}).
\item (Optional) Relax the path to thickness $\Delta$ as follows:
for each successive triple of points $(y_{i-1},y_i,y_{i+1})$ on the path,
shrink $y_i$ as close to $(y_{i+1}+y_{i-1})/2$ while remaining in $\hat\Gamma$.
Iterate until the reduction in thickness is below a fixed threshold.
\end{enumerate}

\smallskip
{\sc Algorithm (Closed-Curve Case)}:
\begin{enumerate}
\item Fix $0 < \eta \ll \epsilon$.
\item Let $\hat y$ be the point defined in equation (\ref{eq::yhat}) 
that determines $\hat \sigma$.
\item Let $\cA_8$ be the union of all line segments through $\hat y$ with end points 
on $\partial\hat\Gamma$ and whose length is $\le 8\epsilon$.
\item Define $\cA = (\cA_8 \intersect \hat\Gamma)\oplus\eta$.
\item Apply the open-curve algorithm to $\hat\Gamma - \cA$ with the constraint that the end points
of the curve, $\hat x_0$ and $\hat x_1$, must both lie on $\partial\cA$
(i.e., set $\cE_0 = \cE_1 = \partial\cA$).
\item Join $\hat x_0$ and $\hat x_1$ by a curve contained within $\cA$, producing a single
closed curve.
\end{enumerate}

\smallskip
{\sc Algorithm (General-Curve Case)}:
\begin{enumerate}
\item Construct $\cA$ as in the closed curve algorithm
\item If $\hat\Gamma - \cA$ has one connected component, continue with the closed-curve algorithm.
(This can, for instance, be determined using a friends-of-friends with a threshold distance
of $\eta$ from the closed-curve algorithm.)
\item Otherwise, $\hat\Gamma - \cA$ must have two connected components. Do the following:
\begin{enumerate}
\item Apply the open-curve algorithm to each component with the constraint that the one of the
end points in each component must lie on the boundary of $\cA$
(i.e., $\cE_0 = \mathbb{R}^2$ and $\cE_1 = \partial\cA$ for the first component and vice
versa for the second).
\item Join the endpoints on the boundary of $\cA$ with any path through $\cA$
to create a single curve.
\end{enumerate}
\end{enumerate}

\HRule
\medskip

For the open-curve case, specification of $\xi$ is
arbitrary. Smaller $\xi$ give larger nets and lead more
convoluted initial paths but allow more effective smoothing
in the relaxation step.  The minimum spanning tree end
points can be refined by using the expected hitting times
for a random walk on the $\xi$-net. Restricting the random
walk to suitably small steps of order $\epsilon$ gives a
sparse transition matrix. The expected hitting time from one
end point to all other points can be maximized to refine the
other end point and so on, alternating end points. This
process tends to converge rather quickly and produces better
results in practice.
Relaxation is optional but must be used if a smooth $\hat\Gamma$
is desired.

For the closed curve case, the choice of $\eta$ is again
arbitrary, a non-zero value is needed to provide clean
separation. The set $\cA$ can be replaced in practice with
the intersection of $\hat\Gamma$ and a ball of radius
$6\epsilon$ around $\hat y$, which is easier to compute, if
somewhat more conservative.

The following theorem shows that the algorithm produces curves
with the desired properties.

\begin{theorem}
\label{thm::extract}
Let $\hat\Gamma$ denote the curve extracted from the EDT
estimator by the algorithm described above. 
Assume that $d_H(\partial S,\hat{\partial S}) \le \epsilon$.
Then, 
\begin{enumerate}
\item If $f$ is closed, $d_H(\hat\Gamma,\Gamma_f) \le 4\epsilon$.
\item If $f$ is open, $d_H(\hat\Gamma,\Gamma_f) \le 16\epsilon$.
\end{enumerate}
\end{theorem}

\noindent An example of curve extraction is shown in Figure  \ref{fig::example1-EDT-extract}.

\subsection{Decluttering} \label{sec::decluttering}

Assume now that
$Y_i$
has density $m(y)=(1-\eta)q_0(y) + \eta q(y)$
where $q_0$ is the uniform density over a compact set ${\cal C}$
and $q$ is the density of points from the filament.
We assume that $S\subset {\cal C}$ where $S$ is the support of $q$.
Thus,
$q_0(x) = I(x\in {\cal C})/V$ where
$V$ is the area of ${\cal C}$.

Let $Z_i=1$ if $Y_i$ is from $q$ and
$Z_i=0$ if $Y_i$ is from $q_0$.
To identify clutter,
we want to find a classifier $c(y)$ where
$c(Y)=1$ means that we guess that $Z=1$ and
$c(Y)=0$ means that we guess that $Z=0$.

The best classifier is the Bayes' rule, 
\begin{equation}
c_*(y) = I\bigl(\mathbb{P}(Z_i=1|Y_i) \geq 1/2\bigr) = I\bigl(m(y) \geq 2 (1-\eta)\ q_0(y)\bigr)
\end{equation}
where
$$
\mathbb{P}(Z_i=1|Y_i) = \frac{q(Y_i)\ \eta}{m(Y_i)}.
$$
The Bayes rule is not identifiable.
Since $1-\eta \leq 1$, a conservative approximation to the Bayes rule is
\begin{equation}
I\bigl(m(y) \geq 2\  q_0(y)\bigr).
\end{equation}
An estimate of $c$ is
$\hat{c}\ (y) = I(\hat{m}(y) \geq 2 q_0(y))$
where $\hat{m}$ is a density estimator
obtained from $Y_1, \ldots, Y_n$.
In practice we use a kernel density estimator.
We can now apply the previous filament algorithms to the decluttered data set
\begin{equation}
\Bigl\{ Y_i:\ \hat{c}\ (Y_i) =1 \Bigr\}.
\end{equation}
An investigation into the properties
of this decluttering process 
is beyond the scope of this paper
and will be reported elsewhere.
However, we will illustrate the decluttering procedure in the examples
and show that it appears to perform well in practice.

\section{Examples}

We have tested our procedures on a few simulated data-sets.
We start by considering two smooth filaments, one open and the other closed.
In the first example the two filaments are well separated (top left panel
in Figure \ref{fig::example1}) while in the second dataset
the two filaments intersect (top left panel in Figure \ref{fig::example2}).
The third example considers 12 different smooth open filaments, with several 
intersections.  

Note that the condition on the radius of curvature fails to hold in presence of 
intersections between filaments, thus only the first dataset satisfies the conditions of
this paper completely.

In all the examples we have chosen $\epsilon_n$ according to the suggestion in
\cite{cuevas-boundary} as follows:
\begin{equation}
\epsilon_n = 
\max_{1\leq i \leq n}  \min_{j \neq i}  || Y_i - Y_j ||.
\end{equation}

The first two datasets contain 1500 points: 500 of which on each filament and
500 points of background clutter (top right panels in Figures \ref{fig::example1} 
and \ref{fig::example2}).

A summary of the results from the decluttering procedure is given in Figure 
\ref{fig::table} for both dataset. The procedure seems to work well in separating 
filament from clutter points. 

\begin{figure}[h]
\begin{center}
\begin{tabular}{l|cc|c}
                  & \multicolumn{2}{c}{Marked as} & \\ 
True              & filament & clutter & Total \\ \hline
filament          &      990 &      10 &  1000 \\
clutter           &       82 &     418 &   500 \\ \hline
Total             &     1072 &     428 &  1500    
\end{tabular}
\qquad
\begin{tabular}{l|cc|c}
                  & \multicolumn{2}{c}{Marked as} & \\ 
True              & filament & clutter & Total \\ \hline
filament          &      965 &      35 &  1000 \\
clutter           &       89 &     411 &   500 \\ \hline
Total             &     1054 &     446 &  1500    
\end{tabular}
\caption{Summary of decluttering on the first dataset (left) and second dataset (right).}
\label{fig::table}
\end{center}
\end{figure}

The filaments were estimated with the EDT and the Medial Estimator methods
of subsection \ref{sect::estimation}, applied to the decluttered datasets.
The estimated filaments obtained for the first dataset are very close to the true
(bottom panels in Figure \ref{fig::example1}). 

\begin{figure}
\begin{tabular}{cc}
\includegraphics[width=2.5in]{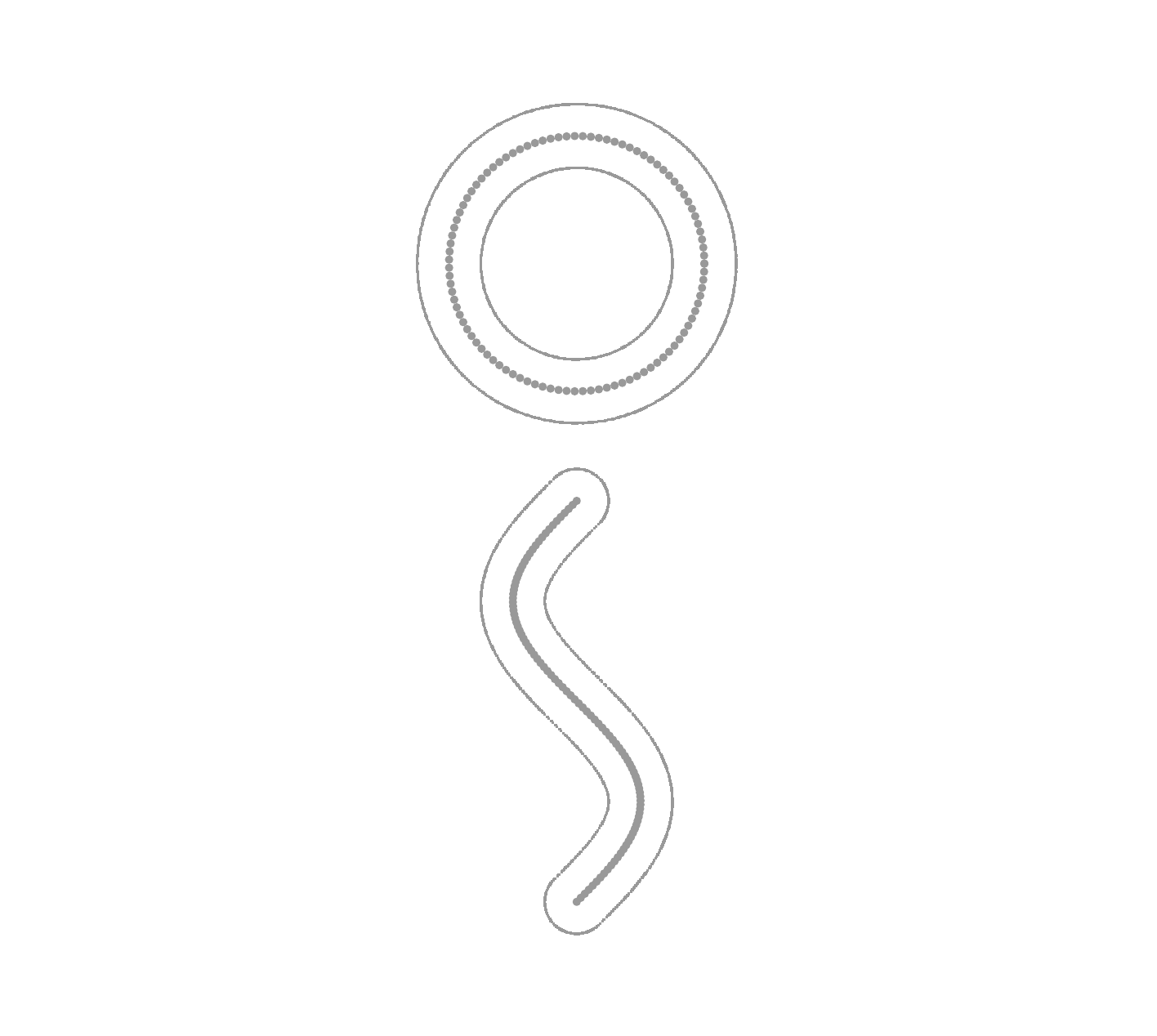} & \includegraphics[width=2.5in]{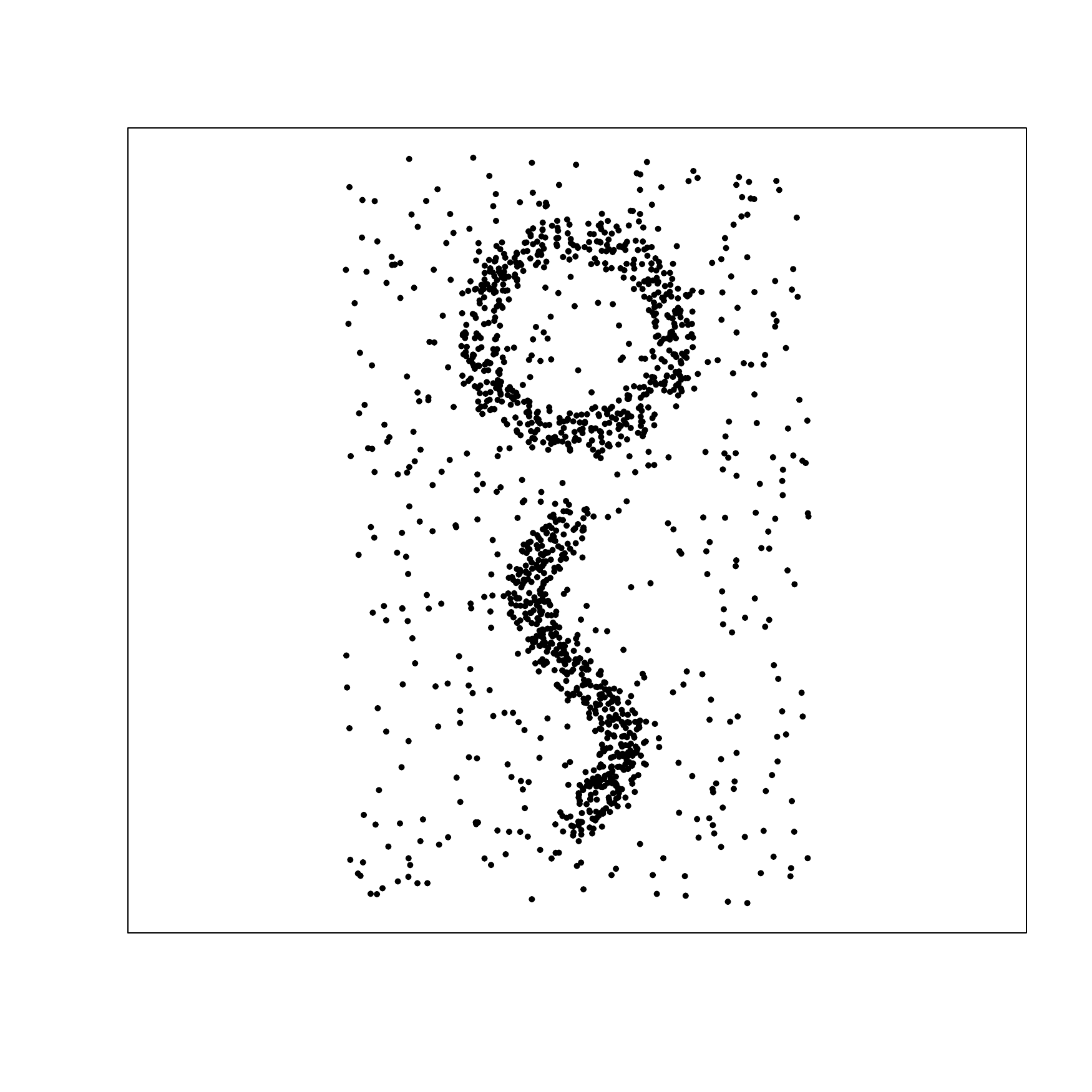} \\
\includegraphics[width=2.5in]{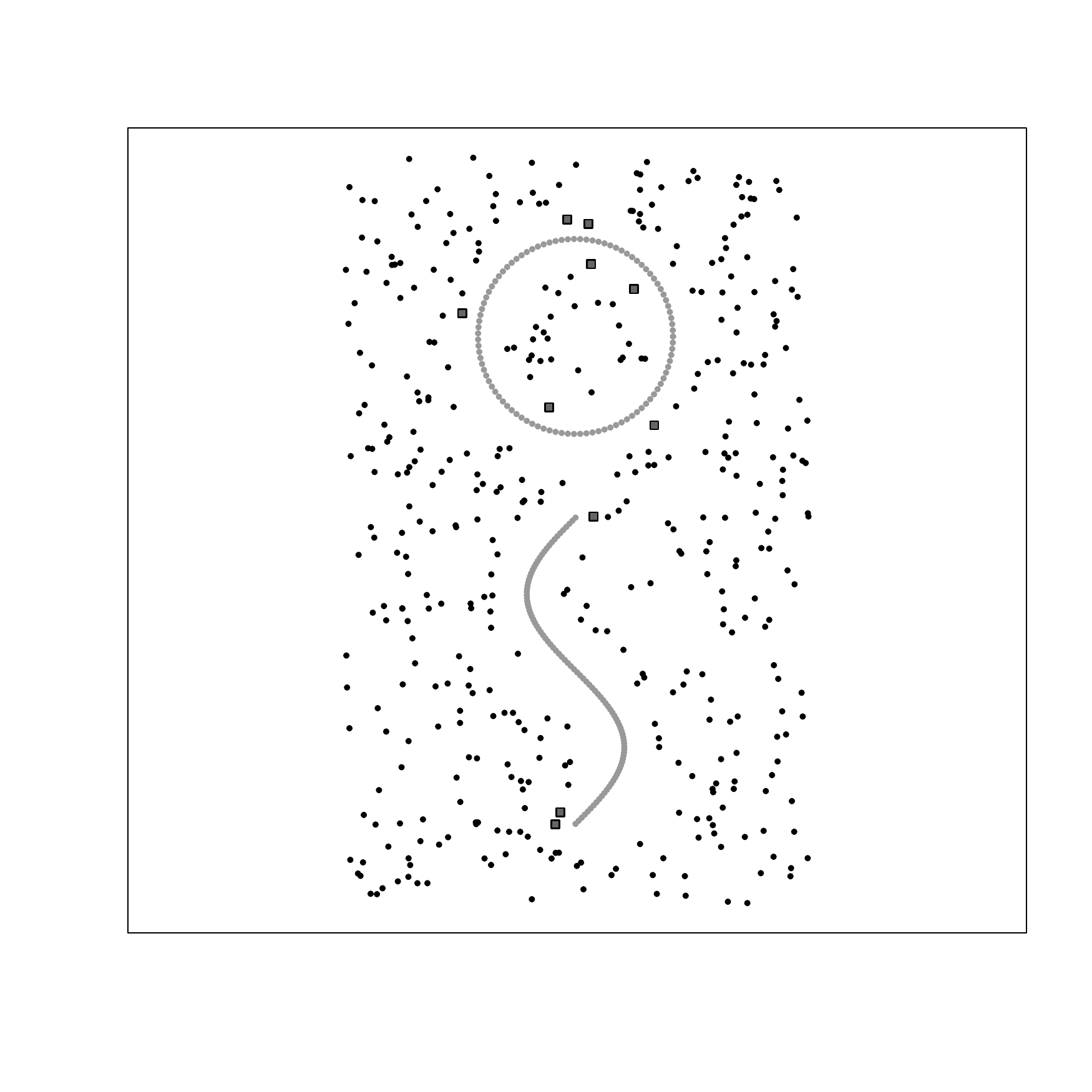} & \includegraphics[width=2.5in]{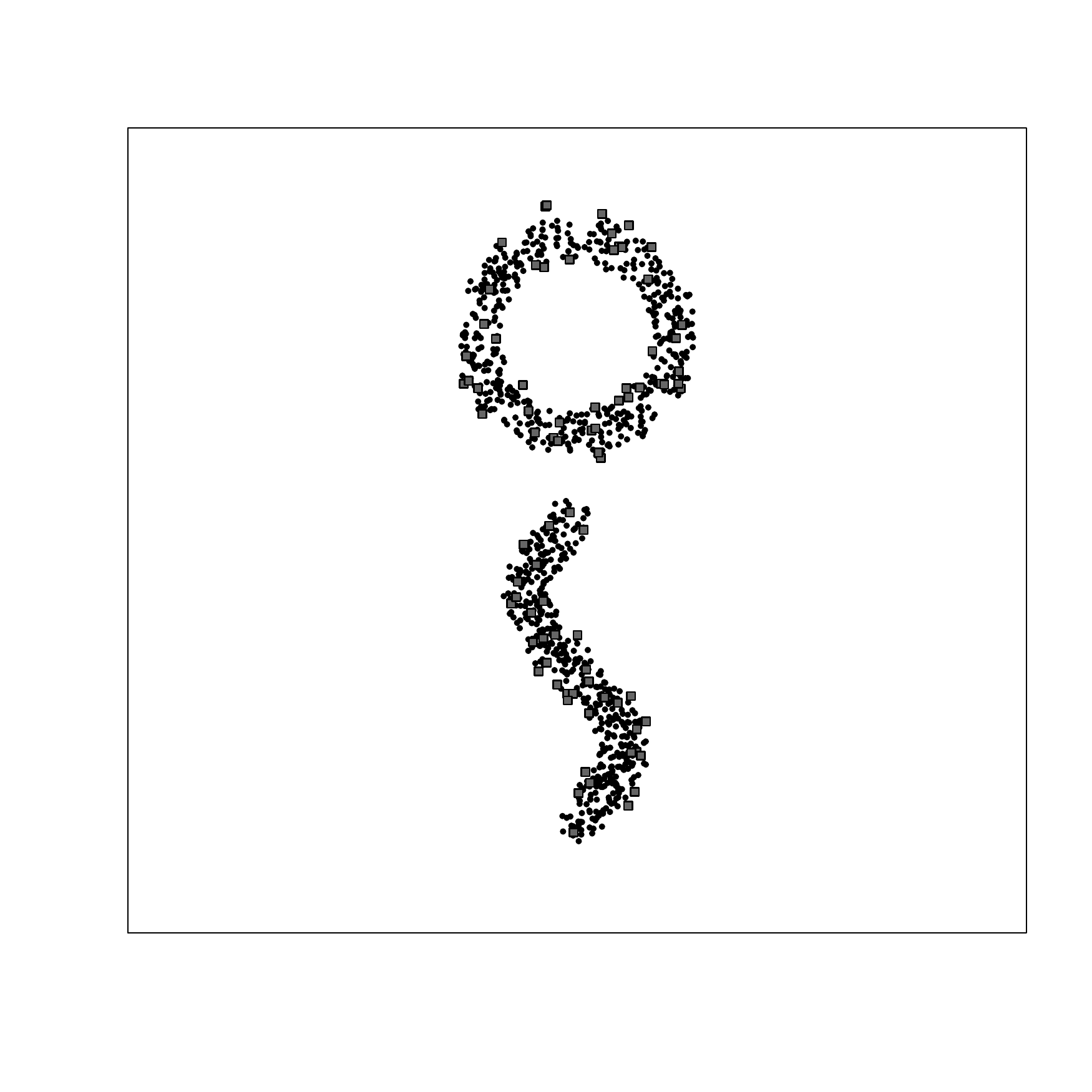} \\
\includegraphics[width=2.5in]{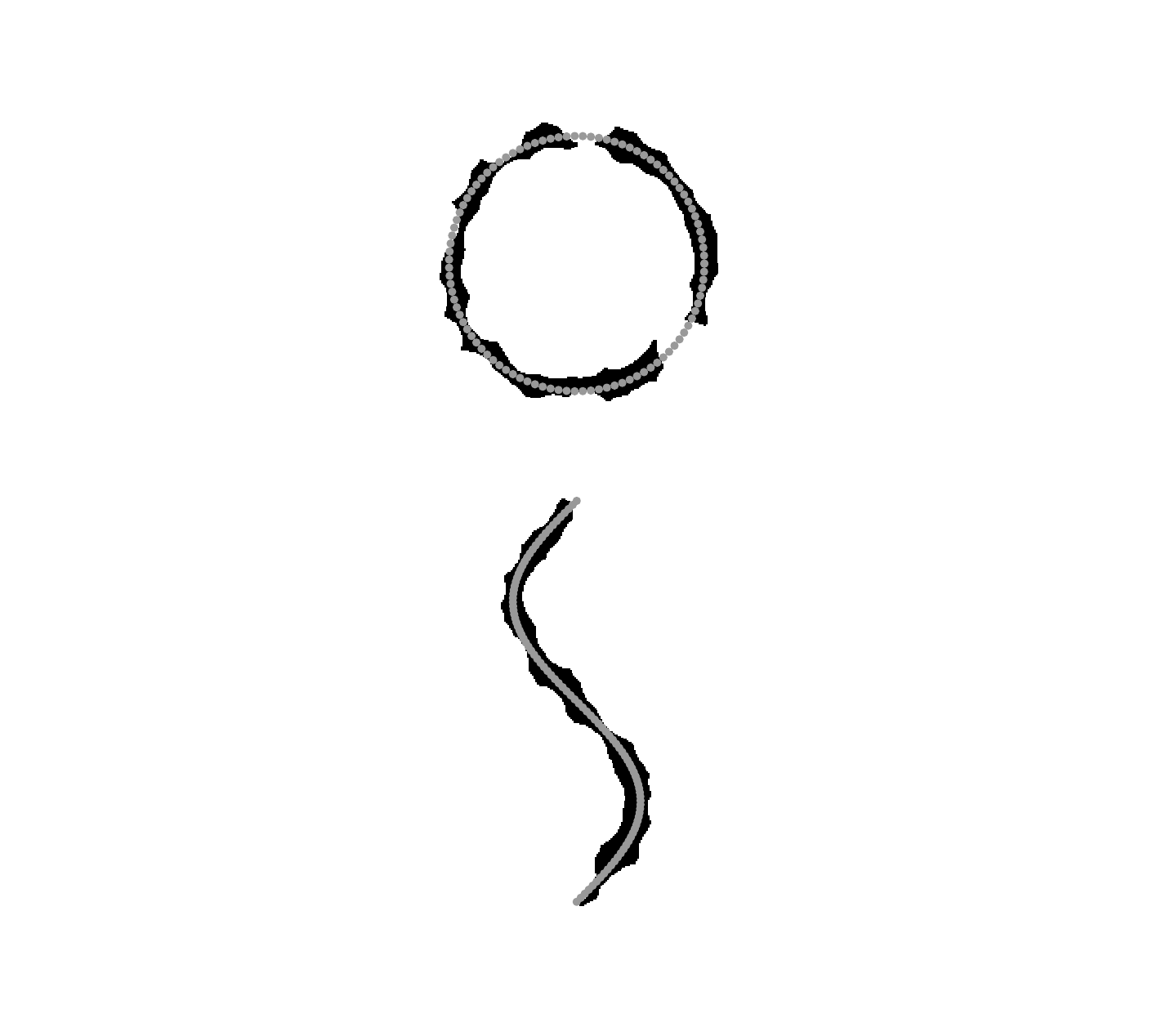} & \includegraphics[width=2.5in]{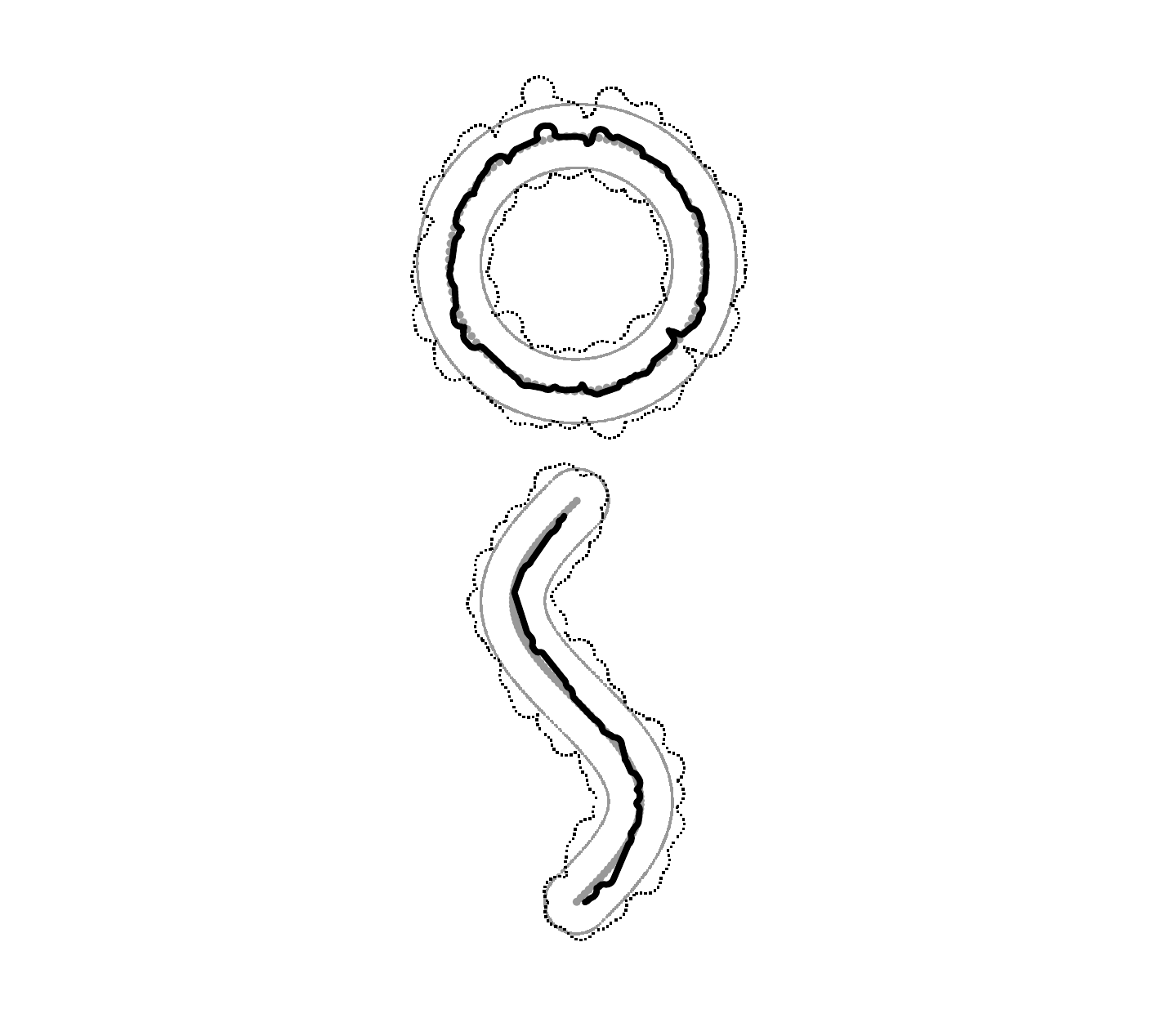} \\
\end{tabular}
\caption{First example.
Top line: true curves and the support of the distribution (left), the data (right).
Center line: points identified as clutter (left), decluttered data (right).
Bottom line: EDT estimator (left), Medial estimator (right).
}
\label{fig::example1}
\end{figure}

\begin{figure}
\hbox to\textwidth{\hss\hskip -.1in \vbox to 5in{\includegraphics[width=2.5in]{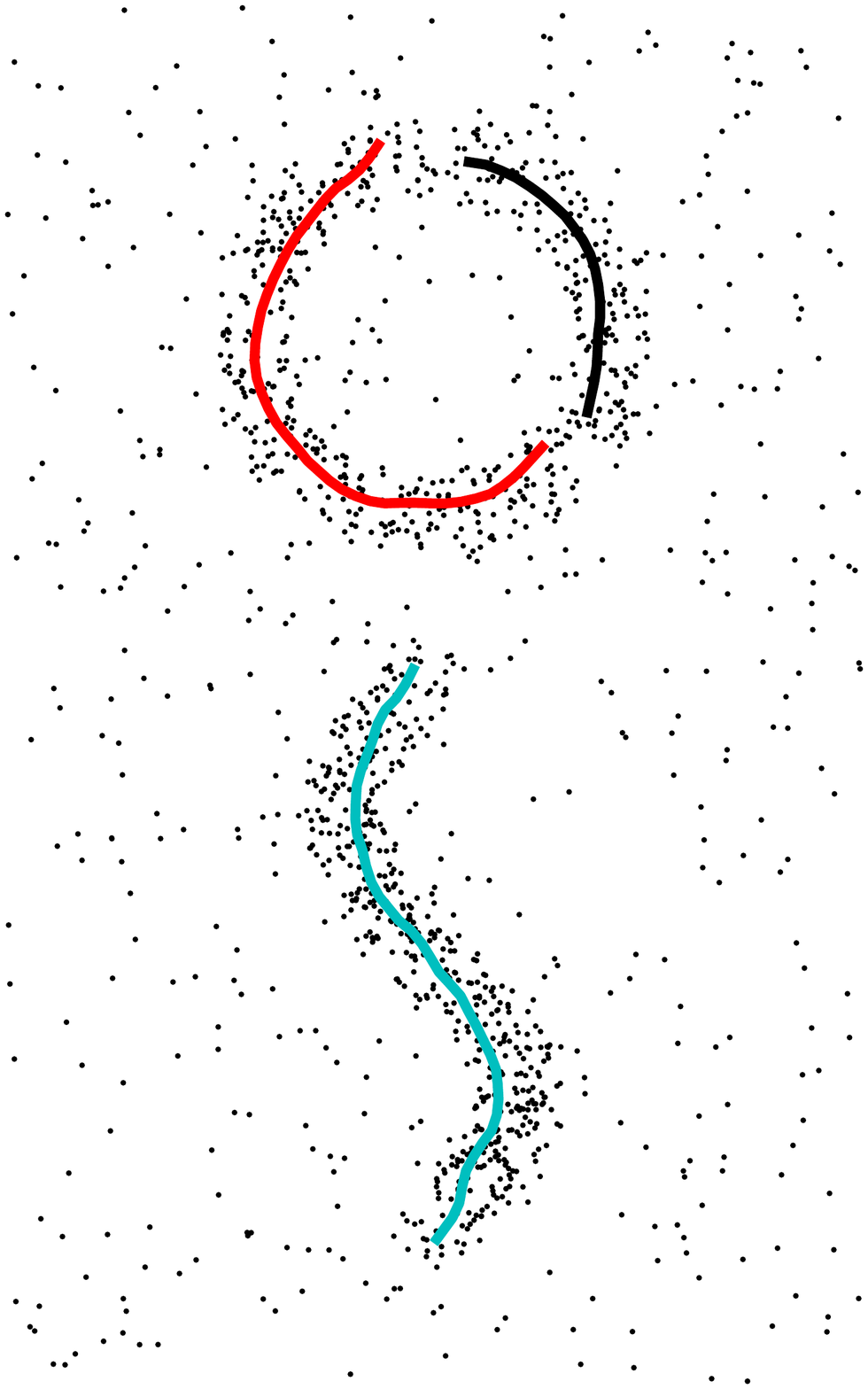}\vss}\hss}
\vspace{-2.2in}
\caption{First example. Curves extracted from the EDT estimator. Data with background clutter
overlayed.}
\label{fig::example1-EDT-extract}
\end{figure}

We applied the curve extraction procedure of subsection \ref{sect::ALGextract} to
the EDT estimator shown in the bottom left panel of Figure \ref{fig::example1}.
Figure \ref{fig::example1-EDT-extract} shows the extracted curves.

The estimated filaments obtained for the first dataset are very close to the true
(bottom panels in Figure \ref{fig::example1}). For the second dataset (bottom panels
in Figure \ref{fig::example2}) the medial estimator fails to detect the true filament 
near the intersection and becomes more and more accurate as it moves away from the 
intersection. Considering that the condition on the radius of curvature is violated, 
even in the second dataset the estimate seems to be quite satisfactory.

\begin{figure}
\begin{tabular}{cc}
\includegraphics[width=2.5in]{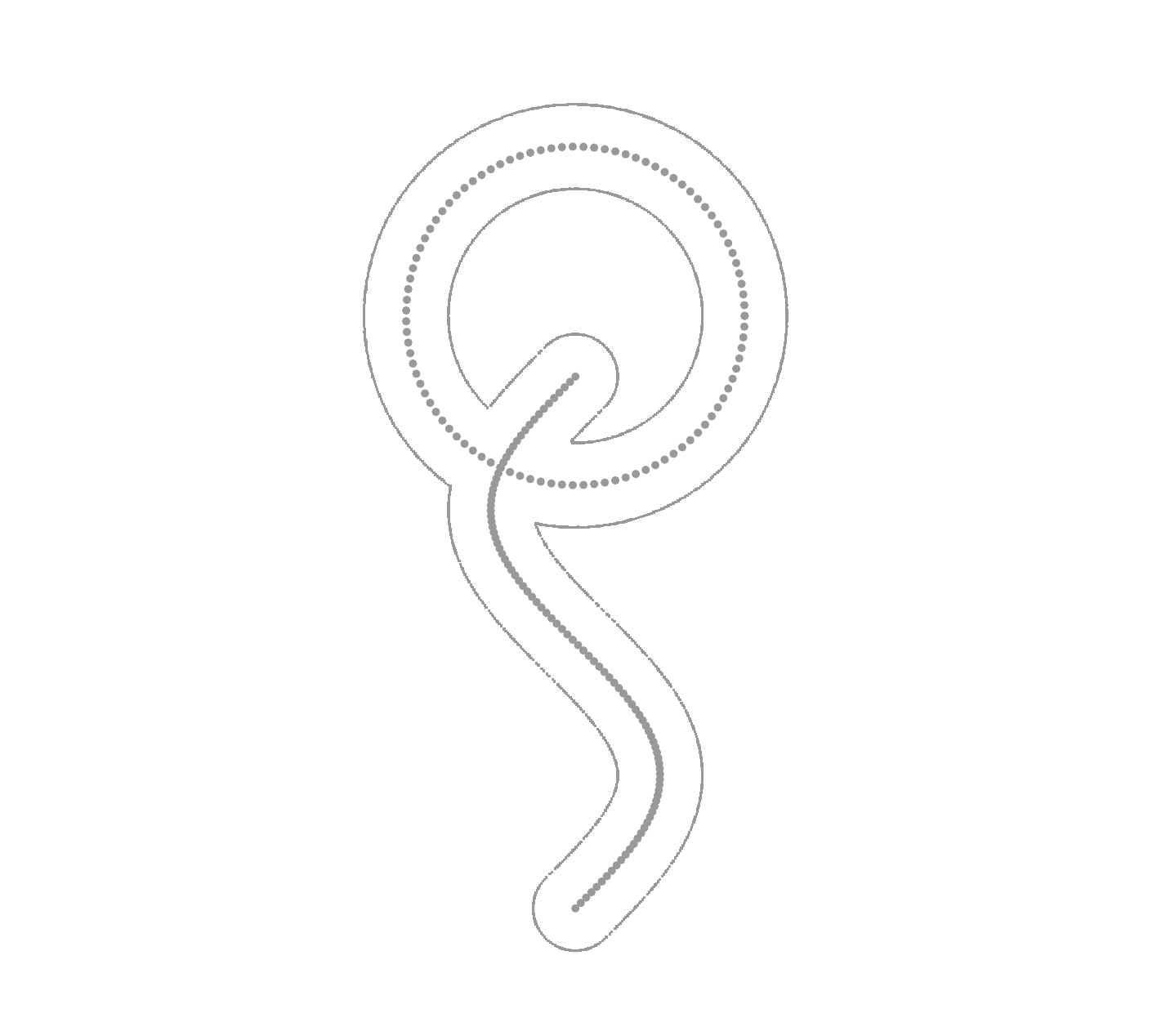} & \includegraphics[width=2.5in]{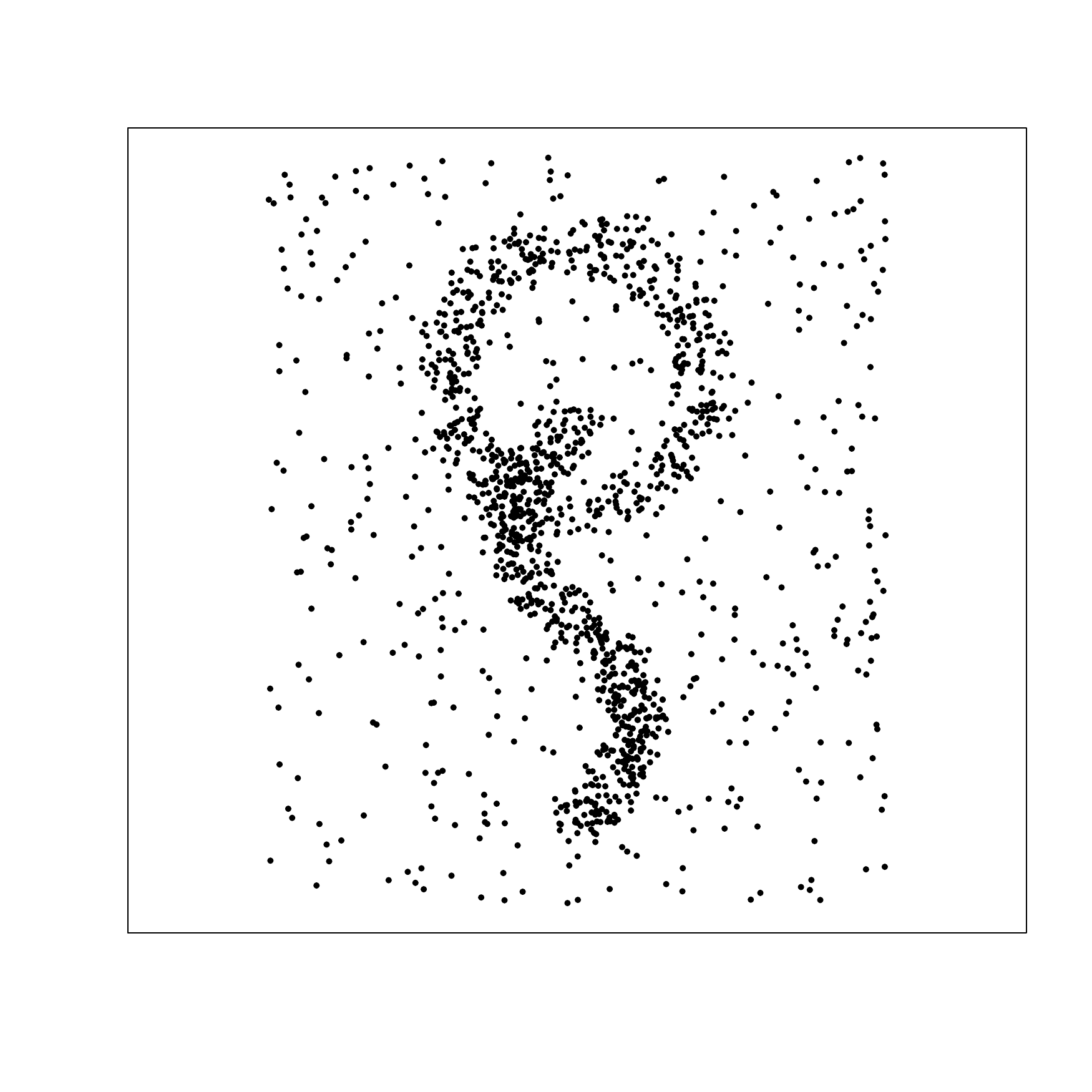} \\
\includegraphics[width=2.5in]{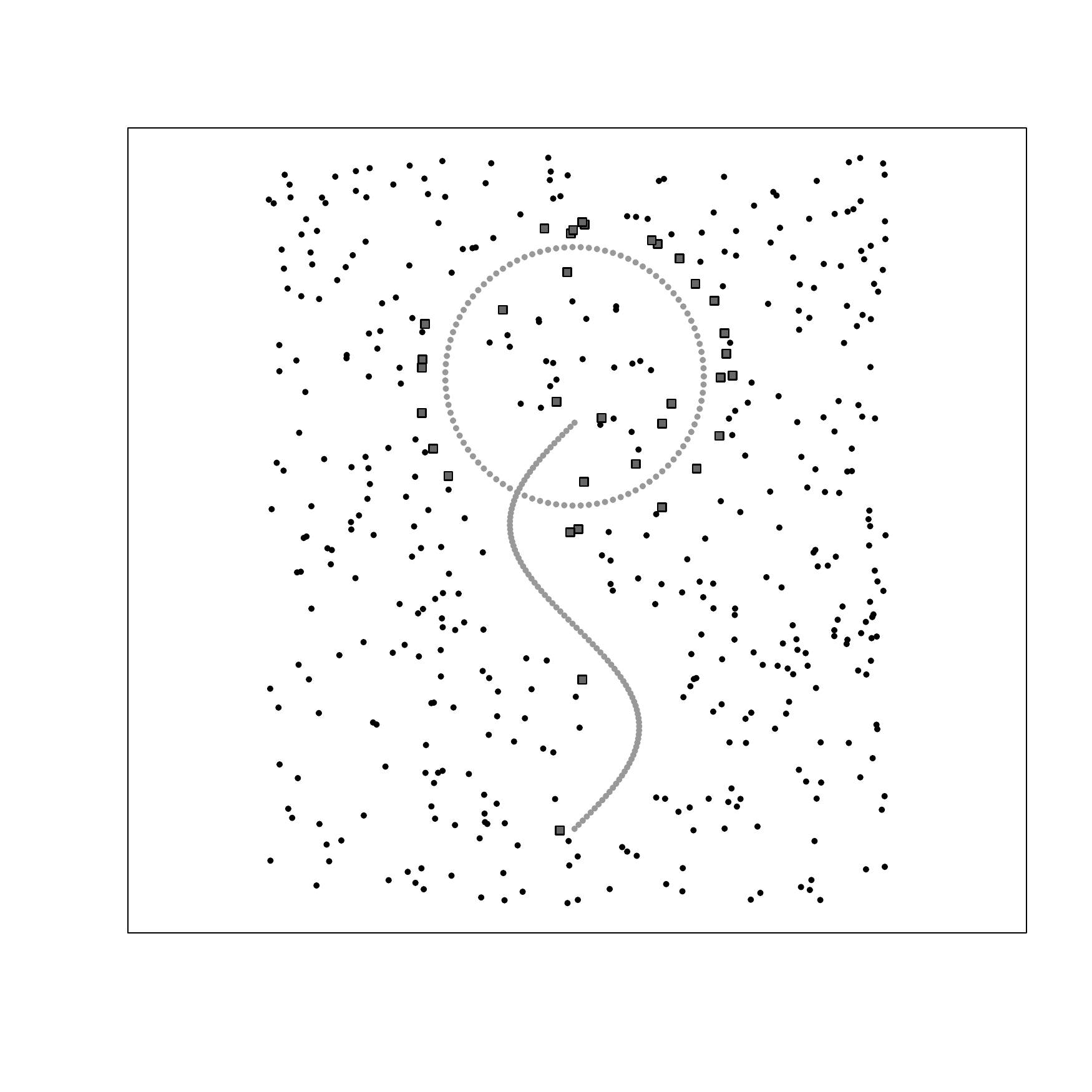} & \includegraphics[width=2.5in]{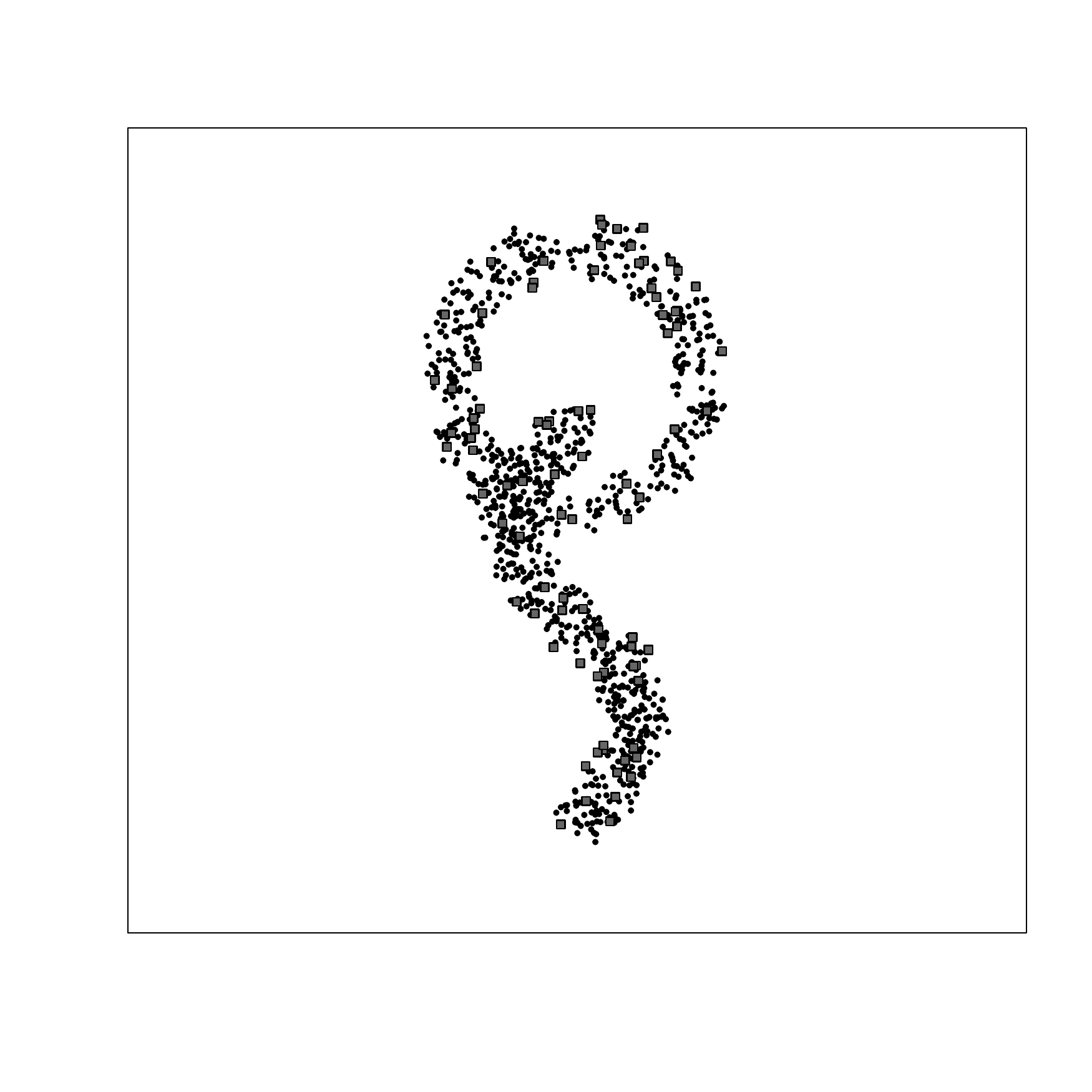} \\
\includegraphics[width=2.5in]{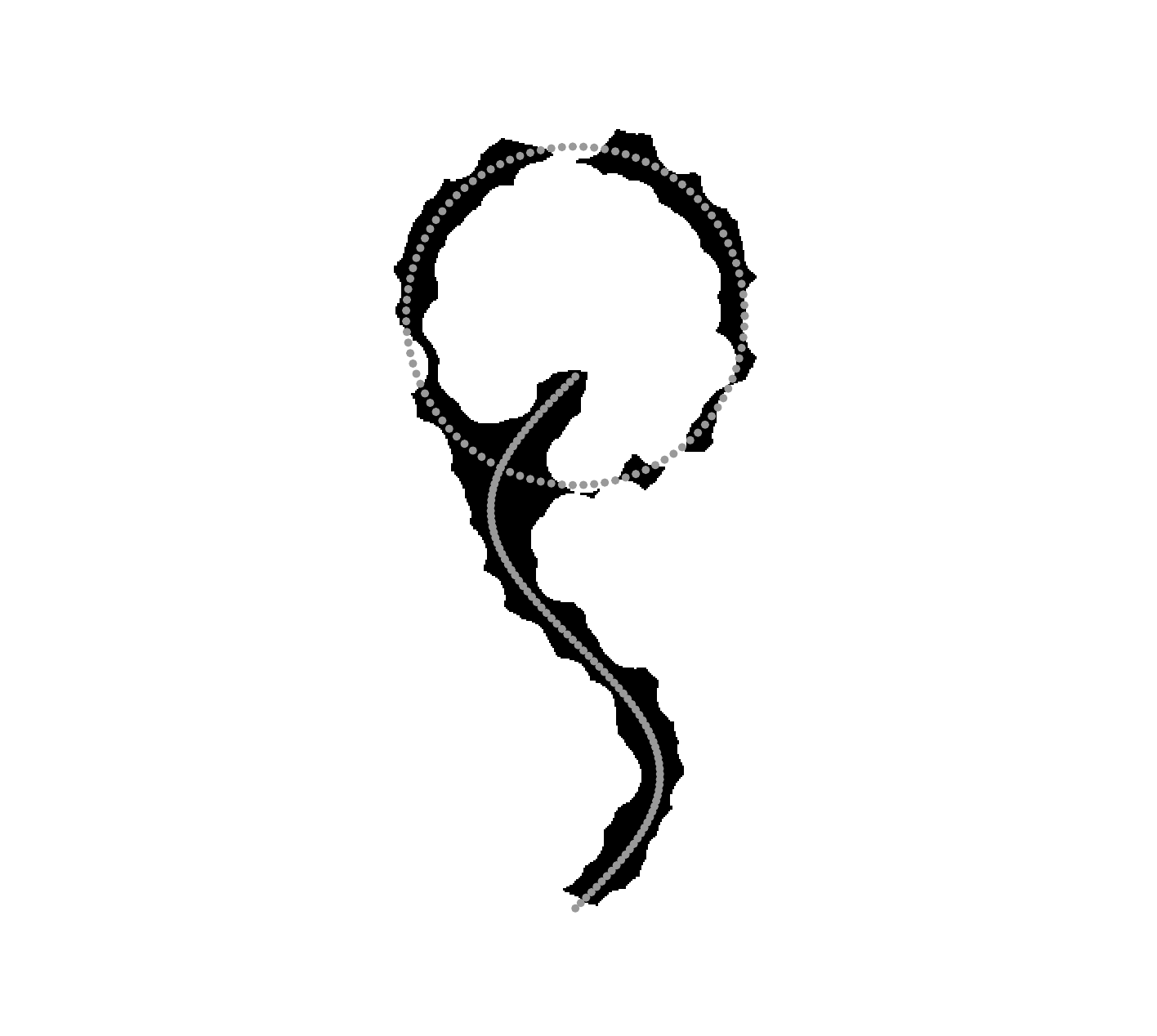} & \includegraphics[width=2.5in]{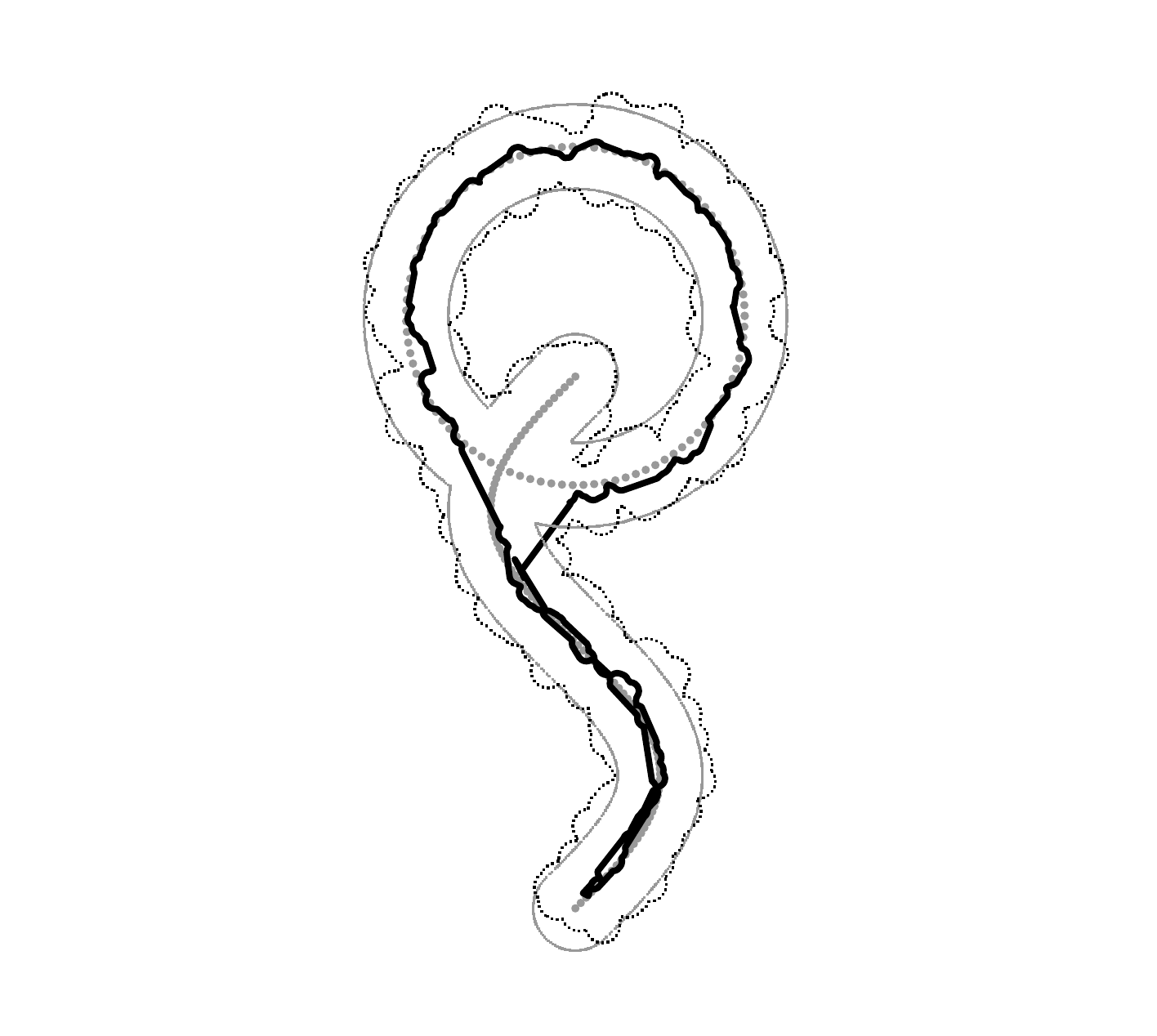} \\
\end{tabular}
\caption{Second example.
Top line: true curves and the support of the distribution (left), the data (right).
Center line: points identified as clutter (left), decluttered data (right).
Bottom line: EDT estimator (left), Medial estimator (right).
}
\label{fig::example2}
\end{figure}

The third dataset is more challenging as it contains 12 filaments, with several 
intersections. Eighty points were generated from each filament and 350 more points 
were generated as background clutter, for a total of $n=1310$ data points (top panels 
in Figure \ref{fig::example3}). The decluttering procedure (central panels in Figure
\ref{fig::example3}) resulted in 989 points marked as filament (34 of which were 
generated as clutter) and 321 points marked as clutter (5 of which were filament points). 
The estimates, obtained from the points marked as filament, are shown in the bottom panel 
of Figure \ref{fig::example3}. These estimates are accurate for filaments with no 
intersections. The accuracy is less satisfactory for intersecting filaments or for filaments 
that are too close to each other. This was to be expected, as the condition on the radius 
of curvature is not satisfied in these cases.

\begin{figure}
\begin{tabular}{cc}
\includegraphics[width=2.5in]{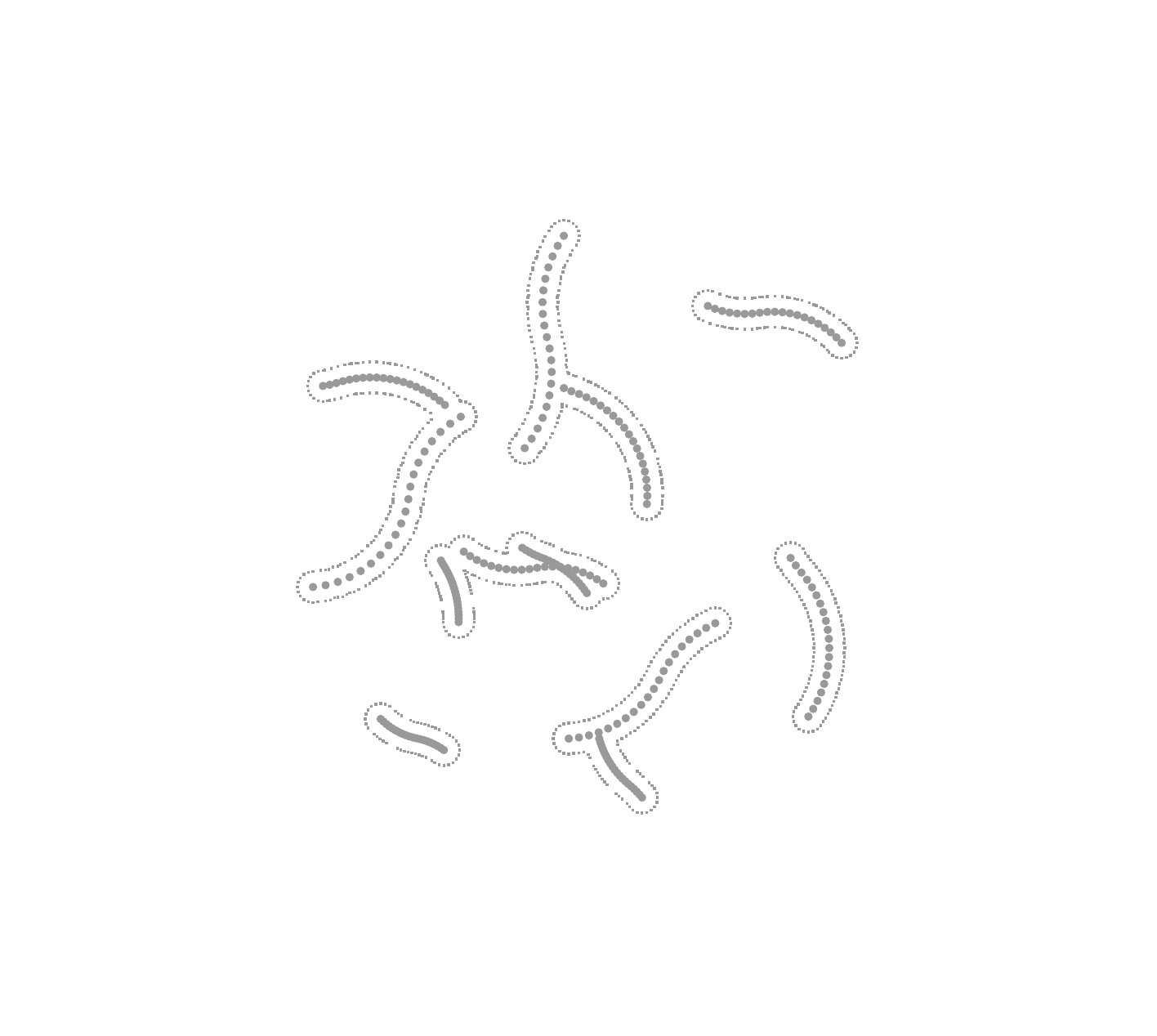} & \includegraphics[width=2.5in]{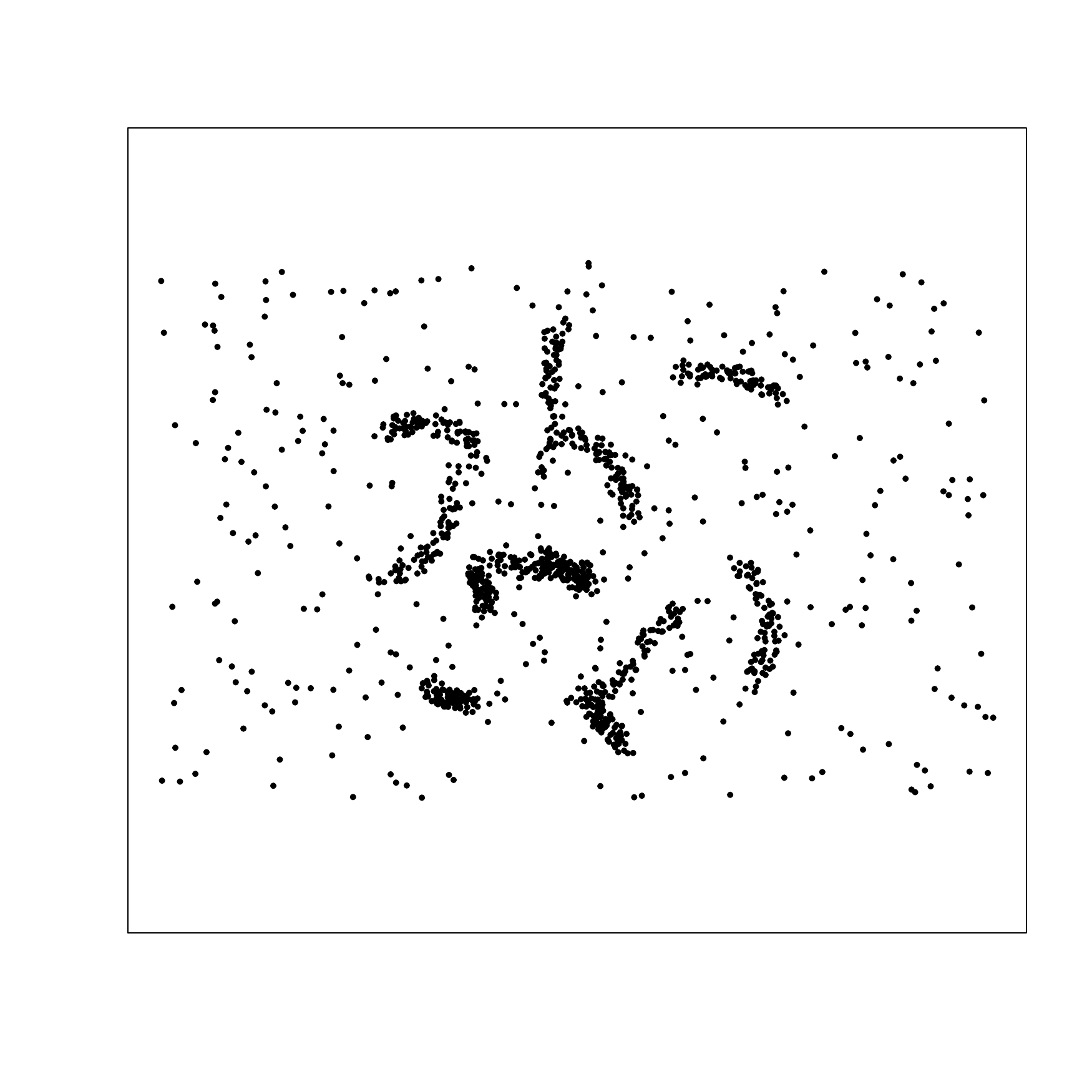} \\
\includegraphics[width=2.5in]{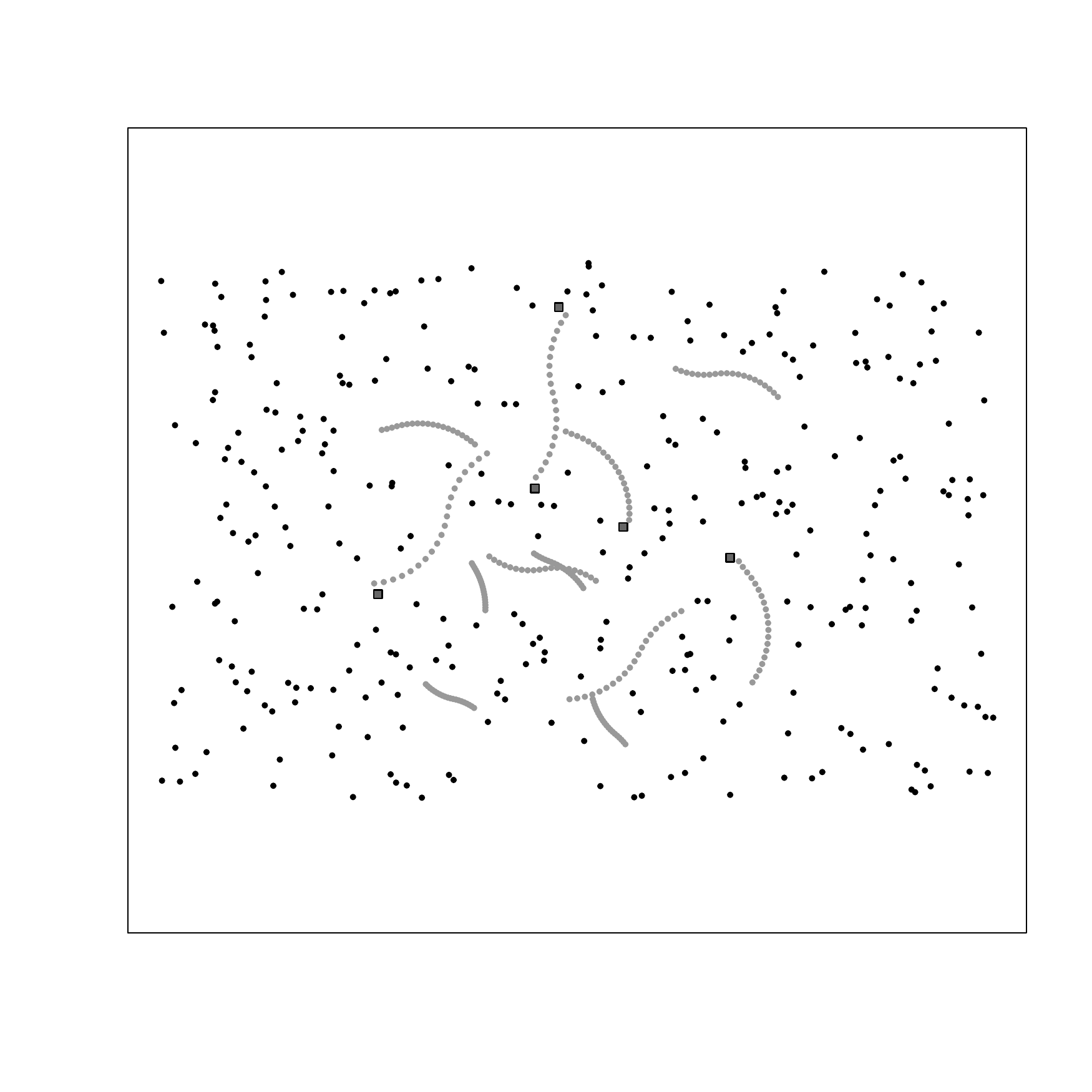} & \includegraphics[width=2.5in]{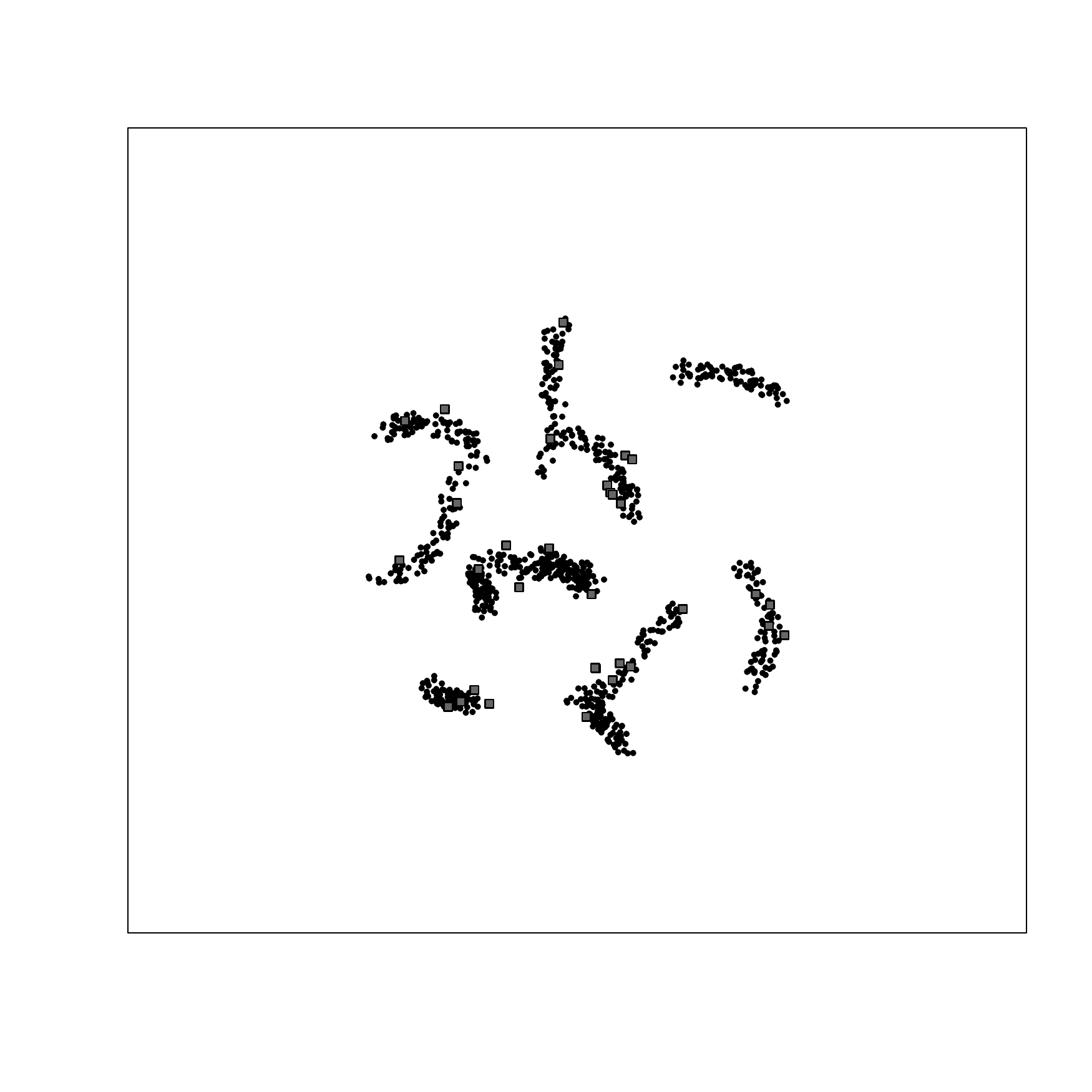} \\
\includegraphics[width=2.5in]{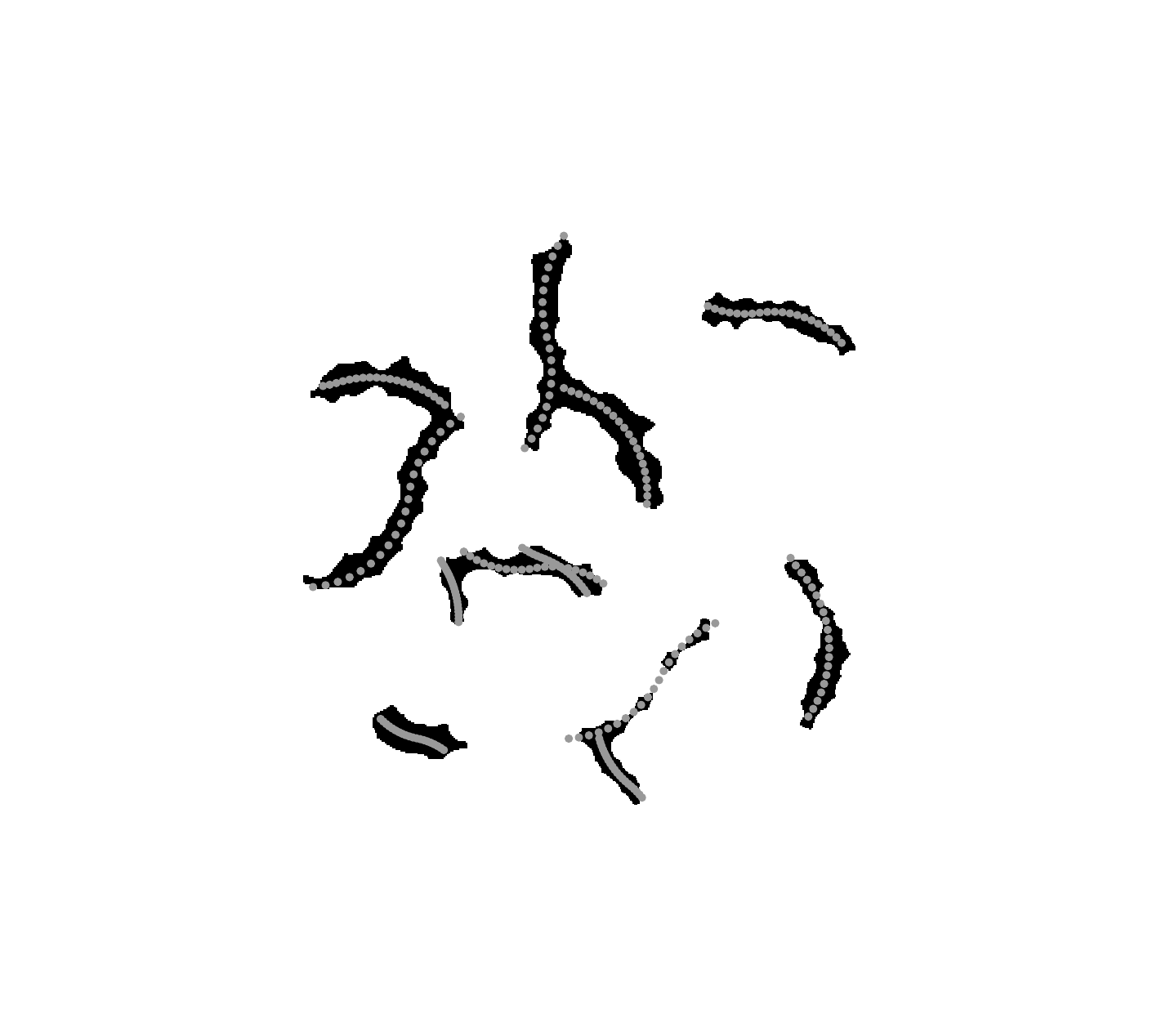} & \includegraphics[width=2.5in]{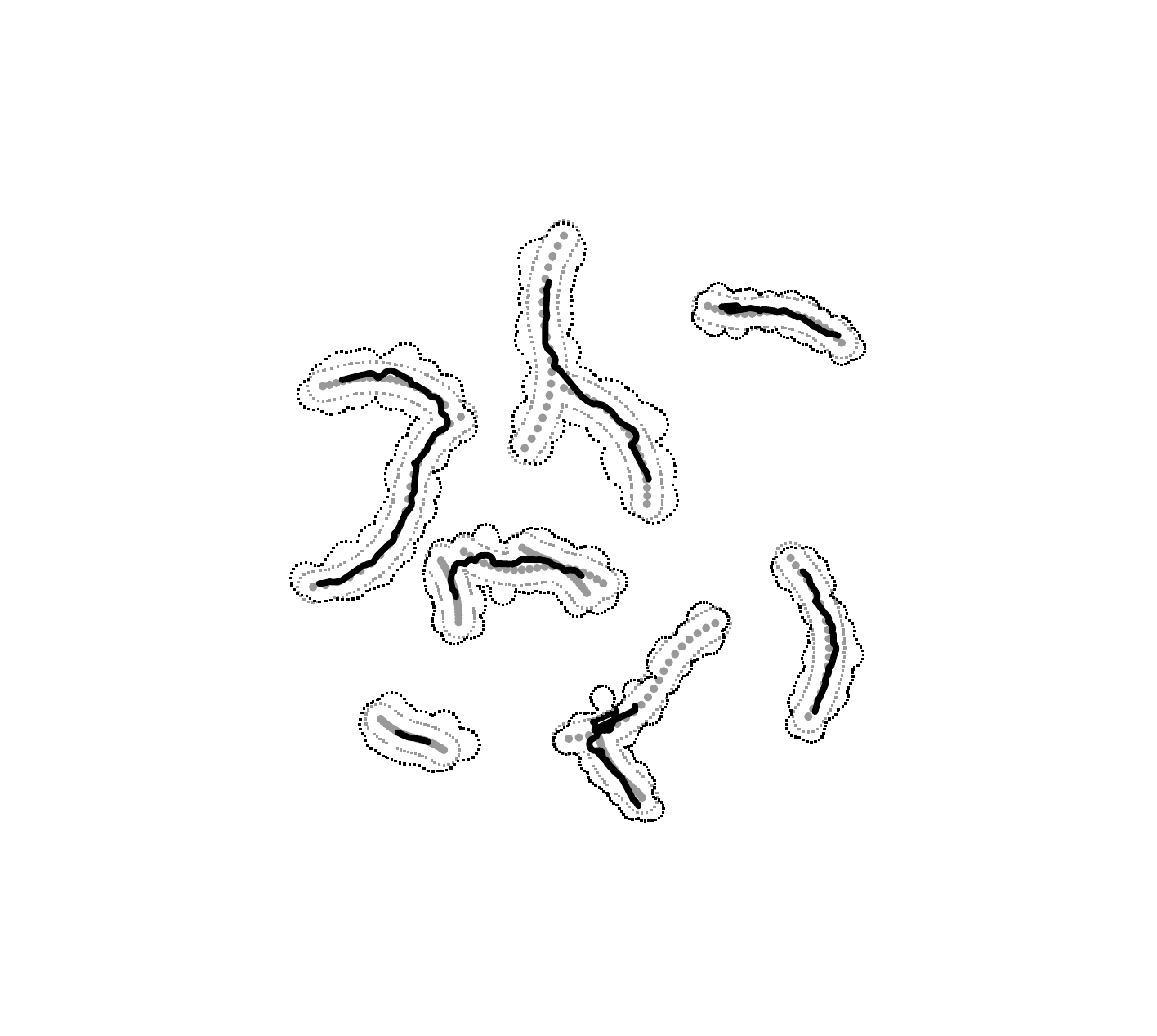} \\
\end{tabular}
\caption{Third example.
Top line: true curves and the support of the distribution (left), the data (right).
Canter line: points identified as clutter (left), decluttered data (right).
Bottom line: EDT estimator (left), Medial estimator (right).
}
\label{fig::example3}
\end{figure}

\section{Discussion}

In recent work (\cite{us}) we found the minimax rate for this problem
under restrictive conditions (but in general dimensions).
In current work,
we are finding the minimax rates in general.
This is a difficult problem because the rate
depends critically on features of the noise distribution $F$.
Moreover, the problem is essentially a deconvolution problem
since the variables $\xi_i = f(U_i)$ are unobserved
and corrupted by noise.
We will report on these results elsewhere.

The estimators presented here are not minimax but are appealing because of their simplicity.
Finding a practical estimator that achieves the minimax rate is an open question.
Our approach, instead, consists of two steps:
producing a set of fitted values $\hat\Gamma$ and then
extracting a curve from $\hat\Gamma$.
We gave two specific methods for obtaining the fitted values
and a curve extraction method for each of the two approaches.
The resulting estimators have reasonably fast rates of convergence.

The noise model is critical.
We assumed compact support which is reasonable for many applications.
Without compact support, the behavior of the methods changes substantially
as it does in nonparametric measurement error problems.

It is interesting to compare our results to those in
\cite{Cheng:2005}.
They show that each of their fitted values
is $O_P((\log n/n)^{1/8})$ from the filament.
Under weaker conditions than they assumed, we get a rate 
which is faster as long as $\alpha$ is not too large.
(They implicitly assume that $\alpha=0$.)
Also, our rate is in Hausdorff distance which is a stronger notion of closeness
than used in their paper.

Currently, we are pursing several extensions of our results.
These include:
the aforemetioned extensions to higher dimensions (manifold learning),
relaxing the smoothness condition,
relaxing the constant $\sigma$ condition,
noise distributions with non-compact support
and comparisons with beamlets.
We are also investigating data-driven methods for choosing 
the tuning parameter $\epsilon$ and
we are studying the
theoretical properties of the decluttering technique.

\newpage

\section{Supplementary Material}

\subsection{Proofs}
\label{sec::proofs}

\smallskip
\begin{proofof} {\bf Lemma \ref{lemma::new-boundary}.}
For $y \in B(0,\sigma)$ the density $\phi$ satisfies
\begin{equation} \label{eqn:tails}
\phi(y) \geq C_1 \cdot d(y, \partial B(0,\sigma))^\beta =  
C_2 \cdot \left[1- \frac{||y||}\sigma \right]^\beta.
\end{equation}
Note also that, monotonicity of $\phi$ implies
that $\phi(y) \leq \phi(0)$.

Let $d=d(y,\partial S)$ and let $y_0\in \partial S$
be the point on $\partial S$ closest to $y=(y_1,y_2)$.
Without loss of generality, assume that
$y_0 = (0,0)^T$ and that the tangent vector to $\partial S$ at $y_0$ is
$(1,0)^T$.
We now prove that 
$q(y_2|y_1) \geq C_4 d(y, \partial S)^{\beta+1/2}$.
In Lemma \ref{lemma::disjoint} we show that 
$S =  \Union_{0 \le u \le 1} \ L(u)$,
where $L(u)$ is defined in (\ref{eq::tube}) as
$L(u) = \bigl\{ f(u) + t N(u):\ -\sigma \leq t \leq \sigma \bigr\}$
and
$N(u)$ is the normal vector at $f(u)$.
Moreover, we show that the $L(u)$'s are disjoint.
Let $\bar u \in [0,1]$ such that $y \in L(\bar u)$, hence
$||y-f(\bar u)|| = \sigma-d$.
Continuity of $f$ implies that there exists an interval $(u', u'') \subset [0,1]$ such that
$||y-f(u)|| \leq \sigma-d/2$ for all $u$ in the interval.
We will show later that $|u' - u''| \geq C_5 \cdot \sqrt d$. 
We can write the density at $y=(y_1, y_2)$ as
$q(y) = \int_{\{ u: ||y-f(u)|| \leq \sigma  \}} \phi(y-f(u)) h(u) \; du$
and the conditional density
\begin{eqnarray*}
q(y_2|y_1) =
\frac{\int_{\{ u: ||y-f(u)|| \leq \sigma  \}} \phi(y-f(u)) h(u) \; du}
{\int_{\{ y_2: (y_1,y_2) \in S \}} \int_{\{ u: ||y-f(u)|| \leq \sigma  \}} \phi(y-f(u)) h(u) \; du \; dy_2 }.
\end{eqnarray*}
The denominator is bounded from above by $\phi(0) C_6$.
Hence,
\begin{eqnarray*}
q(y_2|y_1) 
& \geq &
C_7 \int_{\{ u: ||y-f(u)|| \leq \sigma  \}} \phi(y-f(u)) h(u) \; du \geq
C_7 \int_{u'}^{u''} \phi(y-f(u)) h(u) \; du \\
&\geq &
C_8 \int_{u'}^{u''} \left[1- \frac{||y-f(u)||}\sigma \right]^\beta h(u) \; du
\geq 
C_8 \int_{u'}^{u''} \left[1- \frac{\sigma - d/2}\sigma \right]^\beta h(u) \; du
\\
&\geq &
C_8 \left[\frac d{2\sigma}\right]^\beta C \cdot |u'' - u'|
\geq 
C_9 \cdot d^{\beta+1/2}.
\end{eqnarray*}
Now we show that $|u'-u''| \geq C_5 \sqrt d$.
Let $z=d/2$, $||f(u') - f(u'')||$ is bounded below by the distance of the 
intersection of the two balls $B((0,0), \Delta)$
and $B((\Delta + \sigma-2z,0), \sigma-z)$.
Some algebra shows that
$$
||f(u') - f(u'')|| \geq 2 \sqrt{ z } \sqrt{\frac{2\Delta\sigma}{\Delta+\sigma}}
=
2 \sqrt{ d } \sqrt{\frac {\Delta\sigma}{\Delta+\sigma}}
$$
Finally, since
$$
||f(u') - f(u'') || = \int_{u'}^{u''} \nabla f(u) \; du \leq \sup_{u \in [0,1]} \nabla f(u) |u'-u''|
$$
we obtain
$$
|u'-u''| \geq \frac 1{\sup_{u \in [0,1]} \nabla f(u)} \, 2\, \sqrt{d} \,
\sqrt{\frac {\Delta\sigma}{\Delta+\sigma}} = C_5 \sqrt{d}.
$$
It is easy to see that
$q(y_1)\geq c >0$ for all $y\in B(y_0,\epsilon)$.
Hence,
$q(y) \geq c_2 d^{\beta+1/2}$.

Now we find the upper bound.
Let $d = d(y,\partial S)$.
Let $u_0$ be such that $y\in L(u_0)$.
Now
\begin{eqnarray*}
q(y) &=& \int_{u'}^{u''}\phi(y-f(u)) h(u) du \leq
c_2 \int_{u'}^{u''}\phi(y-f(u)) du\\
& \leq &
c_2 \phi(y-f(u_0))\,| u'' - u'| \leq c_2 C d^\beta \,| u'' - u'|
\end{eqnarray*}
Earlier we showed that
$| u'' - u'|\geq C_5 \sqrt{d}$.
By a similar argument,
$| u'' - u'|\leq c_5 \sqrt{d}$
for some $c_5$.
The result follows.
\end{proofof}

\smallskip
\begin{proofof} {\bf Lemma \ref{lemma::disjoint}.}

{\em 1.} First, consider the closed case.
We show that $S = \cT$.
Suppose not.
Then there is a $y\in S$ such that
$$
y \neq f(u) + t N(u)
$$
for any $u\in[0,1]$ and
$t\in[-\sigma,\sigma]$.
Let $f(u)$ be the closest point on the curve to $y$.
Since $y\notin L(u)$,
$\langle y-f(u),T(u) \rangle \neq 0$.
Without loss of generality, suppose that $\langle y-f(u),T(u)\rangle >0$.
So, for sufficiently small $\epsilon$,
\begin{align*}
||y-f(u)||^2 
    &< ||y - f(u+\epsilon)||^2 = ||y - f(u) - \epsilon T(u)||^2 + o(\epsilon^2) \\
    &= ||y-f(u)||^2 + \epsilon^2 - 2\epsilon \langle y-f(u),T(u)\rangle + o(\epsilon^2)\\
    &< ||y-f(u)||^2,
\end{align*}
which is a contradiction.
For the open case, the balls $B(f(0),\sigma)$ and $B(f(1),\sigma)$ do not
intersect. Both balls are contained in $S$. 
The half plane formed by the normal vectors at $f(0)$ and $f(1)$ split 
these balls in two, with half of each in $\cT$ and half in $\cT^c$.
The result follows.

{\em 2.} Now we show that $u\neq v\in[0,1]$ implies that $L(u)\cap L(v) = \emptyset$.
Suppose that $L(u)$ and $L(v)$ intersect at some point $y$. So
$$
y = f(u) + s N(u) = f(v) + t N(v)
$$
for some $s,t\in[-\sigma,\sigma]$.
Let $A_u$ be the ball of radius $\Delta$ tangent to $f(u)$ and containing $y$.
Let $A_v$ be the ball of radius $\Delta$ tangent to $f(v)$ and containing $y$.
Note that $f(v)\notin A_u$ and $f(u)\notin A_v$.
(This follows from the discussion after (\ref{eq::Delta}).)
Now $f(v)\notin A_u$ implies that $t\geq s$ but
$f(u)\notin A_v$ implies that $s\geq t$ and so $s=t$.
So, $y= f(u) + s N(u) = f(v) + s N(v)$. By the triangle inequality,
\begin{align*}
d(f(v),{\rm center}(A_u)) 
    &=   ||f(u) + \Delta N(u) - f(v)||\\
    &\le ||f(u) + \Delta N(u) - y|| + ||y -f(v)||\\
    &=   ||(\Delta-s) N(u)|| + ||s N(v)||\\
    &=   \Delta.
\end{align*}
But $f(u)\notin A_v$ means that  $d(f(v),{\rm center}(A_u)) \geq \Delta$.
So the inequality above must be equality which implies that
$f(u) + \Delta N(u)$, $y$ and $f(v)$ fall on a line.
Hence, $L(u)$ and $L(v)$ cannot intersect.

{\em 3.} Follows directly from {\em 2.} because $\cT$ is the union of the fibers.

{\em 4.} Follows from {\em 1.}~and {\em 5.} follows from {\em 4}.
\end{proofof}

\medskip 

\begin{proofof} {\bf Theorem \ref{thm::filamentismedial}.}

{\em 1.} First we show that $\Gamma_f \subset M(S)$. Pick any $u\in[0,1]$. 
Let $B = B(f(u),\sigma)$. We claim that $B\cap \partial S$ contains 
at least two points. Let $a = f(u) + \sigma N(u)$ and $b = f(u) - \sigma N(u)$.
We will show that $a$ and $b$ are in $B\cap \partial S$.

Note that $a,b \in B$ and hence they are in $S$.
In fact they are boundary points because they are not in the interior of 
$S$. To show this, suppose to the contrary that $a$ is interior. Hence there 
exists $v$ such that $||a-f(v)|| =\delta < \sigma$. That is, $f(v)$ is in the 
interior of $B(a,\sigma)$. But this contradicts the assumption.
The same argument shows that
$b\in \partial S$. Hence, $f(u)\in M(S)$ and so $\Gamma_f \subset M(S)$. 

Now we show that $M(S)\subset \Gamma_f$. Let $y\in M(S)$. We claim that 
$y=f(u)$ for some $u$. Suppose not. From Lemma \ref{lemma::disjoint},
$y\in L(u)$ for some $u$ and $y\notin L(v)$ for any $v\neq u$.
Also, $f(u)\in M(S)$ and $B(f(u),\sigma)\cap \partial S$ contains
$a = f(u) + \sigma N(u)$ and  $b = f(u) - \sigma N(u)$.
Since $y\in L(u)$ and $y\neq f(u)$ either $||y-a|| < \sigma$ or 
$||y-b|| < \sigma$. Without loss of generality, assume that
$||y-a|| < \sigma$.
Set
$r = ||a-y||$ and $s = ||y-f(u)||$ and note that $r+s = \sigma$.
Let $B = B(y,\delta)$ be the medial ball at $y$. If $\delta > r$ then
the interior of $B(y,\delta)$ has nonempty intersection with $\partial S$.
So $\delta$ must be less than or equal to $r$. On the other hand, if 
$\delta < r$ then $B(y,\delta)\cap \partial S =\emptyset$. So we must 
have $\delta = r$. But $B(y,r)$ is stricty contained in $B(f(u),\sigma)$
except for the common point $a$. Thus, all points in $B(y,r)$ are 
interior points of $S$ except for $a$. So $B(y,r)\cap \partial S$ 
contains fewer that 2 points and hence $y$ cannot be in $M(S)$.

\smallskip

{\em 2.} The proof that
$\Gamma_f \subset M(S)$ is the same as in part {\em 1.}
Now suppose that
$||f(1)-f(0)|| > 2\sigma$.
We will show that $M(S)\subset \Gamma_f$ and hence
$\Gamma_f = M(S)$.
From Equations (\ref{eq::tube}) and (\ref{eq::endcaps}), recall that 
$\cT$, $\cC_0$, and $\cC_1$ denote the tube and the end caps.
By Lemma \ref{lemma::disjoint},
$S = \cT\union \cC_0 \union \cC_1$,
where $\cC_0\cap \cC_1 = \emptyset$.
Let $y\in M(S)$.
If $y\in \cT$ then the proof of the previous part implies that $y=f(u)$ for some $u$.
That is,
$M(S)\cap \cT\subset \Gamma_f$.
Now suppose $y\in \cC_0$.
Then
$d(y,\partial S) = r < \sigma$ for some $r$.
We may assume $r>0$ otherwise $y$ is on the boundary and cannot be medial.
Consider a ball $B(y,\delta)$.
We claim that $B(y,\delta)$ cannot be medial.
In fact, if $\delta < r$ then all points in $B(y,\delta)$ are interior to $S$.
If $\delta = r$ then $B(y,\delta)$ intersects $\partial S$ at a single point.
Finally, if $\delta > r$ then ${\rm interior}(B(y,\delta)) \cap \partial S \neq \emptyset$.
Thus $B(y,\delta)$ cannot be medial and $\cC_0 \cap M(S) = \emptyset$. Similarly,
$\cC_1 \cap M(S) = \emptyset$.
Hence, $M(S) = \Gamma_f$.
\end{proofof}

\medskip
\begin{proofof} {\bf Lemma \ref{lemma::edt}.}
We prove the closed case.
The open case is similar.

{\em 1.} If $y\in M(S)$ then $y=f(u)$ for some $u$
by Theorem \ref{thm::filamentismedial}.
From lemma \ref{lemma::disjoint},
the closest point on the boundary is either
$f(u) + \sigma N(u)$ or 
$f(u) - \sigma N(u)$.
In either case,
$d(y,\partial S) = \sigma$.

{\em 2.} We have $y = f(u) + t N(u)$ for some $u$ and $t$.
Since $y\notin M(S)$, it follows that $y\neq f(u)$ and so $t\neq 0$.
Then, from (1), $d(y,\partial S) < \sigma$.

{\em 3.} We have $y = f(u) + t N(u)$ for some $u$ and some
$t\in[-\sigma,\sigma]$.
The closest boundary point is either
$f(u) + \sigma N(u)$ or 
$f(u) - \sigma N(u)$.
Without loss of generality, assume it is
$f(u) + \sigma N(u)$.
Hence, $t \geq 0$.
So, $\sigma = d(f(u),\partial S) = 
d(f(u),y) + d(y,\partial S) = d(y,M(S)) + \Lambda(y)$.
\end{proofof}

\smallskip
\begin{proofof} {\bf Theorem \ref{thm::edt}.}

{\em 1.} Choose any $y\in \mathbb{R}^2$.
Let $z_*$ be the closest point to $y$ on $\partial S$.
Let $\hat z$ be the closest point to $y$ on $\hat{\partial S}$.
Let $\tilde{z}$ be the closest point to $z_*$ on $\hat{\partial S}$.
Then
\begin{align*}
\hat\Lambda(y) 
   &=   ||y-\hat z||  \leq  ||y - \tilde{z}|| \le  ||y-z_*|| + ||z_* - \tilde{z}||\\
   &\le ||y-z_*|| + \epsilon = \Lambda(y) + \epsilon.
\end{align*}
Now let $\bar{z}$
be the point on $\partial S$ closest to
$\hat z$.
Then
\begin{align*}
\Lambda(y) 
   &= ||y-z_*||  \leq  ||y - \bar{z}|| 
   \le ||y-\hat z|| + ||\hat z - \bar{z}||\\
   &\le \hat\Lambda(y) + \epsilon.
\end{align*}

{\em 2.} Let $\hat y = \argmax_{y \in \hat S} \hat \Lambda(y)$ and let $y_*$ be its  closest point 
in $M(S)$. Then
$$
\hat\sigma = \hat\Lambda(\hat y) \leq \Lambda(\hat y) + \epsilon \leq
\Lambda(y_*) + \epsilon = \sigma + \epsilon.
$$
Also
$$
\hat\sigma = \hat\Lambda(\hat y) \geq
\hat\Lambda(y_*) \geq \Lambda(y_*) - \epsilon = \sigma - \epsilon.
$$

{\em 3.} By Lemma \ref{lemma::disjoint}, there is a unique $u$ such that
$\hat y$ is on the fiber $L(u)$, centered at $f(u)$. So
\begin{eqnarray*}
\sigma = d(f(u),\partial S) &=& ||f(u)-\hat y|| + d(\hat y,\partial S) \geq
||f(u)-\hat y|| + d(\hat y,\hat{\partial S})-\epsilon\\
&=& ||f(u)-\hat y|| + \hat\sigma - \epsilon \geq
||f(u)-\hat y|| + \sigma -2\epsilon.
\end{eqnarray*}
Hence,
$d(\hat y, M(S)) \leq d(\hat y,f(u)) \leq 2\epsilon$.
\end{proofof}

\medskip
\begin{proofof} {\bf Lemma \ref{lemma::standard}.}
Let $y$ be a point in $S$ and let $\Lambda(y) \leq \sigma$ be its distance 
from the boundary $\partial S$. If $\Lambda(y) \geq \epsilon$ then
$B(y,\epsilon)\cap S =B(y,\epsilon)$ so that
$\nu(B(y,\epsilon)\cap S) =\nu (B(y,\epsilon)) = \pi \epsilon^2\geq
\chi \nu(B(y,\epsilon))$.

Suppose that $\Lambda(y) < \epsilon$. Let $f(u)$ be the point on the filament 
closest to $y$ and let $y^*$ be the point on the segment joining $y$ to $f(u)$ such 
that $||y-y^*||=\epsilon/2$. The ball $A=B(y^*,\epsilon/2)$ is contained in both 
$B(y,\epsilon)$ and $S$. Hence,
$\nu(B(y,\epsilon)\cap S) \geq \nu (A) = \pi \epsilon^2/4 = \chi \nu(B(y,\epsilon))$.
This is true for all $\epsilon \leq \sigma$, hence $S$ is $(\chi, \lambda)$-standard for
$\chi = 1/4$ and $\lambda=\sigma$.

Now we show that $S$ is expandable. By Proposition 1 in \cite{cuevas-boundary}
it suffices to show that a ball of radius $r$ rolls freely outside $S$ for some $r$,
meaning that, for each $y\in \partial S$, there is an $a$ such that
$y\in B(a,r)\subset \overline{S^c}$, where $S^c$ is the complement of $S$.
Let $O_y$ be the ball of radius $\Delta-\sigma$ tangent to $y$ such that
$O_y\subset S^c$. Such a ball exists by virtue of the conditions on $\sigma$.
\end{proofof}

\smallskip
\begin{proofof} {\bf Lemma \ref{lemma::nice-boundary}.} 
Suppose first that 
$f$ is open. Then $\partial S$ is a closed, simple curve. Let $A_n = S \union \hat{S}$.
We will first show that $\partial A_n$ is a closed, non-self-intersecting curve for all $n$.
Consider one observation $Y_1$ and note that $A_1 = S \union B(Y_1,\epsilon_n)$. Since $Y_1\in S$, 
${\rm interior}(B(Y_1,\epsilon_n))\cap S \neq \emptyset$.
It is then easy to see that $\partial A_1$ is a closed, non-self-intersecting curve.
A simple induction argument verifies that $\partial A_n$ is a closed, non-self-intersecting 
curve for all $n$. Now, when $S\subset\hat{S}$, we have $A_n = \hat S$ and the conclusion 
follows. The proof for closed curves is similar.
\end{proofof}

\smallskip
\begin{proofof} {\bf Theorem \ref{thm::edtestimator}.}

{\em 1.} First we show that $y\in\hat\Gamma$ implies that $d(y,M) \leq 4\epsilon$.
Let $y\in\hat\Gamma$. Then
$d(y,\partial S)  \geq  
d(y,\hat{\partial S}) -\epsilon \geq \hat\sigma - 2\epsilon-\epsilon \geq
\sigma - \epsilon - 2\epsilon - \epsilon = \sigma - 4\epsilon$.
So
$d(y,M) = \sigma - d(y,\partial S) \leq 
\sigma - \sigma + 4\epsilon  = 4 \epsilon$.
Now we show that $M\subset\hat\Gamma$. Suppose that $y\in M$. Then,
$$
d(y,\hat{\partial S})  \geq  
d(y,\partial S) -\epsilon =\sigma - \epsilon \geq
\hat\sigma - 2\epsilon = \hat\sigma - \delta
$$
so that $y\in\hat\Gamma$.

{\em 2.} The proof of the second statement follows from {\em 1.}~and Lemma~\ref{lemma::supp}.
\end{proofof}

\smallskip 

The next two {\bf Lemmas \ref{thm::key}} and {\bf \ref {thm::key2}} are needed to prove 
Theorem \ref{thm::methodII}.

\begin{lemma}
\label{thm::key}
Let $\Gamma$ and $\Gamma_1$ be two curves in $\mathbb{R}^2$ such that 
$d_H(\Gamma,\Gamma_1) \leq \epsilon$. Given a point $a \in \mathbb{R}^2$, let $a^*$ be the 
point on $\Gamma$ closest to $a$. Let $d = ||a - a^*||$.
Consider a ball, with radius $r> d+\epsilon$ and center $\alpha$, that contains $a$, $a^*$ 
and no other points in $\Gamma$. Then there exists a point $\hat a \in \Gamma_1$ 
such that $||a - \hat a|| = d(a, \Gamma_1)$ and 
\begin{equation}
||a^* - \hat a||^2 \leq \frac{4rd}{r-d} \epsilon + \epsilon^2.
\end{equation}
Thus,
$||a^* - \hat a|| = O(\sqrt{\epsilon})$.
\end{lemma}

\begin{proofof} {\bf Lemma \ref{thm::key}.} See Figure \ref{fig::marcoplot}. 
Consider the ball $A=B(\alpha, r)$ with $r> d+\epsilon$. Let $a$ be a point 
along the radius that joins $\alpha$ with $a^*$. Since 
$d_H(\Gamma, \Gamma_1) \leq \epsilon$, there exists a point 
$g \in \Gamma$ within $\epsilon$ distance from $\hat a$, other than $a^*$. 

\begin{figure}
\vspace{-.3in}
\begin{center}
\includegraphics[width=3.5in]{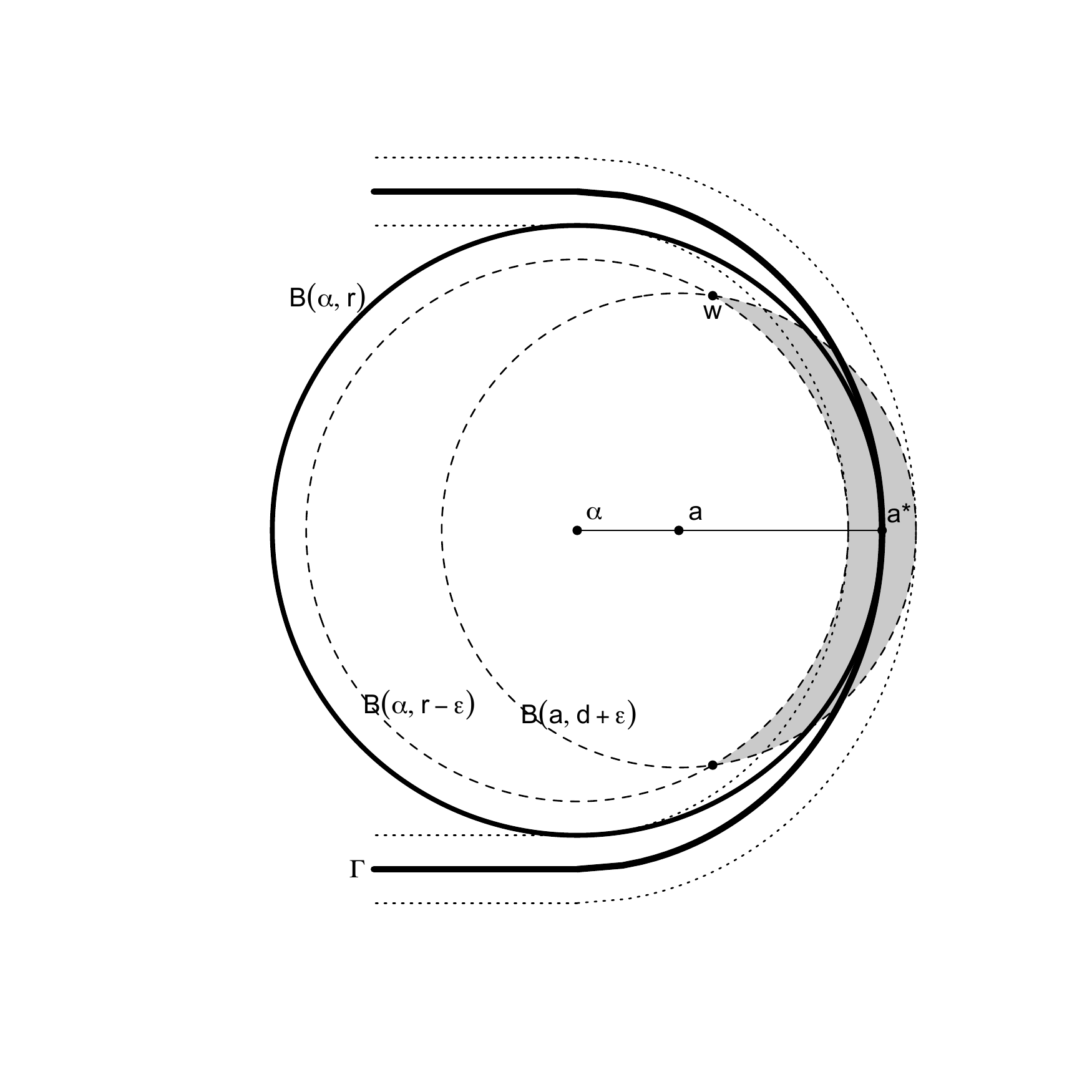}
\end{center}
\vspace{-.2in}
\caption{Diagram for proof of Lemma \ref{thm::key}.}
\label{fig::marcoplot}
\end{figure}

\noindent Note that $\hat a \notin B(\alpha, r-\epsilon)$, otherwise $g$ would be in $B(\alpha, r)$, 
but, by construction, $a^*$ is the only point that belongs to $B(\alpha,r)\cap\Gamma$.
To show that $\hat a \notin B(\alpha, r-\epsilon)$, assume by contradiction that 
$\hat a$ were in $B(\alpha, r-\epsilon)$, then $||\alpha - \hat a|| \leq r-\epsilon$ and
$$
||\alpha - g|| \leq ||\alpha - \hat a|| + || \hat a - g|| \leq r-\epsilon + \epsilon,
$$
thus implying that $g \in B(\alpha,r)$. 

Now, since $||a-\hat a|| \leq d+\epsilon$, then $ \hat a \in B(a, d+\epsilon) 
\cap B(\alpha, r-\epsilon)^c$, the shaded region in Figure \ref{fig::marcoplot}.
Thus, $||\hat a - a^*|| \leq ||w-a^*||$, where $w$ is either one of the two points 
where the two balls $B(\alpha, r-\epsilon)$ and $B(a, d+\epsilon)$ 
cross in Figure~\ref{fig::marcoplot}.

\noindent Without loss of generality assume the system of coordinates is such that: 
$$
\alpha \equiv (0,0) \qquad a\equiv(r-d, 0) \qquad a^*\equiv(r,0),
$$
and the equation of the two balls are:
\begin{eqnarray*}
B(a, d+\epsilon) &:& \left(x-(r-d)\right)^2 + y^2 = (d+\epsilon)^2 \\
B(\alpha, r-\epsilon) &:& x^2 + y^2 = (r-\epsilon)^2.
\end{eqnarray*}
Thus the coordinates of $w$ are
$$
\left ( x_w = r- \frac{r+d}{r-d}\ \epsilon,\  \ 
y_w= \pm \sqrt{(r-\epsilon)^2 - x_w^2}\right),
$$
and the distance between $\hat{a}$ and $a^*$
\begin{eqnarray*}
||\hat a - a^*||^2 &\leq & ||w-a^*||^2 = (x_w-r)^2 + y_w^2 = 
\left(\frac{4rd}{r-d}\right) \epsilon + \epsilon^2.
\end{eqnarray*}
\end{proofof}

\medskip

\begin{lemma}
\label{thm::key2}
Suppose $\epsilon < (\Delta-\sigma)/2$. Let
${\cal Y}(u) = \{ y = f(u)+tN(u):\ 
{\rm\ for\  some\ } u \in [0,1]\ {\rm\  and\ } |t| \leq \sigma + \epsilon\}$ 
be the extended fiber.
It can be shown that the extended fibers
${\cal Y}(u) =\{ f(u) + t N(u):\ -\Delta \leq t \leq \Delta\}$ are disjoint.
For a given $u$ let $y\in{\cal Y}(u)$.
Let $y^*_i$ be the point of $\partial S_i$ ($i=0,1$) closest to $y$.
There exists $\hat y_i \in \hat{\partial S}_i$ 
such that $||y - \hat y_i|| = d(y, \hat{\partial S}_i)$ and 
\begin{equation}
||y^*_i - \hat y_i||^2 \leq \frac{16 \Delta^2}{\beta} \epsilon + \epsilon^2,
\end{equation}
with $0 <\beta < (\Delta-\sigma)/2-\epsilon$.
Hence $||y^*_i - \hat y_i|| =O( \sqrt{\epsilon})$ uniformly over ${\cal Y}(u)$.
\end{lemma}

\begin{proofof} {\bf Lemma \ref{thm::key2}.}
Let $r = (\sigma+\Delta)/2$ and note that $\sigma + \epsilon < r < \Delta$.
Consider two balls of radius $r$ tangent to the filament at $f(u)$ on either side of 
$\Gamma_f$. Both balls contain no points of $\Gamma_f$ other than $f(u)$.
Let $\alpha_i$ be the center of the ball on the side opposite to $y^*_i$, so that
$\alpha_i$ is on the normal through $f(u)$.

Now we show that the balls $B(\alpha_i, r+\sigma)$, $i = 0,1$ centered in $\alpha_i$ satisfy the 
conditions required in Lemma \ref{thm::key}. 
By construction, $B(\alpha_i, r+\sigma)$ is tangent to $\partial S_i$ at $y_i^*$ and 
$y\in B(\alpha_i, r+\sigma)$. The center $\alpha_i$ of $B(\alpha_i,r+\sigma)$ is on the normal 
through $f(u)$, thus $y_i^*$ is the closest point to $\alpha_i$ on the boundary 
$\partial S_i$ and there are no other points in $\partial S_i$ interior to the ball.
Also $\partial S_i$ cannot be tangent to the ball in a point $z\neq y^*_i$, otherwise $\alpha_i$ 
would be on the extended fiber ${\cal Y}(u')$ for some $u' \neq u$.
But
the extended fibers
${\cal Y}(u) =\{ f(u) + t N(u):\ -\Delta \leq t \leq \Delta\}$ are disjoint.
This shows that 
$B(\alpha_i, r+\sigma)$ is the ball $A$ of Lemma \ref{thm::key}, with 
$\alpha = \alpha_i$, $y=a$ and $y^* = a^*$. Hence, from Lemma \ref{thm::key}:
$$
||y^*_i - \hat y_i||^2 \leq \frac{4(r+\sigma)d}{r+\sigma-d} \epsilon + \epsilon^2
$$
where $d= ||y - y^*_i|| \leq 2\sigma + \epsilon$.
The result follows since $r+\sigma < 2\Delta$, $d \leq 2\sigma + \epsilon < 2\Delta$ and 
$r+\sigma-d \geq (\Delta-\sigma)/2-\epsilon > \beta$. \end{proofof}

\begin{figure}[h]
\begin{center}
\includegraphics[width=3.5in]{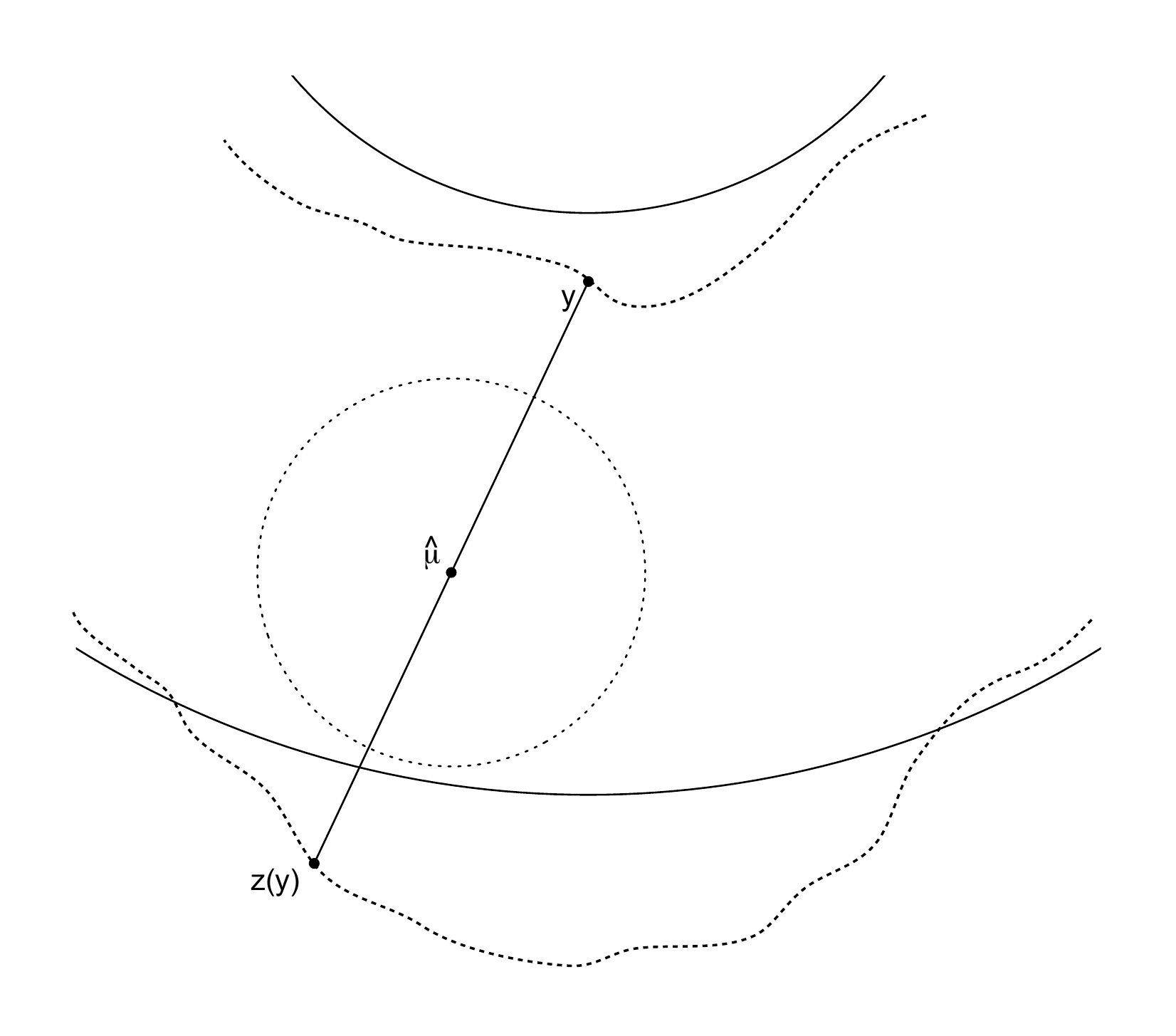}
\end{center}
\vspace{-.3in}
\caption{First illustration for the proof of Theorem \ref{thm::methodII}.}
\label{fig::plot1}
\end{figure}

\begin{proofof} {\bf Theorem \ref{thm::methodII}.}
See Figures \ref{fig::plot1} and \ref{fig::plot2}.

\smallskip
{\em 1.} Recall that $\hat \sigma$ $=$ $\max_{y \in \hat S} d(y,\hat{\partial S})$ 
$<$ $\sigma + \epsilon$ 
and that for each 
$f(u) \in \Gamma_f$ we have $d(f(u), \partial S_0)=d(f(u), \partial S_1) = \sigma$.

{\em (i)} Let $\hat{\mu} \in \hat\Gamma$, and let $y \in \hat{\partial S_0}$ and 
$z(y) \in \hat{\partial S_1}$ be the points that generated it, as in 
Figure \ref{fig::plot1}. Let $\ell(y,z(y))$  be the line segment that joins $y$ to $z(y)$.
The distance between any $x \in \ell(y, z(y))$ and the boundary curves is, respectively,
$d(x,\hat {\partial S_1})= ||x-z(y)||$ and $d(x, \hat {\partial S_0}) \leq ||x-y||$.  
The midpoint $\hat{\mu}$ on $\ell(y,z(y))$ is such that $||\hat{\mu}-z(y)|| \leq \hat \sigma$
and $||\hat{\mu}-y|| \leq \hat \sigma$. Consider the point $f(u) \in \Gamma_f$ at the 
intersection between $\Gamma_f$ and $\ell(y,z(y))$. 

To show that $||\hat{\mu} - f(u)|| \leq 2 \epsilon$, and $f(u)$ belongs to the 
ball $B(\hat{\mu}, 2\, \epsilon)$ in Figure \ref{fig::plot1}, suppose to the contrary 
that $||f(u)-\hat{\mu}||>2\epsilon$, then either 
$||f(u) - y || < \hat \sigma - 2\epsilon$ or
$||f(u)-z(y)||< \hat \sigma -2\epsilon$. 
But if $||f(u) - y || < \hat \sigma - 2\epsilon$, then
$$
d(f(u), \partial S_0) \leq 
d(f(u), \hat{\partial S_0}) + \epsilon \leq ||f(u)-y|| +\epsilon < 
\hat \sigma - 2\epsilon + \epsilon < \sigma
$$
which contadicts the fact that $d(f(u), \partial S_0)=\sigma$. If, instead
$||f(u)-z(y)||< \hat \sigma -2\epsilon$, then
$$
d(f(u), \partial S_1) \leq 
d(f(u), \hat{\partial S_1}) + \epsilon = ||f(u)-z(y)|| +\epsilon < 
\hat \sigma - 2\epsilon + \epsilon < \sigma
$$
that contradicts the fact that $d(f(u), \partial S_1)=\sigma$.

\begin{figure} [h]
\begin{center}
\includegraphics[width=3.5in]{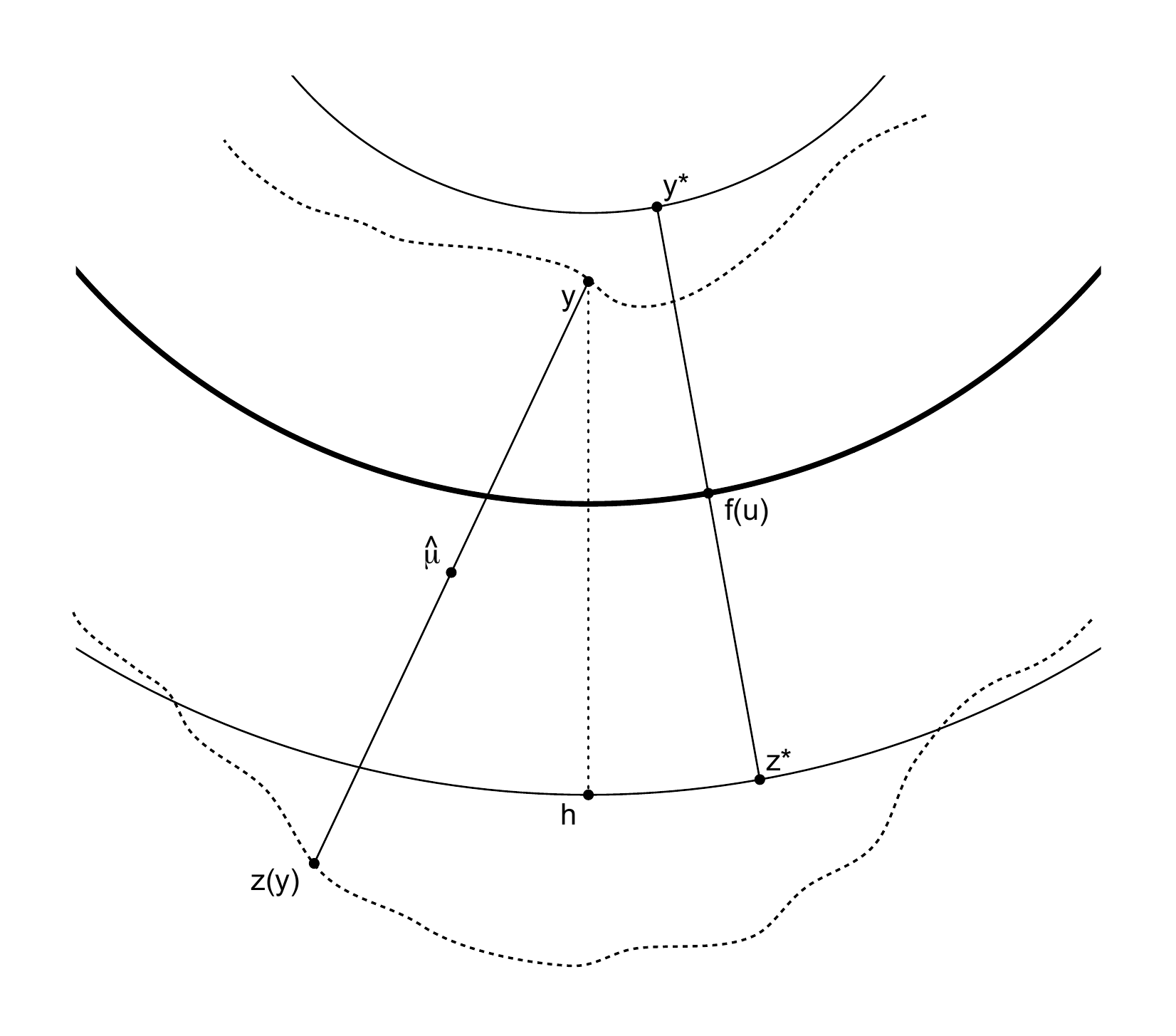}
\end{center}
\vspace{-.2in}
\caption{Second illustration for the proof of Theorem \ref{thm::methodII}.}
\label{fig::plot2}
\end{figure}

\smallskip

{\em (ii)} Let $f(u) \in \Gamma_f$, and let $y^*$ and $z^*$ be its closest points 
on $\partial S_0$ and $\partial S_1$ respectively, as in Figure \ref{fig::plot2}. 
By construction, $f(u)$ is on the midpoint of the segment $\ell(y^*,z^*)$, hence
$f(u)= (y^* + z^*)/2$. Consider a point $ y \in \hat{\partial S_0}$ such that 
$||y^*- y||\leq \epsilon$.  Let $z(y)$ and $h$ be the projections of $y$ on 
$\hat{\partial S_1}$ and $\partial S_1$ respectively.  The midpoint 
$\hat{\mu} = (y + z(y))/2$ belongs to $\hat\Gamma$. From Lemma \ref{thm::key2}, 
$||z(y) - h|| \leq C_1 (\sqrt \epsilon)$ uniformly.  Moreover, from {\bf Lemma 
\ref{lemma::painful}} that follows below, we have 
$||h-z^*|| \leq C_2 \epsilon$ uniformly. Hence:
$$
||z(\hat y) - z^*|| \leq ||z(\hat y) - h|| + ||h - z^*|| 
                     \leq C_1 (\sqrt \epsilon) + C_2 \epsilon \leq C_3 \sqrt \epsilon.
$$
It follows that
$$
||f(u)- \hat{\mu})||= \left\Vert \frac{y^* + z^*}2 - \frac{y + z(y)}2 \right\Vert 
\leq \frac{||y^* - y|| + ||y^* - y||}2 \leq \frac{C_1 + C_3}2 \sqrt\epsilon .
$$

{\em (iii)} is a consequence of {\em (i)} and {\em (ii)}.

\medskip
{\em 2.} The second statement follows from statement {\em 1.} and Lemma \ref{lemma::supp}.  
\end{proofof}

\medskip
Some terminology and the next {\bf Lemma \ref{lemma::inner.outer}} are needed for stating and 
proving {\bf Lemma \ref{lemma::painful}}. 

Now we examine the two disjoint curves $\partial S_i, i=0,1$ that 
constitute the boundary $\partial S$ when $\Gamma_f$ is closed, and the
set $\partial \cT$ when $\Gamma_f$ is open. For each boundary 
curve $\partial S_i$ we can distinguish two sides: one side that faces towards 
$\Gamma_f$, and a second side that faces away from $\Gamma_f$. Each point 
$x \in \partial S_i$ supports two tangent balls that contain no other points of 
$\partial S_i$, one on each side.

Analogously to the definition of thickness of a curve in Section \ref{sect::model},
we define the {\em Outer Thickness} of the boundary $\partial S_i$ 
to be the minimum radius of curvature $r_O$ of all the balls tangent to one point of $\partial S_i$
on the side facing away from $\Gamma_f$. 
We also define the {\em Outer Critical Ball} 
$O_x$ to be the ball facing away from $\Gamma_f$ and tangent to any point 
$x \in \partial S_i$, with radius $r_O$. Similarly, 
we define the {\em Inner Thickness} 
of the boundary $\partial S_i$ to be the minimum radius of curvature $r_I$ of all the 
balls tangent to one point of $\partial S_i$ on the side facing towards $\Gamma_f$,
and the {\em Inner Critical Ball} $I_x$ to be the ball facing towards $\Gamma_f$ 
and tangent to any point $x \in \partial S_i$, with radius $r_I$. Both balls can 
roll freely on the side of $\partial S_i$ where they are constructed, but not 
necessarily on the other side. The thickness of the boundary curves is 
$\Delta(\partial S_i) = \min \{r_O, r_I\}$.

\begin{lemma} \label{lemma::inner.outer}
For every point $y \in \partial S_i$ the outer critical ball $O_y$ has radius 
$r_O=\Delta - \sigma$ and the inner critical ball $I_y$ has radius
$r_I = \Delta + \sigma$. Thus the thickness of $\partial S_i, i=0,1$  is 
$\Delta(\partial S_i) = \Delta - \sigma$.
\end{lemma}

\vspace{.2cm}

\begin{proofof} {\bf Lemma \ref{lemma::inner.outer}.}
We start with the inner ball.
Let $y$ be a point on the boundary $\partial S_i$.
Hence, $y=f(u)+\sigma N(u)$ say.
Let $A=B(c,\sigma+\Delta)$
where $c=f(u)-\Delta N(u)$.
We claim that if $x$ is any other point on $\partial S_i$ then
$x\notin A$.
Let $\ell$ be the line segment connecting $x$ to $c$.
We wil show that the length of $\ell$ is strictly larger than 
$\sigma + \Delta$.
Now $x = f(v) + \sigma N(v)$ for some $v$.
The line $\ell$ crosses $\Gamma_f$ at some point $f(t)$.
The closest point on $\Gamma_f$ to $x$ is $f(v)$
and the distance from $x$ to $f(v)$ is $\sigma$.
Hence,
$||x- f(t)|| \geq \sigma$.
Let $A' = B(c,\Delta)$.
Then $f(t)\notin A'$ and hence
$||f(t)-c|| > \Delta$.
Therefore,
$||x -c|| = ||x-f(t)|| + ||f(t)-c|| > \sigma + \Delta$
as required.
The proof for the outer ball is similar.
\end{proofof}

\medskip
\begin{lemma} \label{lemma::painful} 
Let $y\in \partial S_i$, and let $y'$ be a second point in $\mathbb{R}^2$, such 
that $||y-~y'|| < \Delta - \sigma$. Denote by $z$ and $z'$, respectively, the
projections of $y$ and $y'$ on $\partial S_{1-i}$, then the distance between
$z$ and $z'$ is
$$
d(z,z') = ||z-z'|| \leq 2 \ \frac{\Delta+\sigma}{\Delta-\sigma} ||y-y'||.
$$
\end{lemma}

\medskip

\begin{proofof} {\bf Lemma \ref{lemma::painful}.} If $||y'-z|| \leq ||y'-y||$ then
$$
||z-z'|| \leq ||z-y'|| + ||y'-z'|| \leq 2||z-y'||
 \leq 2||y'-y||.
$$
If instead $||y'- z|| > ||y'-y||$ (see Figure \ref{fig::painful}), let $c$ be
the center of the outer critical ball $O_y \equiv B(c, \Delta-\sigma)$, and 
let $\theta$ be the angle $\hat{ycy'}$. Consider the triangle with vertices in 
$y, c$ and $y'$. Since $||c-y|| = \Delta - \sigma$, from the law of sines
applied to $\theta$ and to the angle facing $\ell(c,y)$
$$
\sin \theta = \sin\,(\hat{c y y'})\, \frac{||y-y'||}{||c - y||} 
\leq \frac{||y-y'||}{\Delta-\sigma}.
$$

Now we show that the point $z'$, projected from $y'$ onto $\partial S_{1-i}$, 
lies in the shaded region of Figure \ref{fig::painful}. In fact, 
the inner critical ball $I_z$, tangent to $z$ is such that
$z' \notin B(c, \Delta + \sigma) = I_z$. Moreover, since $z'$ is the 
closest point to $y'$, it follows that $||y'-z'|| \leq ||y'-z||$. 
Thus $z' \in B(y', ||y'-z||)$, and so
$z' \in B(y', ||y'-z||) \cap B(c, \Delta + \sigma)^c$, the shaded
region in Figure \ref{fig::painful}. 
The two balls $B(y', ||y'-z||)$ and $B(c, \Delta + \sigma)$ intercept in the two
points $z$ and $w$. 

Also, in the following {\bf Lemma \ref{lemma::painful2}}, we show that 
$||y-y'|| < \Delta - \sigma$ implies that $w$ is the farthest point from $z$ in 
the shaded region. 

The angle $\hat{zcw} = 2\,\theta$ and the length of the chord $\ell(z,w)$ 
is $||z-w|| = 2(\Delta+~\sigma) \sin \theta$. Thus
$$
||z-z'|| \leq ||z-w|| = 2(\Delta+\sigma) \sin \theta 
\leq 2(\Delta+\sigma) \frac{||y-y'||}{\Delta-\sigma}.
$$
\end{proofof}

\begin{lemma} \label{lemma::painful2}
If $||y-y'|| < \Delta - \sigma$ then $w$ is the farthest point from $z$ 
in the shaded area of Figure \ref{fig::painful}.
\end{lemma}

\begin{figure} [h]
\begin{center}
\includegraphics[width=3.5in]{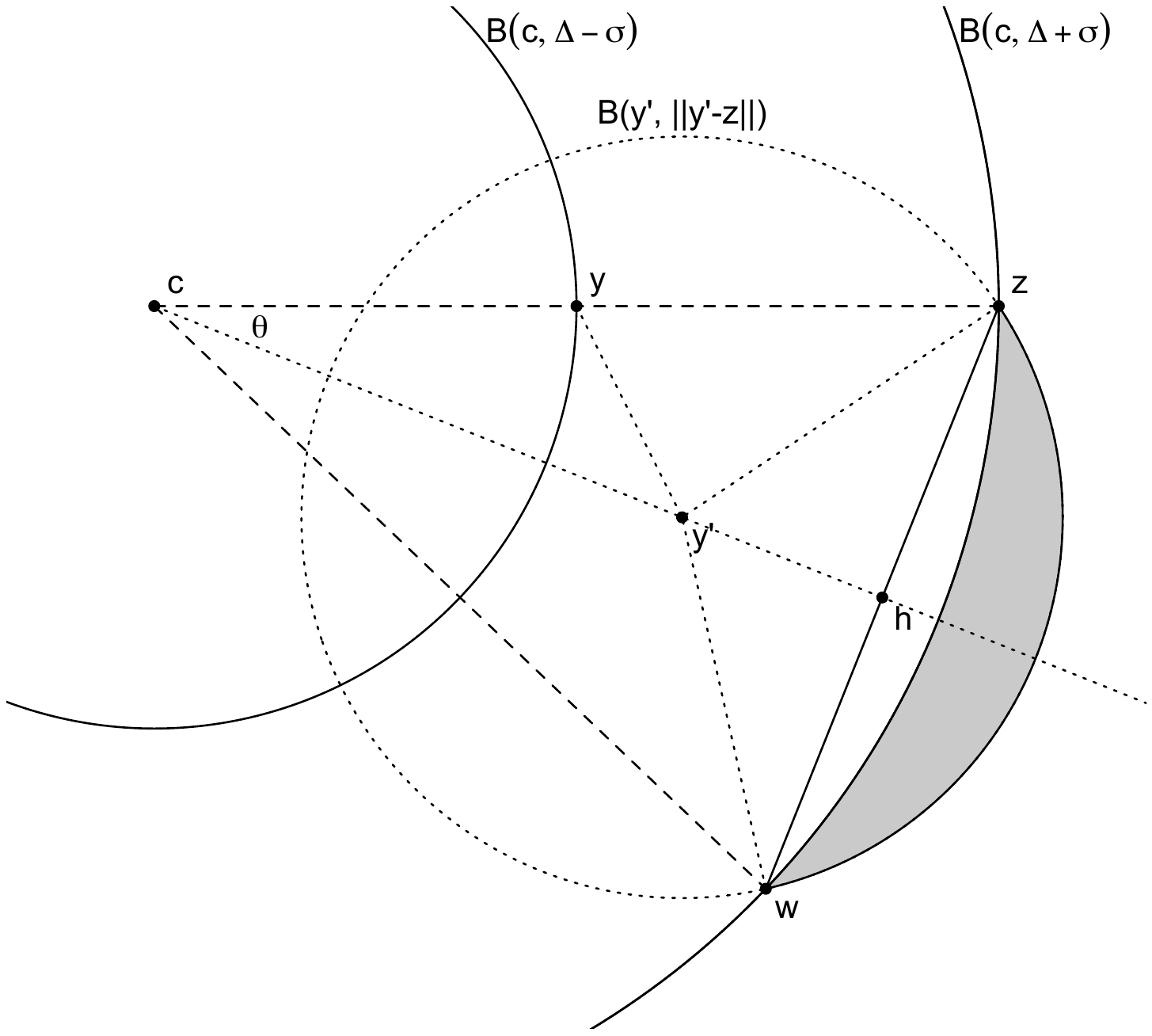}
\end{center}
\caption{Illustration for the proof of Lemma \ref{lemma::painful} and Lemma \ref{lemma::painful2}.}
\label{fig::painful}
\end{figure}

\begin{proofof} {\bf Lemma \ref{lemma::painful2}}. See Figure 
\ref{fig::painful}. While keeping the angle $\theta$ fixed, and as long as 
$||y-y'|| < \Delta - \sigma$, one can move the location of $y'$ along the 
radius of $B(c, \Delta + \sigma)$ from its center $c$ through $y'$. Let 
$h=(z+w)/2$ be the midpoint between $z$ and $w$. If $y'$ is chosen on the 
segment $\ell(c,h)$, then repeating the proof of Lemma \ref{lemma::painful} 
generates the same point $w$ as in Figure \ref{fig::painful}, and $w$ is still 
the point in the shaded region farthest away from $z$. 

If, instead, $y'$ is chosen along the line from $h$ onwards, then, by construction,
there are points in the shaded region that have distance from $z$ larger than 
$||z-w||$. But such $y'$ will violate the condition $||y-y'|| < \Delta - \sigma$. 
In fact, assume without loss of generality that the system of coordinates is such 
that: 
$$
c \equiv (0,0), \quad y \equiv (\Delta-\sigma,0), \quad z \equiv (\Delta+\sigma, 0),
$$
and the coordinates of $w$ and $h$ are
$$
w \equiv \biggl((\Delta+\sigma) \cos 2\,\theta, (\Delta+\sigma) \sin 2\,\theta \biggl);
\;
h \equiv \biggl(\frac{\Delta+\sigma}2 (1 + \cos 2\, \theta),
\frac{\Delta+\sigma}2 \sin 2\,\theta \biggl).
$$
By construction, if $y'$ lies along the line from $h$ onwards, then
$||y-y'|| \geq ||y-h|$. This implies that $ ||y-y'||\geq ||y-h||$, and
\begin{eqnarray*}
||y-h||^2 &=&
\left( \frac{\Delta + \sigma}2 \right)^2 + \left( \frac{\Delta-3\sigma}2 \right)^2
-2\left( \frac{\Delta + \sigma}2 \right)\left( \frac{\Delta-3\sigma}2 \right) \cos 2\theta 
\\
& \geq &
\left( \frac{\Delta + \sigma}2 \right)^2 + \left( \frac{\Delta-3\sigma}2 \right)^2
-2\left( \frac{\Delta + \sigma}2 \right)\left( \frac{\Delta-3\sigma}2 \right)
\\
&=&
(\Delta - \sigma)^2.
\end{eqnarray*}
\end{proofof}

\medskip
\begin{proofof} {\bf Lemma \ref{lemma::union-of-open-curves}.}
This follows from the fact that $\hat{\partial S}_0$ and
$\hat{\partial S}_1$ are each closed simple curves and that each
consists of finitely many arcs of a circle.
\end{proofof}

\medskip
\begin{proofof} {\bf Theorem \ref{thm::the-completion}.}
The fact that $f^*$ is a simple closed curve is straightforward.
We have already shown that each fitted value is within distance
$\epsilon_n$ of $M(S)$.
It is easy to see that this is true of the linear completion as well.
We still need to show that for each $y\in M(S)$ there is a fitted value
with distance $O(\epsilon_n)$.

Choose any $y = f(u)\in M(S)$. The fiber $L(u)$ divides $S$ into two disjoint sets.
Let $\hat{y}_1$ be the fitted value closest to $y$ from the first set and let
$\hat{y}_2$ be the fitted value closest to $y$ from the second set.
Let $\hat\ell = \{ \alpha \hat y_1 + (1-\alpha)\hat y_2 :\ 0\leq \alpha \leq 1\}$.
Let $y_1^*$ be the projection of $\hat y_1$ onto $M(S)$ and
$y_2^*$ be the projection of $\hat y_2$ onto $M(S)$.
Let $\ell$ be the line connecting $y_1^*$ and $y_2^*$.
Since the endpoints of $\hat\ell$ and $\ell$ are
$O(\epsilon_n)$ apart, it follows that
$d_H(\ell,\hat\ell)=O(\epsilon_n)$.

There are two balls $B_1$ and $B_2$
of radius $\Delta$ passing through
$y_1^*$ and $y_2^*$.
The arc of the curve $\Gamma_f$ from $y_1^*$ and $y_2^*$ is contained in the lens
$A=B_1\cap B_2$.
(If not, then a ball of radius $\Delta$ could not roll freely.)
So
$d(y,\ell) \leq d(y,\partial A)$.
But $d(y,\partial A)$ is simply the distance
from the chord of a circle to the circle,
where the chord has length $O(\sqrt{\epsilon})$.
It follows that
$d(y,\partial A) = O(\epsilon)$.
Finally, $d(f(u),M^*) \leq d(y,\ell)+ d_H(\ell,\hat\ell) = O(\epsilon)$.
\end{proofof}

\medskip
\begin{proofof} {\bf Theorem \ref{thm::extract}}.
For the open-curve case with $\cE_0 = \cE_1 = \mathbb{R}^2$, 
claim {\em 1.}~follows directly from 
Theorem \ref{thm::extreme-points-open-case}.
For the open-curve case with $\cE_0 = \partial\cA$ and $\cE_1 = \mathbb{R}^2$, 
as used in the general variant of the algorithm,
note that the endpoint in $\partial\cA$ is a distance $\le 8\epsilon$
from $\cU\intersect\partial\cA$,
where $\cU$ is the corresponding component of $\hat\Gamma$.
Moreover, every point in $\cA$ lies within $8\epsilon$ of $\cA\intersect M$.
Therefore {\em 1.}~also follows from 
Theorem \ref{thm::extreme-points-open-case}.

For the closed case, all that must be proved is that the estimated curve $\hat\Gamma$
is a closed curve that lies within $\hat\Gamma$ and that has (absolute) winding number 1 
around a point $y_0$ in the inner component of ${\hat S}^c$.
Notice also that in the closed case, both the closed and general variants of the algorithm
produce the same curve.

First, recall that $d(\hat y, M) \le 2 \epsilon$
and notice that the unique fiber through $\hat y$, $L(u_0)$, intersects with $\hat\Gamma$
in a line segment of length $\le 8\epsilon$.
This portion of a fiber is thus contained in the set $\cA_8$ (defined in EDT curve 
extraction algorithm for the closed curve case) and thus in $\cA$;
in fact, the fibers $L(u) \intersect \hat\Gamma$ for $u$ in an open set containing $u_0$ are also 
contained in $\cA$.
$\hat\Gamma - \cA$ is thus cut at a fiber and contains a single connected component because
$\epsilon \ll \Delta(f)$.
(If the latter were false, two separated parts of $f$ would lie within $8\epsilon<\Delta$ 
of each other.)

Second, applying the open curve algorith with $\cE_0 = \cE_1 = \partial A$
produces a curve within $\hat\Gamma$ that connects one side of $\cA$ to the other.
(If the latter were false, the curve between the endpoints would have 
length $\le 8\epsilon$, but a longer minimum path length can be obtained
by winding around $y_0$. The winding number cannot be greater than 1
because the curve in $\hat\Gamma - \cA$ is not closed.)
A path between these end points that is contained in $\cA$ closes this curve
and keeps it within $\hat\Gamma$.
The resulting closed curve thus lies within $\hat\Gamma$ and has (absolute) winding number 1
with respect to $y_0$.
Claim {\em 2.}~follows.
\end{proofof}

\medskip

\begin{lemma} \label{lemma::boundary-inclusion}
Suppose $S$ is a compact, connected set in $\R^2$.
Then,
\begin{enumerate}
\item If $y\in S^c$ and $x \in S$ and if $L$ is the line segment from
$x$ to $y$, then $L \intersect \partial S \ne \emptyset$
\item Fix $r > 0$. If $B(x,r) \intersect S \ne \emptyset$ but $B(x,r) \intersect\partial S = \emptyset$,
then $B(x,r) \subset S$. 
\end{enumerate}
\end{lemma}

\noindent
\begin{proof}

{\em 1.} Define $d_* = \inf\Set{d(x,w):\; w\in L\intersect S^c}$.
We know that the infimum exists because $L$ is a compact set
and that there is a unique point $z\in L$
for which $d(x,z) = d_*$.
It follows directly that every neighborhood of $z$
contains a point in $S$ and a point in $S^c$,
so $z\in\partial S$.

{\em 2.} Suppose the conclusion does not hold;
that is, there exists a $y\in B(x,r)\intersect S^c$.
By assumption, there is an $z\in B(x,r)\intersect S$.
Apply Result 1 in the lemma to the line segment
between $z$ and $y$, which is contained in $B(x,r)$ by convexity.
This implies that $B(x,r) \intersect \partial S \ne \emptyset$,
contradicting the initial supposition.
The result follows.
\end{proof}

\subsection{Open Curves}
\label{sec::open}

This subsection deals with two issues related to open curves. 
First, the EDT curve extraction algorithm (Subsection \ref{sect::ALGextract})
for open curves required that we estimate the endpoints of the curve. 
Second, for constructing the medial estimator in the open curve case, the
estimated boundary $\hat{\partial S}$ needs to be split into two pieces 
$\hat{\partial S}_0$ and $\hat{\partial S}_1$. 
Both issues are addressed here after the following 
discussion of some basic properies of open curves.

\medskip

Let $f$ be an open curve and let $\partial S$ the boundary support. Define 

$$
E_0(c) = B(f(0),\sigma+ c) \cap \partial S
$$
and
$$
E_1(c) = B(f(1),\sigma+ c) \cap \partial S
$$
to be the extended end caps of the support's boundary.
Let, $E_0 \equiv E_0(0)$ and $E_1 \equiv E_1(0)$.
Define
\begin{eqnarray} \label{eq::V0}
\nonumber
V_0(a) &=& E_0 \Union \left( \union \{ f(u) \pm \sigma N(u):\ \ 0 \leq u \leq a\}\right)\\
\\
\nonumber
V_1(a) &=& E_1 \Union \left( \union \{ f(u) \pm \sigma N(u):\ \ 1- a \leq u \leq 1\}\right).
\end{eqnarray}

Let $\varphi$ denote the arclength of $f$.
Because $f$ is parameterized by arclength normalized to $[0,1]$,
it follows that
the gradient $f'$ of the filament
is such that
$||f'(u)|| = \varphi$ for all $u$.
Define
\begin{equation}
\label{eq::ace}
a \equiv a(c,\epsilon) = 
\sqrt{\frac{2\sigma c + c^2\epsilon}{\varphi^2\left(1 - \sigma/\Delta\right)}}.
\end{equation}

\begin{theorem}
\label{theorem::intersect}
Let $c \geq 0$.
\begin{enumerate}
\item 
$E_0(c\epsilon) \subset V_0(a\sqrt{\epsilon})$ and
$E_1(c\epsilon) \subset V_1(a\sqrt{\epsilon})$.
\item
Let 
$b = a  \varphi \left( 1 + \sigma/\Delta \right)$.
Then
$V_0(a\sqrt{\epsilon})\subset E_0 \oplus b\sqrt{\epsilon}$
and
$V_1(a\sqrt{\epsilon})\subset E_1 \oplus b\sqrt{\epsilon}$.
\item
Let $d = \varphi\, a$.
Then
$V_0(a\sqrt{\epsilon})\subset E_0(d\sqrt{\epsilon})$ and
$V_1(a\sqrt{\epsilon})\subset E_1(d\sqrt{\epsilon})$.
\end{enumerate}
\end{theorem}

\noindent 
\begin{proof}
(See Figure \ref{fig::FIGFIG}).\\
{\em 1.} Let $y \in E_0(c\epsilon)$.
We will show that $y \in V_0(a \sqrt\epsilon)$.
Note that $y$ cannot belong to $E_0$ hence, being $y \in \partial S$,
necessarily $y=f(u) \pm \sigma N(u)$ for some $u$.  We only need to
prove that $u \leq a\sqrt \epsilon$.
Moreover, since for all $u$
$$
||f(0) - f(u)|| \geq \varphi u,
$$
proving that $||f(0) - f(u)|| \leq \varphi a \sqrt\epsilon$ would be sufficient for the claim.

We know that $||y-f(0)|| = \sigma + c\epsilon$ and that $y=f(u)
+ \sigma N(u)$ for some $u$.  The line from $f(u)$ to $y$ defines the
direction of the normal at $f(u)$. Extend the normal at $f(u)$ to the
point $z= f(u) + \Delta N(u)$, so that $||z-f(u)||= \Delta$ and
$||z-f(0)|| > \Delta$, hence $z$ must lie outside the circle
$C3=B(f(0), \Delta)$.
Let $A$ be the intersection between $C3$ and the segment from $y$ to $z$;
such intersection exixts because $y \in C3$ and $z \not \in C3$. We have
\begin{equation*}
||A - f(0)|| = \Delta,  \quad ||A - f(u)|| = \Delta - h \quad \mbox{and} \quad
||A - y||= \Delta - \sigma - h,
\end{equation*}
where $h$ is positive.

\newpage
\begin{figure}
\vspace{-1.2cm}
\includegraphics[width=4.5in]{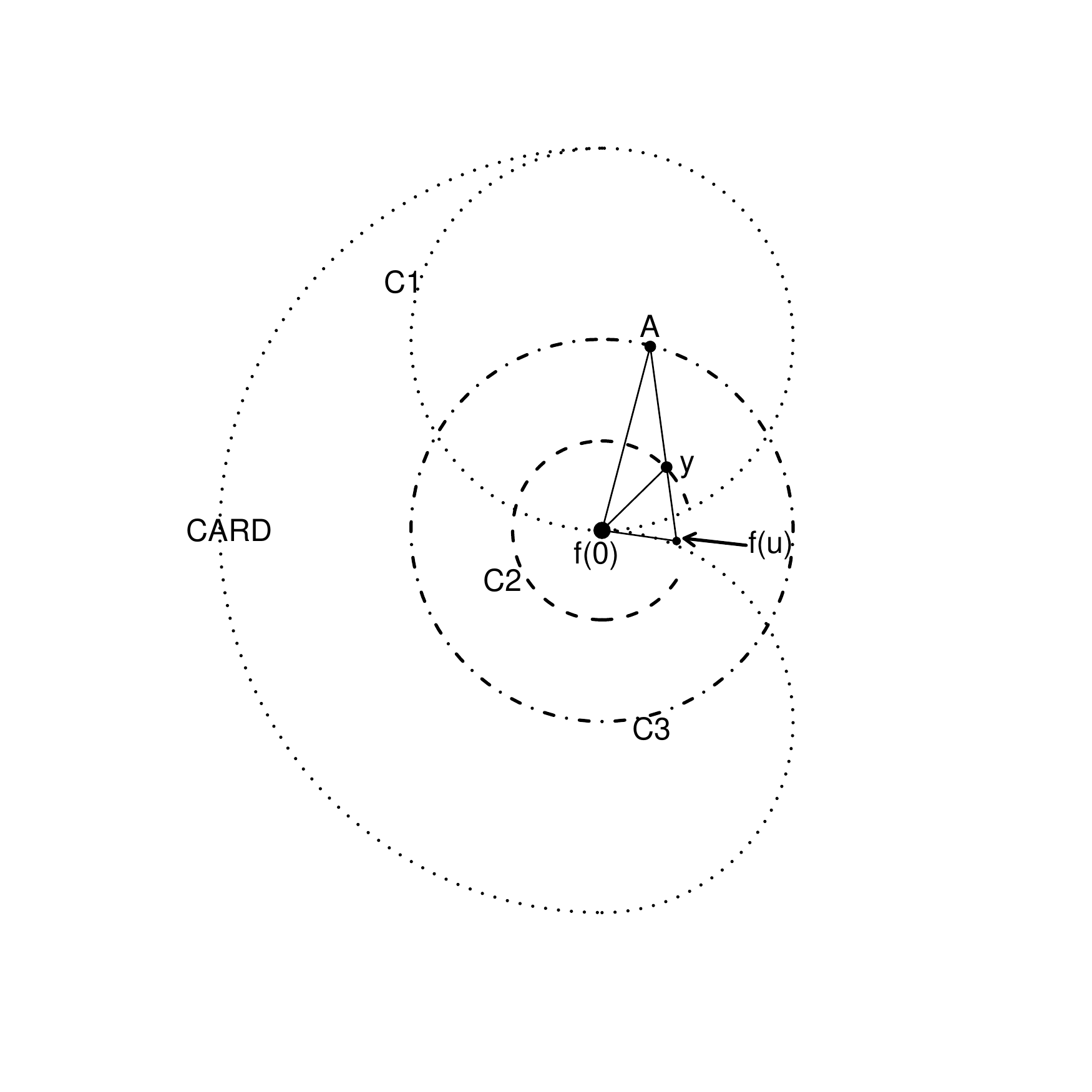}
\vspace{-1.cm}
\caption{Diagram for Theorem \ref{theorem::intersect}.}
\label{fig::FIGFIG}
\end{figure}

Consider the triangle with vertices $A$, $f(0)$ and $y$ and denote by $\theta$ 
the angle at $y$. The cosine theorem gives
$$
|| A - f(0)||^2 = ||A-y||^2 + ||f(0)-y||^2  -2 ||A-y|| \cdot ||f(0)-y|| \cos \theta
$$
so that
$$
\Delta^2 = (\Delta - \sigma - h)^2 + (\sigma + c\epsilon)^2 
-2 (\Delta - \sigma -h)(\sigma + c\epsilon) \cos \theta
$$
and
$$
\cos \theta = \frac{(\Delta - \sigma -h)^2 + (\sigma + c\epsilon)^2 - \Delta^2}
                   {2(\Delta - \sigma - h)(\sigma + c\epsilon)}.
$$
Now consider the triangle with vertices $f(0)$, $f(u)$ and $y$, where the angle 
at $y$ is $\pi - \theta$. 
From the cosine theorem we obtain
$$
||f(0)-f(u)||^2 = ||f(0)-y||^2 + ||y-f(u)||^2 - 2 ||f(0)-y|| \cdot ||y-f(u)|| \cos(\pi-\theta).
$$ 
And, since $\cos(\pi - \theta)= -\cos\theta$
\begin{eqnarray*}
||f(0)-f(u)||^2 &=&
(\sigma + c\epsilon)^2 + \sigma^2 + 2 (\sigma + c\epsilon) \sigma \cos \theta \\
&=&
(\sigma + c\epsilon)^2 + \sigma^2 +  
\sigma\frac{(\Delta - \sigma -h)^2 + (\sigma + c\epsilon)^2 - \Delta^2}{(\Delta - \sigma - h)} \\
&=&
(\sigma + c\epsilon)^2 + \sigma^2 +  \sigma (\Delta - \sigma -h) + 
\sigma\frac{(\sigma + c\epsilon)^2 - \Delta^2}{(\Delta - \sigma - h)} .
\end{eqnarray*}

Note that for small $\epsilon$ (as long as $\Delta^2 \geq (\sigma + c\epsilon)^2$) 
$$
s(h) \equiv (\sigma + c\epsilon)^2 + \sigma^2 +  \sigma (\Delta - \sigma -h) + 
\sigma\frac{(\sigma + c\epsilon)^2 - \Delta^2}{(\Delta - \sigma - h)}
$$
is a decreasing function of $h$ and then for all $h>0$
$$
s(h) \leq s(0) = \frac{\epsilon (c^2\epsilon + 2\sigma c)\Delta}{\Delta - \sigma}.
$$
As a consequence
$$
||f(0)-f(u)||^2 \leq \frac{\epsilon (c^2\epsilon + 2\sigma c)\Delta}{\Delta - \sigma}
$$
and
$$
||f(0)-f(u)|| \leq \sqrt \epsilon \sqrt{ \frac{(c^2\epsilon + 2\sigma c)}{1 - \sigma/\Delta}} = 
\sqrt\epsilon \varphi a(c,\epsilon).
$$

{\em 2.} Let $y = f(u) + \sigma N(u)\in V_0(a\sqrt{\epsilon})$.
Let
$y_0 = f(0) + \sigma N(0)$.
Then
$y_0 \in E_0$ and
$$
||y-y_0|| \leq
||f(u) - f(0)|| + \sigma ||N(u) - N(0)|| \leq
\varphi\, u + \frac{\varphi \sigma u}{\Delta} \leq 
a \sqrt{\epsilon} \varphi \left( 1 + \frac{\sigma }{\Delta} \right) = b\sqrt{\epsilon}
$$
where we used
Theorem 1(iii) of 
\cite{Walther}, namely,
$$
||N(u) - N(v)|| \leq \frac{|| f(u) - f(v)||}{\Delta}.
$$

{\em 3.}
Let $y = f(u) + \sigma N(u)\in V_0(a\sqrt{\epsilon})$.
Then
$$
||y-f(0)|| \leq ||f(u)-f(0)|| + \sigma \leq \varphi\,u + \sigma \leq \sigma + 
       \varphi\,a \sqrt{\epsilon}.
$$
\end{proof}

\subsection{Estimating the Endpoints} \label{sec::Extrac}

In this subsection we derive estimators for
$f(0)$ and $f(1)$.
First we will need some lemmas.
Let $\hat \Gamma$ be the EDT estimator.

\begin{lemma} \label{lemma::Mhat-properties-top}
For fixed $\epsilon > 0$, the set $\hat\Gamma$ has the following properties.
Suppose $u\in \Gamma$ and $\cF_u$ is the intersection of $\hat\Gamma$ and the fiber of $S$ containing $u$.
\begin{enumerate}
\item If $f$ is closed, then $\cF_u$ is a connected line segment through $u$.
\item If $f$ is open and $u$ lies at least $2\epsilon$ from $f(0)$ and $f(1)$
 then $\cF_u$ is a connected line segment.
\end{enumerate}
\end{lemma}

\noindent
\begin{proof}
Without loss of generality, we can assume that the fiber through $u$ is oriented
vertically
and that $u = (0,\Delta)$. 
$\Gamma$ must lie above the circle of radius $\Delta$ centered on the origin,
and thus the boundary of the support (on that side of $\Gamma$) must lie above the circle
of radius $\Delta + \sigma$ centered on the origin.
It follows that the outer portion of $\hat{\partial S}$ must lie above
the circle of radius $\Delta + \sigma - \epsilon$ centered on the origin.

First, consider the point 
$x=(0,y)$ where
$y \Delta + h\epsilon$ for $0\le h \le 4$ on the fiber
through $u$. Let $d = \hat\sigma - \delta = \sigma - c \epsilon$ for
some $c\in[-1,3]$ and $r = \Delta + \sigma - \epsilon$. 
And let $z = (r\sin\theta, r\cos\theta)$.
We want to find the maximum $|\theta|$ such that $\norm{y-z} < d$;
this will show limit the range of closest points.

We have that 
$$
\norm{x-z}^2
  = r^2\sin^2\theta + (y - r\cos \theta)^2 
  = r^2 + y^2 - 2 r y \cos\theta.
$$
Taking
$d^2 > r^2 + y^2 - 2 r y \cos\theta$ yields
$$
\cos\theta > \frac{(r-y)^2 + 2 r y - d^2}{2 r y}
           = 1 - \frac{d^2 - (r-y)^2}{2 r y}
     = 1 - \frac{(\sigma - c\epsilon)^2 - (\sigma - \epsilon - h\epsilon)^2}{(\Delta + \sigma - \epsilon)(\Delta + h\epsilon)}.
$$
Hence,
\begin{align}
|1 - \cos(\theta)|&< \left|\frac{(1+h^2+c^2)\epsilon^2 + 2\sigma(1+h-c)\epsilon}{(\Delta + \sigma - \epsilon)(\Delta + h\epsilon)}\right| \\
                  &= \frac{\epsilon}{\Delta}\,\frac{\Delta}
{\Delta+\sigma-\epsilon}\,
\left|\frac{(1+h^2+c^2)\epsilon + 2\sigma(1+h-c)}{\Delta + h\epsilon}\right| < 2\frac{\epsilon}{\Delta}
\end{align}
and thus
$|\sin(\theta)| < 2\sqrt{\frac{\epsilon}{\Delta}}$.

Second, consider a wedge of half angle $\theta \ge 0$ around the vertical axis.
Consider points 
$x_0=(0,y_0)$ and $x_1 = (0,y_1)$ where
$y_0 = \Delta + h_0\epsilon$ and $y_1 = \Delta+h_1\epsilon$
with $0 \le h_1 \le h_0$.
Let $z = (r\sin\varphi,r\cos\varphi)$ for $r \ge \Delta + \sigma - \epsilon$
and $|\varphi| \le \theta$.
We want to find the value of $\theta$ such that $\norm{x_0-z} \le \norm{x_1-z}$ for all such $z$.
In this wedge, in other words, distance to the estimated boundary is monotone along
the filament.

We have
$$
\norm{x_i-z}^2 = r^2\sin^2\varphi + (y_i - r\cos\varphi)^2 =
 r^2 + y_i^2 - 2 r y_i \cos\varphi.
$$
Hence, $\norm{x_0-z} \le \norm{x_1-z}$ requires that
$$
2 r\cos\varphi (h_0 - h_1)\epsilon \ge (h_0 - h_1)\epsilon (2 \Delta + (h_0 + h_1)\epsilon)\\
$$
or equivalently
$$
\cos\varphi \ge \frac{\Delta + \frac{h_0 + h_1}{2}\epsilon}{r}.
$$
This is satisfied whenever $\cos\varphi \ge 1/2$ or equivalently
when $|\varphi| \le \pi/3$.

Combining these two parts, we see that the closest point to the boundary
must lie within a wedge of angular extend $O(\sqrt{\epsilon/\Delta})$,
which is contained in the wedge for which distance to the estimated boundary
is monotone along the filament.
Claim {\em 1.}~follows.
For open curves, 
claim {\em 2.}~follows from the same argument for a point $u$ for
which the fiber through $u$ is sufficiently far from the endcaps.
\end{proof}

\begin{lemma} \label{lemma::Mhat-pcsmooth-bnd}
$\hat\Gamma$ has a finite piecewise $C^2$ 
(two continuous derivatives)
boundary.
\end{lemma}

\begin{proof}
Because $\hat {\partial \Gamma} =  \hat{\partial\Gamma^c}$, 
it is sufficient to show that $\hat\Gamma^c$ has a 
piecewise smooth boundary. Since $\hat\Gamma$ consists of all points $x\in S$ 
such that $d(x, \partial\hat S) \ge \hat\sigma_n - \delta \equiv w$ for some 
constant $\delta > 0$, it follows that
\begin{equation}
\hat\Gamma^c = \hat S^c \,\Union\, \Biggl(\Union_{z\in\partial\hat S} B(z, w)\Biggr).
\end{equation}
Thus, $\hat {\partial \Gamma^c}$ consists of the points in $S$ that are exactly $w$ away
from $\hat {\partial S}$. 
Because $\hat{\partial  S}$ is a finite union of circular arcs of radius $\epsilon$,
it follows that $\hat {\partial \Gamma^c}$ ($=\hat {\partial \Gamma}$)
is the boundary of a finite union of $w$-enlargements of circular arcs.
A set that is a finite union of sets with piecewise smooth boundaries itself
must have a piecewise smooth boundary.
Thus, it is sufficient to show that the $w$-enlargement of a single circular arc
has a piecewise smooth boundary.

To do this, let ${\cal A}$ be a circular arc, which we can take without loss
of generality to be of the form 
$$
{\cal A} = \{(r \cos t, r\sin t):\; t\in[-\theta_0,\theta_0]\},
$$
for $\theta_0 \in [0,\pi)$.
Let $x_+ = r(\cos \theta_0, \sin\theta_0)$ and $x_-= r(\cos \theta_0, -\sin\theta_0)$
be the two endpoints.
Let $v(x)$ denote the point(s) in ${\cal A}$ that is (are) closest to $x\in\mathbb{R}^2$.
For $x$ in the cone $\lambda_- x_- + \lambda_+ x_+$ for $\lambda_-,\lambda_+ \ge 0$,
$v(x) = r x/||x||$. For $x$ on the negative horizontal axis,
$v(x)$ contains the two endpoints of the arc.  For all other $x$,
$v(x)$ contains the endpoint of the arc on the same side of the horizontal
axis as $x$.
It follows that the set of points $x$ for which $d(x,v(x)) = w$
is a union of three circular arcs: one in the cone consisting of points at radius
$r+w$, one for $x_-$ consisting of part of the circle around $x_-$, and one
for $x_+$ consisting of part of the circle around $x_+$.
This proves the lemma.
\end{proof}

\begin{theorem} \label{thm::extreme-points-open-case}
Let $f$ be an open curve.
Let ${\cal P}_{u,v}$ denote the set of paths between $u, v\in\hat\Gamma$
that are contained in $\hat\Gamma$.
Define $x, y \in \hat\Gamma$ by
\begin{equation}\label{eq::end-est}
x, y = \argmax_{u, v\in\hat\Gamma} \min_{\pi\in{\cal P}_{u,v}} {\rm length}(\pi),
\end{equation}
where length denotes the arclength of the path.

Then,
\begin{equation}
  d_H\left(\{x,y\}, \{f(0), f(1)\}\right) \le 16 \epsilon  
\end{equation}
\end{theorem}

The two quantities
$x,y$ defined in equation
(\ref{eq::end-est}) are the estimates 
of the endpoints.

\smallskip

\begin{proof}
Suppose $d_H\left(\{x,y\}, \{f(0), f(1)\}\right) > 16\epsilon$.
Then, either $x$ or $y$ must be farther than $16\epsilon$ from $f(0)$ or $f(1)$.
Suppose without loss of generality that
$$ 
\min_{\pi \in {\cal P}_{x,f(0)}} {\rm length}(\pi) < \min_{\pi \in {\cal P}_{x,f(1)}} {\rm length}(\pi). 
$$
That is, we are labeling the two points so that $x$ is ``paired'' with $f(0)$ 
and $y$ is ``paired'' with $f(1)$. Assume that $|x - f(0)| > 16\epsilon$; we show 
that a contradiction follows.

Because $\hat\Gamma \subset  \Gamma \oplus (4\epsilon)$,
it follows that $x$ lies on one of the fibers through $\Gamma$.
(That is, it lies in the ``body'' of $\hat\Gamma$, not in the ``caps,''
whose points are all $< \epsilon$ from $f(0)$.)
Call this fiber ${\cal F}$.
Let $\pi$ be the shortest path from $x$ to $y$,
and let $\pi'$ be the shortest path from $f(0)$ to $y$.
By the assignment of $x$ and $y$ above, it follows that
$\pi'$ must pass through the fiber ${\cal F}$.
Let $x'$ be the point that $\pi'$ passes through on ${\cal F}$
and define $\ell(z)$ to be the length of the shortest
path from $z\in\hat\Gamma$ to $y$.

Again because $\hat\Gamma \subset  \Gamma \oplus (4\epsilon)$,
it follows that ${\cal F}$ has length $\le 8\epsilon$.
Because $f$ is an open curve, it follows immediately that $\hat\Gamma$ is
simply connected.
We claim that
$\ell(x') - \ell(x) \leq ||x-x'||$.
To see this, let $\pi_0$ be the shortest path within $\hat\Gamma$ from $x$ to $y$ and $\ell(x)$  be its length.
Consider the following path joining $x'$ to $y$: start at $x'$, move linearly to $x$ along the fiber,
then follow $\pi_0$. From the previous lemma, this path is entirely within $\hat\Gamma$.
Since the length of this path is $||x-x'|| + \ell(x)$, we have $\ell(x') \leq \ell(x) + ||x-x'||$ and
$$
\ell(x') - \ell(x) \leq ||x-x'||.
$$
Now invert the roles of $x$ and $x'$ and get 
$$
\ell(x) - \ell(x') \leq ||x-x'||
$$
so that the claim follows.

Now,
$$
||f(0) - x'|| \ge ||f(0) - x|| - ||x - x'|| > 16\epsilon - 8\epsilon = 8\epsilon.
$$
It follows that 
$\ell(f(0)) > 8\epsilon + \ell(x')> 8\epsilon + \ell(x) - 8\epsilon > \ell(x)$
which contradicts the assumption that $||x - f(0)|| > 16\epsilon$.
Applying this same argument to $y$ and $f(1)$ shows by contradiction
that $||y - f(1)|| \le 16\epsilon$.
This proves the theorem.
\end{proof}

\subsection{Estimating the Boundaries} \label{sec::boundaries}

Now we consider estimating
$\partial S_0$ and $\partial S_1$.
The estimators are defined in Theorem \ref{thm::this-is-the-est}
but we need some preliminary results first.
Let $\hat{\partial S}$ be an estimate of $\partial S$ such that
$d_H(\partial S,\hat{\partial S})\leq \epsilon$ and let
$\hat x_0$ and $\hat x_1$ be 
the endpoint estimators from Theorem \ref{thm::extreme-points-open-case},
that are such that
$$
||\hat x_0 - f(0)||\leq C \epsilon,\ \ \ ||\hat x_1 - f(1)||\leq C \epsilon.
$$
Define
\begin{center}
\begin{tabular}{ll}
$\hat B_0 = B(\hat x_0,\hat\sigma+ c\epsilon)$      & $\hat B_1 = B(\hat x_1,\hat\sigma + c\epsilon)$\\
$\hat E_0 = \hat{\partial S} \cap \hat B_0$ & $\hat E_1 = \hat{\partial S} \cap \hat B_1$\\
\end{tabular}
\end{center}

Recall the definitions of $V_0, V_1$ and $a$ given in
(\ref{eq::V0}) and
(\ref{eq::ace}).

\begin{theorem}  \label{Th.Th}
Suppose that, 
$d_H(\partial S,\hat{\partial S})\leq \epsilon$,
$||\hat x_0 - f(0)||\leq C \epsilon$,
$||\hat x_1 - f(1)||\leq C \epsilon$ and that
$\hat{\partial S}$ is connected.
Assume that $c \geq C+1$.
Let $a = a(2+c+C,\epsilon)$.
Let $V_0 = V_0(a\sqrt{\epsilon})$ and 
$V_1 = V_1(a\sqrt{\epsilon})$.
Then:
$$
d_H(V_0,\hat E_0)\leq b \sqrt{\epsilon}\ \ \ {\rm and}\ \ \ 
d_H(V_1,\hat E_1)\leq b \sqrt{\epsilon}
$$
where $b = a(\varphi + \sigma/\Delta)$.
\end{theorem}

\begin{proof}
Let $\hat x \in \hat E_0$.
Thus $||\hat x - \hat x_0|| \leq \hat\sigma + c \epsilon$.
There exists $x \in\partial S$ such that
$||\hat x - x|| \leq \epsilon$.
Now,
\begin{eqnarray*}
||x - f(0)|| & \leq &
||x - \hat x|| + ||\hat x - \hat x_0|| + ||\hat x_0 - f(0)||\\
& \leq & \epsilon + (\hat\sigma + c \epsilon) + C\epsilon\\
& \leq & \epsilon + (\sigma + (c+1) \epsilon) + C\epsilon\\
&=& \sigma + (2+c+C)\epsilon.
\end{eqnarray*}
Thus
$x\in B(f(0),\sigma + (2+c+C)\epsilon)\cap \partial S \in V_0$, by 
Theorem \ref{theorem::intersect}, and
$$
\hat E_0 \subset V_0\oplus\epsilon \subset V_0 \oplus (b\sqrt{\epsilon}).
$$

Now let $x \in V_0$.
There exists $z \in B(f(0),\sigma+ c\epsilon) \cap \partial S$
such that $||x-z||\leq b\sqrt{\epsilon}$.
There is a $\hat z \in \hat{\partial S}$ such that
$||\hat z - z|| \leq \epsilon$.
Now,
\begin{eqnarray*}
||\hat z - \hat x_0|| & \leq &
||\hat z - z|| +  ||z - f(0)|| + ||f(0)-\hat x_0||\\
& \leq & \epsilon + \sigma + C \epsilon\\
& \leq & \hat\sigma + (C+1)\epsilon \leq \hat\sigma + c \epsilon.
\end{eqnarray*}
Therefore $\hat x\in \hat E_0$ and so
$V_0 \subset \hat E_0 \oplus (b\sqrt{\epsilon})$.
\end{proof}

\medskip

\vspace{.5cm}

There is no guarantee that $\hat E_0$ and $\hat E_1$ are connected sets.
But this is crucial if we want to use them for the medial 
estimation procedure.
Define the completion
of $\hat E_0$ denoted by
$[\hat E_0]$ to be the smallest connected subset of
$\hat{\partial S}$ containing
$\hat E_0$.
That is,
$$
[\hat E_0] = \bigcap \Biggl\{ C:\ C\ {\rm is\ connected},\ 
C \subset\hat{\partial S},\ \hat E_0 \subset C\Biggr\}.
$$
Define
$[\hat E_1]$ similarly.
Finally, define
$$
\hat R = \hat{\partial S} - ( [\hat E_0] \union[\hat E_1]).
$$
Now by construction,
$[\hat E_0]$ and $[\hat E_1]$ are connected.
If they are disjoint, it follows that
$\hat R$ consists of two connected components.
To make sure that the completion procedure successfully combines elements
of $\hat E_0$ without adding other elements, we need the following.

\begin{theorem} \label{B}
Under the assumptions of  Theorem \ref{Th.Th}:
$$
\max_{\hat x , \hat y \in \hat E_0} ||\hat x - \hat y||\leq 2\hat\sigma + 2c\epsilon,
\ \ \ {\rm and}\ \ \ 
\max_{\hat x , \hat y \in \hat E_1} ||\hat x - \hat y||\leq 2\hat\sigma + 2c\epsilon.
$$
If $\sigma \leq ||f(1)-f(0)||/5\,$ then 
$\, \min_{\hat x \in \hat E_0,\hat y \in \hat E_1}||\hat x - \hat y|| > 
2\hat\sigma + 2c\epsilon.$
\end{theorem}

\begin{proof}
For any
$\hat x , \hat y \in \hat E_0$, we have
$||\hat x - \hat y||\leq ||\hat x - \hat x_0|| + ||\hat y - \hat x_0|| \leq
2\hat\sigma + 2 c\epsilon$.
Now let
$\hat x \in \hat E_0$ and
$\hat y \in \hat E_1$.
Now
\begin{eqnarray*}
||\hat x_0 - \hat x_1|| & \leq &
||\hat x_0 - \hat x|| + ||\hat x - \hat y|| + ||\hat y - \hat x_1|| \\
& \leq &
2\hat\sigma + 2c\epsilon + ||\hat x - \hat y||.
\end{eqnarray*}
Hence,
\begin{eqnarray*}
||\hat x - \hat y|| & \geq & ||\hat x_0 - \hat x_1||  - 2\hat\sigma - 2 c \epsilon\\
& \geq &
||f(0)-f(1)|| -  ||\hat x_0 - f(0)|| - ||\hat x_1 - f(1)|| - 2\hat\sigma - 2 c \epsilon\\
& \geq & 
||f(0)-f(1)|| -  2C\epsilon - 2\hat\sigma - 2 c \epsilon\\
&=&
||f(0)-f(1)|| -  2\hat\sigma - 2(C+c)\epsilon\\
& \geq &
5\sigma - 2\hat\sigma - 2(C+c)\epsilon\\
&=& 3\sigma + 2\sigma - 2\hat\sigma - 2(C+c)\epsilon\\
& \geq & 3\sigma + 2\hat\sigma - 2\epsilon + 2\hat\sigma - 2(C+c)\epsilon\\
&=& 3\sigma - 2(C+c-2)\epsilon > 2\hat\sigma + 2c\epsilon.
\end{eqnarray*}
\end{proof}

Combining the above results we have the following.

\begin{theorem}
\label{thm::this-is-the-est}
Suppose that, $d_H(\partial S,\hat{\partial S})\leq \epsilon$,
$||\hat x_0 - f(0)||\leq C \epsilon$,
$||\hat x_1 - f(1)||\leq C \epsilon$ and that
$\hat{\partial S}$ is connected.
Assume that $c \geq C+1$.
Let $a = a(c,\epsilon)$.
Let $V_0 = V_0(a\sqrt{\epsilon})$ and 
$V_1 = V_1(a\sqrt{\epsilon})$.
If $\sigma \leq ||f(1)-f(0)||/5$ then:
\begin{enumerate}
\item $d_H(V_0,[\hat E_0])\leq c_1 \sqrt{\epsilon}$ and
$d_H(V_1,[\hat E_1])\leq c_1 \sqrt{\epsilon}$.
\item $\hat R$ consists of two connected components,
$\hat{\partial S}_0$ and $\hat{\partial S}_1$, say.
\item $d_H(\partial S_0,\hat{\partial S}_0)\leq c_2 \sqrt{\epsilon}$ and
$d_H(\partial S_1,\hat{\partial S}_1)\leq c_2 \sqrt{\epsilon}$.
\end{enumerate}
\end{theorem}

Thus, statement {\em 2.}~of the above theorem defines the estimators
$\hat{\partial S}_0$ and $\hat{\partial S}_1$.

\begin{proof}
%%\begin{theorem} \label{A}
Parts {\em 1} and {\em 2} follow easily.
Let us turn to {\em 3}.
Let 
$y = f(u)+\sigma N(u) \in \partial S_0$
where
$0 \leq u \leq 1$.
First suppose that
$A\sqrt{\epsilon} < u < 1-A\sqrt{\epsilon}$
where $A = a(2+c+C,\epsilon)$.
Then
$y\notin B(f(0),\sigma + (2+c+C)\epsilon)$.
There exists $\hat y\in\hat{\partial S}$ such that
$||y - \hat y||\leq \epsilon$.
So
\begin{eqnarray*}
\sigma + (2+c +C)\epsilon & < & ||y - f(0)|| \leq
||y - \hat y|| + ||\hat y - \hat x_0|| + ||\hat x_0 - f(0)||\\
& \leq &
\epsilon + ||\hat y - \hat x_0|| + C \epsilon
\end{eqnarray*}
and so
\begin{eqnarray*}
||\hat y - \hat x_0||  & > &  \sigma + (2+c+C)\epsilon - \epsilon - C \epsilon\\
&=& \sigma + (1+c)\epsilon >
\hat\sigma + c\epsilon .
\end{eqnarray*}
Thus, $\hat y \notin \hat E_0$.
A similar argument shows that
$\hat y \notin \hat E_1$ and
$\hat y \notin \hat{\partial S}_1$.
Hence $\hat y \in \hat{\partial S}_0$.
Now suppose that
$0 \leq u \leq A \sqrt{\epsilon}$.
From Lemma \ref{lemma::distances}
$$
||f(u)-f(A\sqrt{\epsilon})|| \leq c_1 \sqrt{\epsilon}
$$
for some $c_1$.
From the first part of the proof,
there is a $\hat y \notin\hat E_0$ such that
$||f(A\sqrt{\epsilon})- \hat y|| \leq \epsilon$.
But
$||f(u) - \hat y|| \leq ||f(u) - f(A\sqrt{\epsilon})|| + ||f(A\sqrt{\epsilon})-\hat y||=
\epsilon + c_1 \sqrt{\epsilon} \leq c_2 \sqrt{\epsilon}$, say.
Hence,
$\partial S_0 \subset \hat{\partial S}_0 \oplus c_2\sqrt{\epsilon}$.

Now let $\hat y$ be in $\hat{\partial S}_0$.
Hence, $||\hat y - \hat x_0|| > \hat\sigma + c \epsilon$.
Let $y\in\partial S$ be such that
$||\hat y - y|| \leq \epsilon$.
Now
\begin{eqnarray*}
\hat\sigma + c\epsilon & < & ||\hat y - \hat x_0||\\
& \leq & ||\hat y - y|| + ||y-f(0)|| + ||f(0)-\hat x_0||\\
& \leq & \epsilon + ||y-f(0)|| + C \epsilon
\end{eqnarray*}
and so
$$
||y-f(0)||  >  \hat\sigma + (C-1-C)\epsilon \geq 
\sigma + (c-1-C)\epsilon =\sigma + \gamma\epsilon
$$
where $\gamma = c-1-C$.
It follows that
$y\notin (V_0(\gamma\sqrt{\epsilon})\union V_1(\gamma\sqrt{\epsilon}))$.
That is, $y = f(u) + \sigma N(u)$ with
$a(\gamma,\epsilon)\sqrt{\epsilon} \leq u \leq (1-a(\gamma,\epsilon)\sqrt{\epsilon})$.
Arguing as above, using Lemma \ref{lemma::distances},
there is a $v$ such that
$a(c,\epsilon)\sqrt{\epsilon} \leq v \leq (1-a(c,\epsilon)\sqrt{\epsilon})$
and such that
$||(f(v)+\sigma N(v)) - (f(u)+\sigma N(u))|| \leq c_3 \sqrt{\epsilon}$
for some $c_3$.
Hence,
$\hat{\partial S}_0 \subset \partial{S}_0 \oplus c_3 \sqrt{\epsilon}$.
A similar argument applies to $\partial S_1$ and 
$\hat{\partial S}_1$.
The theorem follows by taking
$c_4 = \max\{c_2,c_3\}$.
\end{proof}

\smallskip
\begin{theorem}
\label{thm::open-case}
Let $\hat{\partial S}_0$ and $\hat{\partial S}_1$ be the estimators described in
statement 2.~of Theorem \ref{thm::this-is-the-est}. Let $\hat \Gamma$ be the medial 
estimator derived from $\hat{\partial S}_0$ and $\hat{\partial S}_1$.
Then the results of Theorem \ref{thm::methodII} hold.
\end{theorem}

\begin{proofof} {\bf Theorem \ref{thm::open-case}.}
Follows by combining the last four results.
\end{proofof}

\begin{lemma} (\cite{smale})
\label{lemma::distances}
If
$||f(u)-f(v)|| \leq \frac{\Delta}{2}$
then
$$
\alpha(u,v) - \frac{\alpha^2(u,v)}{2\Delta} \leq ||f(u)-f(v)|| \leq
\alpha(u,v) \leq \Delta - \Delta \sqrt{ 1 - \frac{2 ||f(u)-f(v)||}{\Delta}}.
$$
\end{lemma}

\bibliography{paper}

\end{document}